 \documentclass[11pt]{article}
 \usepackage{bm,mathtools}
  \usepackage{amsfonts}
 \usepackage[a4paper, margin=2cm]{geometry}
  \usepackage{enumitem}
 \newtheorem{theorem}{Theorem}[section]
 \newtheorem{lemma}[theorem]{Lemma}
 \newtheorem{proposition}[theorem]{Proposition}
 \newtheorem{corollary}[theorem]{Corollary}

 \newenvironment{proof}{{\it Proof}. }{\hfill\END\\[0.5ex]}

  \newcommand{\specialcell}[2][c]{%
  \begin{tabular}[#1]{@{}c@{}}#2\end{tabular}}

\newcommand{\firstReviewer}[1]{#1}
\newcommand{\secondReviewer}[1]{#1}
\newcommand{\otherCorrections}[1]{#1}

\usepackage{amsmath}
\usepackage{amssymb}

\usepackage[normalem]{ulem}

\usepackage{tikz}
\usetikzlibrary{arrows,decorations.markings,arrows.meta}
\usetikzlibrary{babel}
\usepackage{pgfplots}
\pgfplotsset{compat=1.16}

\newcommand{\R}{\mathbb{R}}

\newcommand{\be}{\begin{equation}}
\newcommand{\ee}{\end{equation}}

\newcommand{\x}{{\bm{x}}}
\newcommand{\y}{{\bm{y}}}
\newcommand{\z}{{\bm{z}}}
\newcommand{\rr}{\bm{r}}
\newcommand{\nn}{\bm{n}}

\newcommand{\Hi}{{\textsf{H}}}

\newcommand{\Vii} {{\bm{V}}}
\newcommand{\Kii} {{\bm{K}}}
\newcommand{\Wii} {{\bm{W}}}
\newcommand{\SLii}{{{\bm{SL}\,}}}
\newcommand{\DLii}{{{\bm{DL}\,}}}  
\newcommand{\Yiim}{{\bm{Y}_-}}
\newcommand{\Yiip}{{\bm{Y}_+}}
\newcommand{\Yii} {{\bm{Y}}}
\newcommand{\Tiip}{{\bm{T}_+}}
\newcommand{\Tiim}{{\bm{T}_-}}
\newcommand{\Tii} {{\bm{T}}}
\newcommand{\Sii} {{\bm{S}}}

\newcommand{\RD}  {{\bm{R}}^{{\rm D}}}
\newcommand{\RN}  {{\bm{R}}^{{\rm N}}}

\newcommand{\Hiib}{{\mathbf{H}}}
\newcommand{\Liib}{{\mathbf{L}}}

 \makeatletter\@addtoreset{equation}{section}\makeatother

\newcommand{\END}{\hfill$\Box$}

\newenvironment{remark}%
{\refstepcounter{theorem}\noindent%
\textbf{Remark \thetheorem.}}%
{\hfill $\Box$ \\[0.75ex]}

\usepackage{fancyhdr}
\title{Boundary integral equation methods for the solution of scattering and transmission 2D elastodynamic problems}
\date{April 12, 2022}
\author{V\'{\i}ctor Dom\'{\i}nguez\thanks{Dep. Ingenier\'{i}a Matem\'atica e Inform\'atica, Universidad P\'{u}blica de Navarra. Campus de Tudela 31500 - Tudela, Spain, e-mail: victor.dominguez@unavarra.es.}\and Catalin Turc\thanks{  Department of
Mathematical Sciences and Center for Applied Mathematics and Statistics, New Jersey  Institute of Technology,
Univ. Heights. 323 Dr. M. L. King Jr. Blvd, Newark, NJ 07102, USA, e-mail: catalin.c.turc@njit.edu.}}
\allowdisplaybreaks

\begin{document}
\maketitle
\begin{abstract}
  We introduce and analyze various Regularized Combined Field Integral Equations (CFIER) formulations of time-harmonic Navier equations in media with piece-wise constant material properties. These formulations can be derived systematically starting from suitable coercive approximations of Dirichlet-to-Neumann operators (DtN), and we present a periodic pseudodifferential calculus framework within which the well posedness of CIER formulations can be established. We also use the DtN approximations to derive and analyze Optimized Schwarz (OS) methods for the solution of elastodynamics transmission problems. The pseudodifferential calculus we develop in this paper relies on careful singularity splittings of the kernels of Navier boundary integral operators which is also the basis of high-order Nystr\"om quadratures for their discretizations. Based on these high-order discretizations we investigate the rate of convergence of iterative solvers applied to CFIER and OS formulations of scattering and transmission problems.  We present a variety of numerical results that illustrate that the CFIER methodology leads to important computational savings over the classical CFIE one, whenever iterative solvers are used for the solution of the ensuing discretized boundary integral equations. Finally, we show that the OS methods are competitive in the high-frequency high-contrast regime.
   
\end{abstract}

\section{Introduction}

Numerical solutions of elastodynamics scattering and transmission problems based on boundary integral equation (BIE) formulations enjoy certain attractive features over their volumetric counterparts, mainly on account of the dimensional reduction and the explicit enforcement of radiation conditions at infinity. Robust BIE formulations of elastodynamics scattering and transmission problems can be derived via the combined field strategy. Furthermore, in the case of impenetrable scattering problems, alternative robust regularized formulations can be constructed based on the incorporation of approximations of Dirichlet-to-Neumann operators~\firstReviewer{\cite{chaillat2015approximate,chaillat2020analytical,DaLo:2015}}. These formulations are intrinsically more suitable for iterative solvers in the high-frequency regime, where they lead to important computational savings over the classical combined field formulations. 

We pursue in this paper the construction of regularized BIE formulations of both penetrable and impenetrable elastodynamics scattering problems in two dimensions. While for the impenetrable case we use similar ideas to construct regularized formulations as those introduced in~\cite{chaillat2015approximate,chaillat2020analytical}, in the penetrable case we follow the blueprint introduced in~\cite{boubendir2015regularized} for Helmholtz problems. In a nutshell, we use principal symbols of periodic pseudodifferential operators to construct approximations of DtN operators which are then readily incorporated in a scheme to construct well posed BIE formulations of scattering and transmission elastodynamic problems. The DtN approximations we use are square root Fourier multipliers, and they have the distinct advantage that (a) their implementation is straightforward in the trigonometric interpolation framework and (b) the analysis of both impenetrable and penetrable regularized formulations can be performed in the same vein. In addition, we explore these DtN approximations to formulate Optimized Schwarz methods for a domain decomposition approach for transmission elastodynamics problems. Just like in the Helmholtz case~\cite{boubendir2017domain}, we establish rigorously that the Schwarz iteration operators are compact perturbations of the identity in the case of smooth interfaces of material discontinuity, which allows us to prove the well posedness of the Optimized Schwarz approach. 

While the big picture does not differ a great deal from the elastodynamics to the Helmholtz case with respect to the well posedness of regularized formulations of scattering and transmission problems, the details are considerably more involved in the elastodynamics case, notwithstanding the fact that one has to deal with vector (as opposed to scalar) quantities. A first notable difference is encountered in the double layer \otherCorrections{boundary integral operators (BIOs)} which are no longer compact operators in the elastodynamics case, not even for smooth boundaries. Consequently, the DtN pseudodifferential principal symbol calculus is more complicated in elastodynamics, as the contributions arising from the double layer BIOs cannot be any longer ignored. We use in this paper the periodic pseudodifferential calculus~\cite{hsiao2008boundary} which in conjunction with logarithmic singularity splittings for the kernels of the four elastodynamics BIOs allows us to compute their principal symbols in the sense of pseudodifferential operators.  These calculations are the basis on which we construct our approximations of the DtN operators. However, these kernel logarithmic singularity splittings are quite cumbersome~\cite{chapko2000numerical}, significantly more so than their Helmholtz counterparts. Owing to these additional complications, we chose to present the full details of these calculations in an appendix to this paper.

We employ Nystr\"om discretizations for the numerical solution of the various BIEs derived in this paper. We pursue the classical Nystr\"om method based on trigonometric interpolation, logarithmic kernel singularity splittings and the classical Kussmaul and Martensen~\cite{kusmaul,martensen} quadratures for the analytic resolution of periodized logarithmic singularities. This discretization strategy leads to high-order BIE solvers for elastodynamics scattering and transmission problems in the case of smooth boundaries. \secondReviewer{However, since the Kussmaul-Martensen discretizations rely on global interpolation, they are not compatible with fast methods such as the Fast Multipole Methods~\cite{hao2014high} which may limit their appeal.  Other high-order discretization strategies that rely on use of panels (e.g. Alpert quadratures~\cite{alpert1999hybrid}) and which are compatible with fast methods are currently being explored.} One distinct advantage of the use of trigonometric interpolation is facilitating a straightforward implementation of Fourier multipliers, and thus of the DtN approximations we construct on the  basis of the periodic pseudodifferential calculus. \secondReviewer{We follow here the common practice in the community of using square root Fourier multiplier approximations of DtN operators~\cite{AntoineX,Antoine,DaLo:2015} which tend to deliver the best results in the high-frequency regime on account of their nearly optimal treatment of modes in the Fourier space transition regime from propagating to evanescent modes.  The implementation of those non local operators is significantly more challenging when panel discretizations are used, and the prevalent strategy resorts to Pad\'e approximations of square roots~\cite{AntoineX,Antoine}, which, in turn, requires solutions of certain elliptical problems on the boundary. The latter feature Laplace-Beltrami operators in the case of surfaces in 3D, and as such, they are challenging to solve in the context of Nystr\"om discretizations and are the subject of ongoing research.}

 The paper is organized as follows: in Section~\ref{setting} we introduce the Navier equations that govern elastodynamics waves in two dimensions; in Section~\ref{auxr}  we present the boundary layer potentials and integral operators associated with the Navier equations and we construct  their principal parts in the sense of pseudodifferential operators.  An essential part of this construction are detailed calculations of the kernels of the BIOs involved, with a special emphasis being given to their splitting into regular and singular parts. An exhaustive account of these factorizations is presented in the Appendix. \firstReviewer{In Sections~\ref{BIEf}-\ref{BIEt} we construct and analyze regularized BIE formulations for both penetrable and impenetrable scattering problems; 
 Nystr\"om discretizations of the elastodynamics BIOs as well as a variety of numerical results that illustrate the high-order convergence of these methods and the iterative behavior of the various formulations considered in this paper are presented in Section \ref{Nd}.  
 }

 

\section{ Navier equations and boundary integral operators} 
\label{setting} 

In this section we present the Navier equations and their associated four \otherCorrections{boundary integral operators (BIOs)}. We will pay special attentions to the  kernels of such operators for two reasons: (a) to derive the principal part (in the pseudodifferential sense) of the underlying operators; and (b) to describe precisely a factorization of these kernels into regular and singular parts which is subsequently exploited by spectral Kress Nystr\"om methods. Although the latter methods were introduced in the 80s for the Helmholtz equation and then extended in the 90s for the elastodynamic equations, the complexity of the functions involved makes  the splitting technique quite difficult in the elastodynamics case, and many details are unfortunately missing from the literature. Some references can be found in  \cite{chapko2000numerical} and \cite[Section 4]{DoSaSa:2015}. In the former paper, the emphasis is placed on the hypersingular operator, and we actually follow some of the notations introduced in that work. On the other hand, we find in the second paper a brief description of all four of the BIO kernels but the regular and singular factorizations of the kernels is omitted. Our main objective is to write in \firstReviewer{an} appropriate framework a certain integration by parts formula for the hypersingular Navier BIO which is an essential piece in the numerical method object of this paper.

\subsection{Navier equation}

Let us introduce first some notation for geometric quantities. We will  use the following notation for points ${\bm x}=(x_1,x_2)\in\mathbb{R}^2$. Any vector is understood to be a column vector, in such a way that
\[
 \x^\top \y=\x\cdot\y=x_1y_1+x_2y_2,\quad  \x\: \y^\top=\begin{bmatrix} 
 x_1 y_1&x_1 y_2\\
 x_1 y_2&x_2 y_2
 \end{bmatrix}.
\]
We will also denote 
\[
 \rr:=\x-\y,\quad r:=|\rr|=|\x-\y|.
\] 
For any sufficiently regular   curve  $\Gamma$ (Lipschitz is enough),  the unit outward pointing vector  is well defined and will be denoted henceforth by $\nn=(n_1,n_2)$. Consider also  $\bm{t}=(-n_2,n_1)$ the unit tangent vector obtained by rotating $\nn$ clockwise by ninety degrees. The curve $\Gamma$ will be assumed to be simply connected unless  we state otherwise.  

Let ${\bf u}=(u_1,u_2):\mathbb{R}^2\to   \mathbb{R}^2$ be a vector function. For a linear isotropic and homogeneous elastic medium with Lam\'e constants $\lambda$ and $\mu$ such that $\lambda>-\mu$,  the strain and the stress tensor are defined as 
\begin{eqnarray*}
 \bm{\epsilon}({\bf u})&:=&\frac{1}2 (\nabla {\bf u}+(\nabla {\bf u})^\top)
 =\begin{bmatrix}
  \partial_{x_1} u_1& \tfrac{1}2\left(\partial_{x_1} u_2+\partial_{x_2} u_1\right)     \\
   \tfrac{1}2\left(\partial_{x_1} u_2+\partial_{x_2} u_1\right) &                                                    
   \partial_{x_2} u_2
     \end{bmatrix}\\
 \bm{\sigma}_{\lambda,\mu}({\bf u})&:=&2\mu \bm{\epsilon}({\bf u})+\lambda (\nabla \cdot{\bf u}) I_2
\end{eqnarray*} 
where $I_2$ is the identity matrix of order 2. The time-harmonic elastic wave (Navier) equation is given by
\[
 {\nabla}\cdot\bm{\sigma}_{\lambda,\mu}({\bf u})+\omega^2{\bf u}=\mu\Delta {\bf u}+(\lambda+\mu)\nabla(\nabla\cdot {\bf u})+\omega^2 {\bf u}=0 
\]
where the frequency $\omega\in\mathbb{R}^+$ and the {\rm divergence} operator  ``$\otherCorrections{\nabla \cdot}$'' is applied row-wise \secondReviewer{(here we assume the density $\rho$ is equal to $1$)}. We often alleviate the notation and will simply write  $ \bm{\sigma}$ for the stress tensor if the context make clear which values for the Lam\'e constants are being considered.

The trace operator on $\Gamma$ is denoted by $\gamma_\Gamma {\bf u}$ whereas  the normal stress tensor, or traction operator, is 
given by 
\begin{equation}\label{eq:def:T}
 T_\Gamma{\bf u}\otherCorrections{:}=\bm{\sigma}({\bf u})\nn=\lambda (\nabla\cdot  {\bf u}) \nn+2\mu (\nn\cdot {\nabla}){\bf u}-\mu(\nabla\times {\bf u})
 \begin{bmatrix}
&-1\\
1&\end{bmatrix}{\bm n}. 
\end{equation}

Let then $\Omega_-$ and $\Omega_+=\mathbb{R}^2\setminus\overline{\Omega}_+$ be  the interior and exterior of  $\Gamma$. We will study in this paper BIE formulations for the following boundary value problems:
\begin{enumerate}[label=(\alph*)]
\item Solution of the exterior Dirichlet/Neumann, unpenetrable domain,  time-harmonic Navier equation:
\begin{equation}\label{eq:NavD} 
 \left|
 \begin{array}{rcl}
   \multicolumn{3}{l}{{\bf u} \in \Hiib_{\rm loc}^1(\Omega_+):=H_{\rm loc}^1(\Omega_+)\times H_{\rm loc}^1(\Omega_+), }  \\
   {\nabla}\cdot\bm{\sigma}({\bf u} )+\omega^2{\bf u}  &=&0 \firstReviewer{\quad \text{in }\Omega_+,}\\ 
 \multicolumn{3}{l}{\text{DC/NC on $\Gamma$},}\\
 \multicolumn{3}{l}{\otherCorrections{\rm +RC}}.
 \end{array}
 \right.
\end{equation}
 Here $H^1_{\rm loc}(\Omega_+)$ is the space of functions which are locally in the Sobolev space $H^1(\Omega_+)$,  DC stands for Dirichlet condition, i.e., 
 \[
   \gamma_\Gamma {\bf u} ={\bf f},  
 \]
  and NC, for Neumann condition, 
  \[
  T_\Gamma{\bf u}=\bm{\lambda}.
 \]
 Finally, RC stands for the radiation condition, or Kupradze condition, at infinity cf. \cite{KupGeBa:1979} or \cite[Ch. 2]{AmKaLe:2009}.  Hence, if  
 ${\bf u}_{p}$ and ${\bf u}_s$ are  the longitudinal and transversal wave defined as
\begin{equation}\label{eq:up:us}
{\bf u}_p:=-\frac{1}{k_p^2}\nabla\ \nabla\cdot{\bf u}\qquad {\bf u}_s:={\bf u}-{\bf u}_p
\end{equation}
 with 
\begin{eqnarray}
 k_p^2&:=&\frac{\omega^2}{\lambda+2\mu},\quad k_s^2:=\frac{\omega^2}{\mu},\label{eq:k_p:k_s}
 \end{eqnarray}
 the associated wave numbers, 
we require that  uniformly in ${\bm x}$, with $\widehat{\bm{x}}:=\bm{x}/|\bm{\x}|$, 
 \begin{subequations}\label{eqs:RadCondition}
\begin{equation}\label{eqs:RadCondition:01}
\frac{\partial {\bf u}_{p}}{\partial \widehat{\bm{x}}}(\bm{x}) - i k_{p} {\bf u}_p(\bm{x}),\ 
\frac{\partial {\bf u}_{s}}{\partial \widehat{\bm{x}}}(\bm{x}) - i k_{s} {\bf u}_s(\bm{x})  
= o(|{\bm{x}}|^{-1/2}). 
\end{equation}
An equivalent formulation is given by
 \begin{equation}\label{eqs:RadCondition:02}
\bm{\sigma}({\bf u}_p)(\bm{x})\cdot\widehat{\bm{x}}-ik_p(\lambda+2\mu)  {\bf u}_p(\bm{x}),
\quad 
\bm{\sigma({\bf u}_s)}(\bm{x})\cdot\widehat{\bm{x}}-ik_s \mu  {\bf u}_s(\bm{x})
= o(|{\bm{x}}|^{-1/2})\firstReviewer{,}
\end{equation}
which in turns implies 
\begin{equation}\label{eqs:RadCondition:03}
  {\bf u}_p(\bm{x})\cdot {\bf u}_s(\bm{x}),\quad (\bm{\sigma({\bf u}_s)}(\bm{x}) \widehat{\bm{x}})\cdot(
  \bm{\sigma({\bf u}_p)}(\bm{x})\widehat{\bm{x}}) = o(|{\bm{x}}|^{-1})\firstReviewer{.}
\end{equation}
\end{subequations}
We refer to \cite[Ch. III, \S 2]{KupGeBa:1979} for a proof of these results in $\mathbb{R}^3$ which can be easily adapted to our two-dimensional case. 
 

 Both problems are uniquely solvable (see, for 2D and 3D problems, \cite{BrPa:2008}, \cite[Ch. 3]{KupGeBa:1979} or \cite[Ch. 2]{AmKaLe:2009}). 

\item Transmission problems:  For different material properties $(\lambda_+,\mu_+)$ in $\Omega_+$ and $(\lambda_-,\mu_-)$ in $\Omega_-$ we seek a solution of 
\begin{equation}\label{eq:Transmission:0}
 \left|
 \begin{array}{rcl}
 \multicolumn{3}{l}{  {\bf u}_- \in \bm{H}^1(\Omega_-),\quad {\bf u}_+ \in \bm{H}_{\rm loc}^1(\Omega_+),}\\
   {\nabla}\cdot\bm{\sigma}_-({\bf u}_-)+\omega_-^2{\bf u}_-  &=&0 \quad \mbox{in }\Omega_-,\\ 
   {\nabla}\cdot\bm{\sigma}_+({\bf u}_+ )+\omega_+^2{\bf u}_-  &=&0 \quad \mbox{in }\Omega_+,\\ 
 \gamma_\Gamma {\bf u}_+-\gamma_\Gamma {\bf u}_- &=& -\gamma_\Gamma {\bf u}^{\rm inc},\\ 
 \otherCorrections{T_{+,\Gamma}}{\bf u}_+- \otherCorrections{T_{-,\Gamma}} {\bf u}_- &=& - \otherCorrections{T_{+,\Gamma}}{\bf u}^{\rm inc},\\
 \multicolumn{3}{l}{\text{+ RC}.}
 \end{array}
 \right.
\end{equation} We have used above the $\pm$ signs in the 
normal stress tensor ($\Tii_{\pm,\Gamma}$) to clarify the domain from which it is applied,  be it the exterior/interior domain with respect to $\Gamma$ (that is, $\Omega_+$ or $\Omega_-$). We will keep using this notation whenever the context is not clear enough to resolve this aspect.  \firstReviewer{We also implicitly assumed that the densities $\rho_\pm=1$ in both media.}

On the other hand, the function 
${\bf u}_{\rm inc}$ appearing in the right-hand-side above is  a solution of the exterior Navier problem in a neighborhood of $\overline{\Omega}_+$ (typically a shear or a pressure {plane} wave). \firstReviewer{This problem is known to be uniquely solvable. We refer the reader to \cite[Corolary 2.7]{CoSt:1990},\cite[Ch. 2]{AmKaLe:2009} or \cite[Chapter III]{KupGeBa:1979}}.


\end{enumerate}
 
We finish this section by introducing the exterior/interior Dirichlet-to-Neumann operator,  which will play an essential role in this work: for $\bm{g}:\Gamma\to\mathbb{C}$ sufficiently smooth (we will give detailed conditions later) we  define
\begin{equation}\label{eq:defYs}
 \Yii_{+,\Gamma}{\bf g} := \Tii_{+,\Gamma} {\bf u}_+,\quad 
 \Yii_{-,\Gamma}{\bf g} := \Tii_{-,\Gamma} {\bf u}_-.  
\end{equation}
where $ {{\bf u}_+} $ satisfies the RC and 
\[
 \bm{\sigma}_\pm({\bf u}_\pm)+\omega_\pm^2{\bf u}_\pm =0,\quad \gamma_{{\Gamma}}{\bf u}_\pm={\bf g}. 
\]
Let us emphasize that the exterior DtN operator $\Yii_{+,\Gamma}$ is well defined for all frequencies $\omega$, unlike the interior DtN operator $\Yii_{-,\Gamma}$, which fails to be properly defined whenever $\omega^2$ is an eigenvalue for the Dirichlet problem of the Navier differential operator in the bounded domain $\Omega_{-}$. 
Notice that the second transmission condition can be rewritten now as
\[
 \Yii_{+,\Gamma}\gamma_{ \Gamma}  {\bf u}_+-\Yii_{-,\Gamma}\gamma_{ \Gamma} {\bf u}_-= -\otherCorrections{T_{+,\Gamma}}{\bf u}^{\rm inc}
\]
assuming that  the inner DtN operator is  properly defined.

\section{Boundary integral operators for elastodynamics}

\label{auxr}

\subsection{Fundamental solution. Boundary layer potentials and  integral operators}
\otherCorrections{We briefly summarize the more relevant results of the Calder\'on calculus associated with layer potentials and boundary integral operators for two-dimensional elastodynamics. We refer the reader to \cite{CoSt:1990}, \cite[Ch. 2]{KupGeBa:1979}, \cite[\S 2.4]{AmKaLe:2009} and \cite{DoSaSa:2015} for exhaustive studies on this topic. While the first three references deal only with the 3D case, we point out that all of the results established in those works can be easily adapted to the 2D case.}

The fundamental solution of the time-harmonic elastic wave equation is given by the $2\times 2$ matrix function 
\begin{eqnarray*}
 \Phi(\x,\y)&:=&\frac{1}{\mu}\phi_0(k_s r)I_2+\frac{1}{\omega^2}\nabla_{\x}\nabla_{\x}^\top(\phi_0(k_s r)-\phi_0(k_p r)),
 \quad  r =|\x-\y| 
\end{eqnarray*}
where 
\begin{eqnarray}
 \phi_j(z)&:=&\frac{i}4 H_j^{(1)}(z). \label{eq:Hj}
 \end{eqnarray}
In the expression above, $H_j^{(1)}$ denotes the Hankel function of order $j$ and first kind, so that for the particular case $j=0$, $\phi_0(k\firstReviewer{r})$ is just the fundamental solution of the Helmholtz equation  $\Delta+k^2$. 
  
 Using the fundamental solution $\Phi(\x,\y)$ of the Navier equation we can define the single and double layer potentials. Specifically, for a given density (vector) function $\bm{\lambda}_\Gamma$ defined on $\Gamma$, the single layer potential is defined as
\begin{equation}\label{eq:single}
  (\SLii_\Gamma{\bm{\lambda}_\Gamma})(\z):=\int_\Gamma \Phi(\z,\y){\bm \lambda}_\Gamma(\y)\,{\rm d}\y,\quad \z\in\mathbb{R}^2\setminus\Gamma.
\end{equation}
 Analogously, for a given density (vector) function $\bm{g}_\Gamma$  defined on $\Gamma$, the double layer potential is defined as
\begin{equation}\label{eq:double}
  (\DLii_\Gamma \bm{g}_\Gamma)(\z):=\int_\Gamma \left[ \Tii_{\y,\Gamma} \Phi(\z,\y)\right]^\top\bm{g}_\Gamma(\y)\,{\rm d}\y,\quad \z\in\mathbb{R}^2\setminus\Gamma
\end{equation}
where $\Tii_{\y,_\Gamma}\Phi(\z,\y)$ is the normal stress tensor applied column-wise to $\Phi(\z,\y)$ with respect to the  $\y$ variable.

\otherCorrections{We then have the following representation formula which is referred to as the Somigliana identity 
\begin{equation}\label{eq:Somigliana:1}
 {\bf u}_{\pm} = \pm \DLii_\Gamma \gamma_{\Gamma} {\bf u}_{\pm}  \mp \SLii_\Gamma T_{\Gamma} {\bf u}_\pm, 
\end{equation}
for any ${\bf u}_{\pm}$   solution of the homogeneous Navier equation in $\Omega_\pm$ satisfying additionally the radiation condition in the exterior domain. }

The single layer potential is continuous in $\mathbb{R}^2$, and thus, assuming that $\z=\x+\varepsilon {\bf n}(\x),\ \x\in\Gamma$ and taking the limit as $\varepsilon\to 0$ in equation~\eqref{eq:single} we can define the single layer operator
\begin{equation}\label{eq:singleV}
  (\Vii_\Gamma{\bm \lambda}_\Gamma)(\x):=\int_\Gamma \Phi(\x,\y){\bm \lambda}_\Gamma(\y)\,{\rm d}\y,\quad \x\in\Gamma.
  \end{equation}
 The application of the traction operator $\Tii_{\Gamma}$ to the single layer potential $\SLii_\Gamma\bm{\lambda}_\Gamma $ gives rise to jump discontinuities
\[
\lim_{\varepsilon\to 0}\Tii_{\Gamma}(\SLii_\Gamma\bm{\lambda}_\Gamma)(\x\pm \varepsilon {\bf n}(\x))=\mp {\bm \lambda}_\Gamma(\x)+(\Kii^\top_\Gamma \bm{\lambda}_\Gamma)({\x}),\quad \x\in\Gamma
\]
where the adjoint double layer operator is defined as 
\[
(\Kii_\Gamma^\top {\bm \lambda}_\Gamma)({\x}):=\mathrm{p.v.}\,\int_\Gamma 
 \Tii_{\x,\Gamma} \Phi(\x,\y) 
{\bm \lambda}_\Gamma(\y)\,{\rm d}\y\firstReviewer{.}
\]
 {The integral above is singular and it has to be understood in the sense of Cauchy principal value (which is what $\mathrm{p.v.}\,$ stands for).} 
   The double layer potential $\DLii_\Gamma $ undergoes a jump discontinuity across $\Gamma$ so that
\[
\lim_{\varepsilon\to 0}(\DLii_\Gamma{\bm{g}_\Gamma})(\x+\varepsilon {\bf n}(\x))-\lim_{\varepsilon\to 0}(\DLii_\Gamma{\bm{g}_\Gamma})(\x-\varepsilon {\bf n}(\x))={\bm{g}_\Gamma}(\x),\quad \x\in\Gamma\firstReviewer{,}
\]
and
\[
\lim_{\varepsilon\to 0}(\DLii_\Gamma{\bm{g}_\Gamma})(\x+\varepsilon {\bf n}(\x))+\lim_{\varepsilon\to 0}(\DLii_\Gamma{\bm{g}_\Gamma})(\x-\varepsilon {\bf n}(\x))=2(\Kii_\Gamma \bm{g}_\Gamma)({\x}),\quad \x\in\Gamma
\]
where the \otherCorrections{double layer operator}  is defined explicitly as
\[
(\Kii_\Gamma\bm{g}_\Gamma)({\x}):= {\mathrm{p.v.}}\int_\Gamma  \left[  \Tii_{\y,\Gamma} \Phi(\z,\y)\right]^\top \bm{g}_\Gamma(\y)\,{\rm d}\y.
\]
Finally, applying the traction operator $T_\Gamma $ to the double layer potential $\DLii_\Gamma {\bm{g}_\Gamma}$ we obtain \otherCorrections{the hypersingular operator}
\[
\lim_{\varepsilon\to 0}(T_\Gamma \DLii_\Gamma{\bm{g}_\Gamma})(\x\pm \varepsilon {\bf n}(\x))=(\Wii_\Gamma \bm{g}_\Gamma)({\x}),\quad \x\in\Gamma
\]
where the  boundary integral operator \firstReviewer{(BIO)}  $\Wii_\Gamma $ is defined as
\[
 (\Wii_\Gamma\bm{g}_\Gamma)({\x}):=\operatorname{f.p.}\int_\Gamma   \Tii_{\x,\Gamma} \big[\Tii_{\y,\Gamma} \Phi(\rr)\big]  ^\top  {\bm g}_\Gamma({\y})\,{\rm d}\y.
\]
The kernel $W(\x,\y)$ defined above is strongly singular (that is it behaves like $\mathcal{O}(|\x-\y|^{-2})$ as $\y\to\x$), and as such the integral in its definition must be interpreted in a Hadamard finite part sense (which it is what ``{f.p.}'' stands for in the expression above). 
  
%
 
\begin{remark}\label{remark:C:01} It is a well established result that the matrix \otherCorrections{BIO}
\[
 {\cal C}_\Gamma:\begin{bmatrix}
           \Kii_\Gamma & -\Vii_\Gamma\\
           \Wii_\Gamma& - \Kii_\Gamma^\top
          \end{bmatrix}: \Hiib^{1/2}(\Gamma)\times \Hiib^{-1/2}(\Gamma) \to \Hiib^{1/2}(\Gamma)\times \Hiib^{-1/2}(\Gamma)
\]
is continuous. Here $\Hiib^{\pm 1/2}(\Gamma):= H^{\pm 1/2}(\Gamma)\times H^{\pm 1/2}(\Gamma)$ are the standard Sobolev spaces in which Dirichlet and Neumann traces on the boundary are considered. Furthermore~\cite{chaillat2020analytical}, 
\[
 {\cal C}^2_\Gamma=\frac{1}{4}{\cal I}
 \]
where ${\cal I}= \begin{bmatrix}
                  \bm{I}&\\
                        &\bm{I}
                 \end{bmatrix}
$ is the identity operator in $\Hiib^{s}(\Gamma)\times \Hiib^{t}(\Gamma)$. This makes the operators
\[
 \frac{1}{2} {\cal I}\pm {\cal C}_\Gamma
\]
projections in the space in $\Hiib^{1/2}(\Gamma)\times \Hiib^{-1/2}(\Gamma)$, known  in the literature  as the exterior/interior Calder\'{o}n operators. Since 
\[
{\bf u}_{\pm} = \pm \DLii_\Gamma\bm{g}_\Gamma \mp \SLii_\Gamma{\bm \lambda}_\Gamma \quad {\Leftrightarrow} \quad 
 \left(\frac{1}{2} {\cal I}\pm {\cal C}_\Gamma\right)
 \begin{bmatrix}
 {\bm g}_\Gamma\\                                              
 {\bm \lambda}_\Gamma                                               
 \end{bmatrix}                                              
=
 \begin{bmatrix}
 \gamma_{\Gamma} {\bf u}_{\pm}\\        
 {T}_{\Gamma}{\bf u}_{\pm}
 \end{bmatrix}    
\]
we have
\[
 \left(\frac{1}{2} {\cal I}\pm {\cal C}_\Gamma\right)
  \begin{bmatrix}
 \gamma_{\Gamma} {\bf u}_{\pm}\\        
 {T}_{\Gamma}{\bf u}_{\pm}    
 \end{bmatrix}    =   \begin{bmatrix}
 \gamma_{\Gamma} {\bf u}_{\pm}\\        
 {T}_{\Gamma}{\bf u}_{\pm}    
 \end{bmatrix},\  \text{or   equivalently } \ 
(\gamma_{\Gamma} {\bf u}_\pm, {T}_{\Gamma}{\bf u}_\pm)\in \mathop{\rm Ker}    
  \left( \tfrac{1}{2} {\cal I}\mp  {\cal C}_\Gamma\right).
\]

Let us end this remark by stating the convention we have followed so far: We will use boldface roman letters (as ${\bf f}$ or $\bm{g}$) for functions that naturally belong to $\Hiib^{1/2}(\Gamma)$ whereas Greek letters will be used to represent functions (or distributions) in  $\Hiib^{-1/2}(\Gamma)$ ---without excluding the possibility that for Navier problems with regular enough boundary data the functions involved also turn out to be more regular even though $\Gamma$ is only Lipschitz.  
\end{remark}

\subsection{Parameterized operators and functional spaces}

In this section we introduce parameterized versions of the kernels of the Navier layer potentials and their associated \otherCorrections{BIOs}. We will assume from now on that there exists a $2\pi-$periodic smooth regular  parameterization of the closed curve $\Gamma$, ${\bf x}:\R\to\Gamma$. We then denote
\[
  {\bm \nu }(t) = (x'_2(t),-x'_1(t))= (\nn\circ{\bf x})(t)|{\bf x}'(t)|, 
\]
the parameterized  normal 
tangent field  to $\Gamma$.

We will adopt the following convention: 
\[
 \Hiib^{1/2}(\Gamma)\times \Hiib^{-1/2}(\Gamma)  \ni ({\bm g}_\Gamma,\bm{\lambda}_\Gamma)   \quad \rightsquigarrow\quad 
  ( {\bm g},\bm{\lambda} ) := ({\bm g}_\Gamma\circ{\bf x},(\bm{\lambda}_\Gamma\circ{\bf x}){|\bf x'}| )\in 
   \Hiib^{1/2} \times \Hiib^{-1/2} .
\]
Here, 
\[
\Hiib^{s}:= H^{s}\times H^{s},\quad 
 H^{s}:=\left\{ \varphi\in{\cal D}'(\R) \ :\ \varphi(\cdot+2\pi)=\varphi,\quad \|\varphi\|
 _{s}<\infty\right\}
\]
with 
\[
\|\varphi\|_{s}^2 :=|\widehat{\varphi}(0)|^2+ \sum_{n\ne 0} |n|^{2s}|\widehat{\varphi}(n)|^2,\qquad  \widehat{\varphi}(n):=\int_{0}^{2\pi}\varphi(t) \exp(-i n t)\,{\rm d}t. 
\]
We extend actually such identification for any $s\ge 0$:
\[
\Hiib^{s}(\Gamma)\ni \bm{g}_\Gamma \  \rightsquigarrow\ 
 {\bm g} :=  {\bm g}_\Gamma\circ{\bf x} \in \Hiib^s,\qquad 
\Hiib^{-s}(\Gamma)\ni \bm{\varphi}_\Gamma \ \rightsquigarrow\   
 {\bm \varphi} := ({\bm \varphi}_\Gamma\circ{\bf x}) \:|{\bf x'}| \in \Hiib^{-s}.\quad 
\]
When $s=0$, i.e. $\Liib(\Gamma):=L^2(\Gamma)\times L^2(\Gamma)$ we can choose between either of the two identifications above depending on what is most appropriate for our purposes.
 
Notice that  with this convention
\[
\int_\Gamma  \bm{g}_\Gamma\cdot \bm{\lambda}_\Gamma  = 
\int_0^{2\pi} \bm{g} \cdot \bm{\lambda} =: \langle \bm{g},\bm{\varphi}\rangle ,\quad \bm{g} =\bm{g}_\Gamma\circ{\bf x},\quad \bm{\lambda}=(\bm{\lambda}_\Gamma\circ{\bf x})\: |{\bf x'}|. 
\]

For any ${\bf u}:\Omega_\pm\to\mathbb{C}^2$ sufficiently smooth we follow this convention and thus define the $2\pi-$periodic vector fields: 
\begin{equation}\label{eq:def:T:2}
 \gamma{\bf u}: = {\bf u}\circ{\bf x},\quad 
 T{\bf u} := (\Tii_{\Gamma }{\bf u}\circ{\bf x})\,|{\bf x}| = (\bm{\sigma}({\bf u})\circ{\bf x}) \bm{\nu}
\end{equation}
which are the (parameterized) trace and normal stress tensor. With these notations in hand, we then introduce the parameterized version of the Navier layer potentials 
 \begin{eqnarray}\label{eq:singleP}
  (\SLii{\bm{\lambda}})({\bm{z}})&:=&\int_0^{2\pi}\Phi({\bf z},{\bf x}(t)) {\bm \lambda}(t)\,{\rm d}t,\quad \z\in\mathbb{R}^2\setminus\Gamma,\\
\label{eq:doubleP} 
  (\DLii \bm{g})(\z)&:=& 
  \int_0^{2\pi}   \left[\Tii_t\Phi(\z,{\bf x}(t))\right]^\top\,\bm{g}(t)\,{\rm d}t ,\quad\z\in\mathbb{R}^2\setminus\Gamma,
\end{eqnarray} 
\firstReviewer{
where
\begin{equation}\label{eq:Tii}
 \qquad
  \Tii_t\Phi(\cdot,{\bf x}(t)):=
   \left.\Tii_{\y,\Gamma} \Phi(\cdot,\y)\right|_{{\bf y}={\bf x}(t)}|{\bf x}'(t)|,
\end{equation}}
and their  associated four \otherCorrections{BIOs}
\begin{subequations}\label{eq:BLO}
\begin{eqnarray}
   (\Vii{\bm \lambda})({{t}})&:=&\int_0^{2\pi}  V({\tau},t) {\bm \lambda}(t)\,{\rm d}t= \int_0^{2\pi}  \Phi({\bf x}({\tau}),{\bf x}(t)){\bm \lambda}(t)\,{\rm d}t,
   \label{eq:BLO:SL}\\
   (\Kii\bm{g})({\tau})      &:=&\mathrm{p.v.}\,\int_0^{2\pi}   K({\tau},t)\,\bm{g}(t)\,{\rm d}t
 =     {\rm p.v}\int_0^{2\pi}    \left[ \Tii_{t} \Phi({\bf x}({\tau}),{\bf x}(t))\right]^\top \,\bm{g}(t)\,{\rm d}t,\label{eq:BLO:DL}
 \\
   (\Kii^\top {\bm \lambda})({\tau}) &:=&\mathrm{p.v.}\,\int_0^{2\pi}   K^\top (t,{\tau})\,\bm{\lambda}(t)\,{\rm d}t,\label{eq:BLO:AdDL}
  =  {\rm p.v}\int_0^{2\pi}     \Tii_{{\tau}} \Phi({\bf x}({\tau}),{\bf x}(t))  \,\bm{\lambda}(t)\,{\rm d}t,\\
   (\Wii{\bm g})({\tau}) &:=&  \operatorname{f.p.}\int_0^{2\pi}   W({\tau},t)\,\bm{g}(t)\,{\rm d}t,\label{eq:BLO:Hyp}
    =  
   \operatorname{f.p.}\int_0^{2\pi}   \Tii_{{\tau}} \big[\Tii_{t} \Phi({\bf x}({\tau})\otherCorrections{,}{\bf x}(t))\big]^\top  \,\bm{g}(t)\,{\rm d}t.
\end{eqnarray}
\end{subequations}
\firstReviewer{(Here, and similarly to \eqref{eq:Tii}, $ \Tii_\tau \Psi({\bf x}(\tau),\cdot):=
   \left.\Tii_{\x,\Gamma} \Psi( \x,\cdot)\right|_{{\x}={\x}(\tau)}|{\bf x}'(\tau)|,$  for any $2\times 2$ matrix function 
   $\Psi:\Gamma\times \Gamma\to\mathbb{C}^{2\times 2}$).}
Observe that with these definitions, $\Kii$ and $\Kii^\top$ are effectively transpose of each other: 
\[
 \langle \Kii\bm{g},\bm{\lambda}\rangle=
 \langle \bm{g},\Kii^\top \bm{\lambda}\rangle,\quad \forall {\bm g}\in \Hiib^s,\ \forall \bm{\lambda}\in \Hiib^{-s}. 
\]

As in Remark \ref{remark:C:01} the (parameterized) version 
\[
 {\cal C}:\begin{bmatrix}
           \Kii & -\Vii\\
           \Wii & - \Kii^\top
          \end{bmatrix}: \ \Hiib^{1/2} \times \Hiib^{-1/2}  \to \Hiib^{ 1/2} \times \Hiib^{-1/2} 
\]
can be shown to satisfy the  related properties: 
\[
 {\cal C}^2 = \frac{1}4{\cal I}  \quad (\pm \frac{1}2{\cal I}+{\cal C})^2 = \pm \frac{1}2{\cal I}+{\cal C}
\]
and that a similar relation can be established between functional densities defined on $\Gamma$ and boundary traces of their associated combined field layer potentials. Indeed, if we define
\[
{\bf u}_{\pm} = \pm \DLii \bm{g}  \mp \SLii {\bm \lambda}  \quad \text{in $\Omega_\pm$},
\]
we have 
\[
 \left(\frac{1}{2} {\cal I}\pm {\cal C}\right)
  \begin{bmatrix}
 \gamma_{\pm} {\bf u}\\        
 {T}_{\pm}{\bf u}    
 \end{bmatrix}    =   \begin{bmatrix}
 \gamma_{\pm} {\bf u}\\        
 {T}_{\pm}{\bf u}    
 \end{bmatrix},\  \text{or   equivalently } \ 
(\gamma_{\pm} {\bf u}, {T}_{\pm}{\bf u})\in \mathop{\rm Ker}    
  \left( \tfrac{1}{2} {\cal I}\mp  {\cal C}\right).
\]
In particular, the  Somigliana identity  \otherCorrections{cf. \eqref{eq:Somigliana:1}}  looks similar in this framework: 
\begin{equation}\label{eq:Somigliana:2}
{\bf u}_{\pm} = \pm \DLii \gamma_{\pm } {\bf u}  \mp \SLii T_\pm {\bf u}. 
\end{equation}

\firstReviewer{\subsection{Principal symbol pseudodifferential calculus}}
Our next result establishes the fact that the difference between the kernels of the four parametrized BIOs associated with the time harmonic Navier equations (i.e. elastodynamics) and their static counterparts (i.e. elasticity) exhibit integrable singularities  only, and the nature of the latter can be made explicit:

\begin{proposition}\label{prop:reg:sing:decomposition}
 There exist smooth  $2\pi-$periodic $2\times 2$ matrix functions
 \[
A_{\rm log},\quad  A_{\rm reg},\quad  B _{\rm reg},\quad  C_{\rm log},\quad D_{\rm reg}  
 \]
 so that
 \begin{eqnarray*}
  V(\tau,t)&=& V_0(\tau,t)+ A_{\rm log}(\tau,t) \sin^2\frac{\tau-t}2 \log r + A_{\rm reg}(\tau,t)\firstReviewer{,} \\
  K(\tau,t)&=& K_0(\tau,t)+ B_{\rm log}(\tau,t) \sin (\tau-t) \log r   + B_{\rm reg}(\tau,t)\firstReviewer{,} \\
  W(\tau,t)&=& W_0(\tau,t)+ C_{\rm log}(\tau,t)   \log r   + D_{\rm reg}(\tau,t)\firstReviewer{,}
 \end{eqnarray*}
where
 \begin{eqnarray*}
V_0(\tau,t)&:=&- \frac{\lambda+3\mu}{4\pi(\lambda+2\mu)\mu}\log r \: I_2+\frac{\lambda+\mu}{4\pi(\lambda+2\mu)\mu} G(\rr)\firstReviewer{,}\\
K_0(\tau,t)&:=&\frac{\mu}{\lambda+2\mu} \left(\frac{\partial}{\partial t}
        \frac{1}{2\pi }\log r\right) \: \begin{bmatrix}
                                         &-1\\
                                         1&
                                        \end{bmatrix}
+\frac{1}{2\pi r^2}\left({\bm \nu}(t)\cdot\rr\right) \left(\frac{\mu}{\lambda+2\mu}I_2+
 2\frac{\lambda+\mu}{\lambda+2\mu}G(\rr)
 \right)\firstReviewer{,}\\
W_0(\tau,t)&:=&   ={-\frac{\mu(\lambda+\mu)}{\lambda+2\mu} \frac{\partial^2}{\partial \tau\, \partial t}\frac{1}{{\pi}}\left(-\log r\:I_2+G(\rr)\right)} 
 \end{eqnarray*}
with  
 \[
\rr :={\bf x}(\tau)-{\bf x}(t),\quad  r := |\rr|=|{\bf x}(\tau)-{\bf x}(t)|,\quad  G(\rr) :=\frac{1}{r^2}\rr \rr^\top =
\begin{bmatrix}
 r_{1}^2&r_{1} r_{2}\\
 r_{1}r_{2}&r_{2}^2
\end{bmatrix}.
 \]
\end{proposition}
\begin{proof} We refer the reader to Appendix and Remark \ref{remark:3.4}. 
\end{proof} 
Notice that in the previous results we have used the same symbols ${\bm r}$ and $r$ to represent different but very closely related quantities. We believe that this similarity justifies this slight abuse of notation.

Let us stress some consequences of the previous result.  The (matrix) functions $V_0(t,\tau),$ $K_0(t,\tau)$, $K^\top_0(t,\tau)$ and $W_0(t,\tau)$ are the kernels of the corresponding \otherCorrections{BIOs} (single layer, double layer, adjoint double layer and hypersingular operator) for the elasticity layer operators. Hence, if we introduce the integral operators
\begin{eqnarray*}
 \Lambda \varphi &:=& -\frac{1}{2\pi}\int_{0}^{2\pi}\log\Big(4 e^{-1} \sin^2\frac{\cdot-t}2\Big)\varphi(t)\,{\rm d}t\firstReviewer{,}\\
 \label{eq:def:H}\\
\Hi g   &:=&  - i\Lambda g' +\widehat{g}(0) = \mathrm{p.v.}\, \frac{1}{2\pi i}\int_{0}^{2\pi} \cot\frac{t-\cdot}2 g(t)\,{\rm d}t +\frac{1}{2\pi}\int_0^{2\pi}  g(t)\,{\rm d}t,  
\end{eqnarray*}
 which are actually Fourier multiplier operators whose action is explicitly given by 
\begin{equation}\label{eq:Lambda1}
 \Lambda g = \widehat{g}(0) + \sum_{n\ne 0} \frac{1}{|n|}\widehat{g}(n) \exp(in\,\cdot),\quad \Hi {g} = -\sum_{n<0} \widehat{g}(n) \exp(in\,\cdot)+
 \sum_{n\ge 0} \widehat{g}(n) \exp(in\,\cdot),   
\end{equation}
together with the matrix operators 
\begin{equation}\label{eq:LambdasH}
\bm{\Lambda} =\begin{bmatrix}
               \Lambda& \\
               & \Lambda  
              \end{bmatrix},\quad \bm{\Lambda}^{-1} =\begin{bmatrix}
               \Lambda^{-1}& \\
               & \Lambda^{-1}  
              \end{bmatrix},\quad 
\bm{H} = 
              \begin{bmatrix}
               & -\Hi\\
               \Hi &
              \end{bmatrix}=\bm{H}^\top
\end{equation}
the following regularity results can be easily established, cf. Remark \ref{remark:3.4}
\begin{equation}\label{eq:principal:symbols}
 \begin{aligned}
 \Kii-\otherCorrections{\alpha}\bm{H}&:\bm{H}^s\to\bm{H}^{s+2}\firstReviewer{,}&
 \Vii-\otherCorrections{\beta}\bm{\Lambda}&:\bm{H}^s\to\bm{H}^{s+3}\firstReviewer{,}\\ 
 \Wii-\otherCorrections{\delta}\bm{\Lambda}^{-1}&:\bm{H}^s\to\bm{H}^{s+1}\firstReviewer{,}&
 \Kii^\top-\otherCorrections{\alpha}\bm{H}&:\bm{H}^s\to\bm{H}^{s+2} \firstReviewer{,}
 \end{aligned}
 \end{equation}
 where  
\[
  \alpha := \frac{i \mu }{2 (\lambda +2 \mu )}\otherCorrections{\in  i\mathbb{R}^+},\quad \beta := \frac{\lambda +3 \mu }{4 \mu  (\lambda +2 \mu )}\otherCorrections{>0}, \quad \delta :=  -\frac{\mu  (\lambda +\mu )}{\lambda +2 \mu }\otherCorrections{<0}. 
\] 
We note in passing that  
\begin{equation}\label{eq:id:alpha:beta:gamma}
 \alpha ^2+\beta  \delta +\frac{1}{4}=0. 
\end{equation}
On the other hand,  using the 2D \firstReviewer{G\"unther  derivative cf. \cite[\S 2.2]{hsiao2008boundary}}
\[
 {\bm D}\bm{g} \otherCorrections{:=}
 \begin{bmatrix}
                 -g_2'\\
                 g_1'
                \end{bmatrix} 
 \]
it holds that
\[
 \bm{H} =\bm{\Lambda}\bm{D},\quad \bm{\Lambda}^{-1} =  \bm{D}^\top \bm{\Lambda}\bm{D}.
\]
Besides, it is straightforward to derive the following coercivity relations
\begin{equation}\label{eq:coercivity:lambda}
 \langle  \bm{\Lambda}^{-1} \bm{g}       ,\overline{\bm{g}}       \rangle \ge 
 \|\bm{g}\|_{\Hiib^{1/2}}^{2},\quad 
 \langle  \bm{\Lambda}      \bm{\varphi} ,\overline{\bm{\varphi}} \rangle \ge \|\bm{\varphi}\|_{\Hiib^{-1/2}}^{2}.
\end{equation} 
Having introduced these notations, \firstReviewer{set }
 \begin{equation}\label{eq:C0}
 {\cal C}_0 \otherCorrections{:=} \begin{bmatrix}
               \alpha \bm{H} & -\beta \bm{\Lambda}  \\
              \delta  \bm{\Lambda}^{-1}  & -\alpha \bm{H}^{\top}\\
              \end{bmatrix}= \begin{bmatrix}
               \alpha \bm{H} & -\beta \bm{\Lambda}  \\
              \delta  \bm{\Lambda}^{-1}  & -\alpha \bm{H}\\
              \end{bmatrix}.
\end{equation}
\begin{proposition}\label{ps_1}
For any $s\in\mathbb{R}$,  ${\cal C}_0:\Hiib^{s+1}\times\Hiib^s\to \Hiib^{s+1}\times\Hiib^s$ is continuous. Moreover, 
\[
 {\cal C}-{\cal C}_0:\Hiib^{s+1}\times\Hiib^s\to \Hiib^{s+3}\times\Hiib^{s+2}. 
\]
Besides, 
\[
 {\cal C}_0^2=\frac14{\cal I}, \quad \left(\pm \frac1{2}{\cal I} + {\cal C}_0\right)^2 =\pm \frac1{2}{\cal I} + {\cal C}_0. 
\]
\end{proposition}
\begin{proof}
 The first assumption as well as the mapping properties for ${\cal C}-{\cal C}_0$ are  consequence of 
the definition of ${\cal C}_0$ and \eqref{eq:principal:symbols}. Notice also that 
\[
  \bm{H}^2 = -\bm{I}
\]
which with the expressions for $\alpha$, $\beta$ and $\delta$, see \eqref{eq:id:alpha:beta:gamma}, proves that ${\cal C}^2=\frac{1}4{\cal I}$. The last result is straightforward. 
\end{proof}

For a pseudodifferential operator $ \otherCorrections{\bm{A}}$ of order $m$ (i.e.  $ \otherCorrections{\bm{A}}:\Hiib^s\to\Hiib^{s-m}$) we will denote by $\mathrm{PS}( \otherCorrections{\bm{A}})$ its principal symbol, that is we require that (i)  $\mathrm{PS}( \otherCorrections{\bm{A}})$ is a Fourier multiplier and (ii) $ \otherCorrections{\bm{A}}-\mathrm{PS}( \otherCorrections{\bm{A}})$ is a pseudodifferential operator of order $m-1$, i.e. $ \otherCorrections{\bm{A}}-\mathrm{PS}( \otherCorrections{\bm{A}}):\Hiib^s\to\Hiib^{s-m+1}$ (the principal symbol operators are not uniquely defined, and we will use in what follows various such operators depending on the circumstances). This notation is readily extended to matrix operators, and, in view of previous result Proposition~\ref{ps_1}, we have 
\[
 \mathrm{PS}({\cal C})={\cal C}_0.
\]
A key ingredient in deriving alternative BIE formulations of scattering and transmission Navier problems is the incorporation of principal symbols of Dirichlet-to-Neumann operators.  Let us define $\Yii_\pm$ as the usual parameterization of the Dirichlet-to-Neumann operators $\Yii_{\pm,\Gamma}$ introduced in \eqref{eq:defYs}. It is easy to derive from classical results that $\Yii_\pm:\Hiib^{1/2}\to \Hiib^{-1/2}$. 

\begin{lemma}\label{lemma:im:Tpm}
 For any  $\Hiib^{1/2}\ni \bm{g}\ne 0$  it holds
 \[
  \Im \langle \Yiim \bm{g},\overline{\bm g}\rangle =0,\quad 
  \Im \langle \Yiip \bm{g},\overline{\bm g}\rangle {>}0.
 \]
\end{lemma} 
 {
\begin{proof} The proof uses the same arguments, and follow the same lines, as in  the Helmholtz  equation  cf. \cite{KressColton}. We present the proof here just for the sake of completeness. 

Given $\bm{g}\in \Hiib^{1/2}(\Gamma)$, let  ${\bf v}_-\in\Hiib^1(\Omega_-),\ {\bf v}_+\in\Hiib^{1}_{\rm loc}(\Omega_+)$ be the solutions of the Dirichlet problem for Navier equations in interior ($\Omega_-$) and exterior ($\Omega_+$) domain. Therefore
\[
 \gamma{\bf v}_{\pm} = \bm{g},\quad 
 T_\pm {\bf v}_{\pm} = \Yii_{\pm}\gamma{\bf v}_{\pm}=\Yii_{\pm}\bm{g}.
\]

For the interior Dirichlet-to-Neumann operator the result is consequence of the  first Green identity since
\begin{eqnarray*}
     \langle \overline{\gamma \bf v}_- ,  {\Tiim{\bf v}_-} \rangle &=&  
  \int_\Gamma  \overline{\gamma_{\Gamma} {\bf v}}_- \cdot  {\Tii_{\Gamma,-} {\bf v}_-}  \\
  &=& \int_{\Omega_-} 2\mu \varepsilon({\bf v}_-):\overline{\varepsilon({\bf v}_-)}
  +\lambda |\nabla \cdot{\bf v}_-|^2-\omega^2\int_{\Omega_-} |{\bf v}_-|^2\in\mathbb{R}. 
\end{eqnarray*}
(We have denoted above by ``$: $''   the Frobenius  matrix inner product). 

The proof for the exterior Dirichlet-to-Neumann operator is only slightly different. First, take $B_R$  a the ball centered at origin and sufficiently large radius $R$. Then 
\begin{eqnarray*}
\langle \overline{\gamma  \bf v}_+ ,   {\Tiip{\bf v}_+}\rangle\!&=&\!
 \int_\Gamma  \overline{\gamma_{\Gamma} {\bf v}_+} \cdot    {\Tii_{\Gamma,+} {\bf v}_+} \\
 &&\hspace{-2cm}=-\int_{\Omega_+\cap B_R} 2\mu \varepsilon({\bf v}_-):\overline{\varepsilon({\bf v}_-)}
  -\lambda |\nabla \cdot{\bf v}_-|^2+\omega^2\int_{\Omega_-\cap B_R} |{\bf v}_-|^2
  + \int_{ \partial B_R} \overline{\gamma_{\partial B_R} {\bf v}_+}\cdot  {\Tii_{  \partial B_R,+} {\bf v}_+}. 
\end{eqnarray*}
Recall that ${\bf v}_+$ satisfies the radiation condition at infinity  \eqref{eqs:RadCondition}. This condition   implies, by  \eqref{eqs:RadCondition:02}-\eqref{eqs:RadCondition:03}, 
\[
 \lim_{R\to \infty} \int_{\partial B_R} \overline{  \gamma_{\partial B_R}{\bf v}_+} \,\Tii_{\partial B_R,+} {\bf v}_+ 
 = \lim_{R\to\infty} ik_p(\lambda+2\mu) \int_{\partial B_R} |{\bf v}_p|^2 +\lim_{R\to\infty} ik_s \mu \int_{\partial B_R} |{\bf  v}_s|^2  \in i\,\mathbb{R}^+,
\]
with  ${\bf v}_p$ and ${\bf v}_s$ being the longitudinal and transversal wave components  of ${\bf v}$. Thus, 
\[
 \Im \langle \overline{\gamma  \bf v}_+ ,   {\Tiip{\bf v}_+}\rangle
 =\lim_{R\to\infty}  k_p(\lambda+2\mu) \int_{\partial B_R} |{\bf v}_p|^2 +\lim_{R\to\infty}  k_s \mu \int_{\partial B_R} |{\bf  v}_s|^2\ge 0. 
 \]
 Furthermore,  ${\bf v}_p$ and  ${\bf v}_s$  are solutions of   (vector) Helmholtz equations  in $\Omega_+$, with their corresponding Helmholtz  (Sommeferld) radiation conditions at infinity cf. \eqref{eqs:RadCondition:01}. Then
\[
 \Im \langle \overline{\gamma  \bf v}_+ ,   {\Tiip{\bf v}_+}\rangle = 0\quad \text{if and only if }\quad
  \lim_{R\to\infty}   \int_{\partial B_R} |{\bf v}_p|^2 =\lim_{R\to\infty}   \int_{\partial B_R} |{\bf v}_p|^2=0
 \]
 which 
implies that ${\bf v}_p$ and  ${\bf v}_s$, and so ${\bf v}_+={\bf v}_p+{\bf v}_s$,  vanish. The second result is now proven.  
\end{proof} 

We are now ready to compute the principal symbols $\mathrm{PS}(\Yii_{\pm})$.  If $\omega^2$ is not an eigenvalue for the Dirichlet problem for the Navier operator in the interior domain $\Omega_{-}$, which is equivalent to the invertibility of single layer BIO $\Vii$, cf. Remark \ref {remark:C:01}, we derive
\[
 \Tiip{\bf u} = -\Vii^{-1}(\tfrac{1}2\firstReviewer{\bm{I}} -{\Kii})\otherCorrections{\gamma{\bf u}} . 
\]
Similarly, if $\omega^2$ is not an eigenvalue for the Neumann problem, which in turn renders the operator $\tfrac{1}2I+{\Kii}$ invertible, we obtain the alternative formula
\[
 \Tiip{\bf u} = (\tfrac{1}2\firstReviewer{\bm{I}} +{\Kii})^{-1}\Wii\otherCorrections{\gamma {\bf u}}. 
\]
Since $\omega^2$ cannot be simultaneously a  Neumann and a Dirichlet eigenvalue in the bounded domain $\Omega_{-}$, we use either of the formulas above for the exterior Dirichlet-to-Neumann operator $\Yiip{\bf u}=\Tiip{\bf u}$ (which is a pseudodifferential operator of order $1$, i.e., $\Yiip:\Hiib^s\to \Hiib^{s-1}$) together with the principal symbol formulas for the BIOs involved in those equations and we derive a formula for its principal symbol
\begin{equation}\label{eq:PSY+}
 {\rm PS}(\Yiip)= -\beta^{-1}\bm{\Lambda}^{-1}\left(\frac{1}2\firstReviewer{\bm{I}}-\alpha\bm{H}\right)
 =\delta \left(\frac{1}2\firstReviewer{\bm{I}}+\alpha\bm{H}\right)^{-1}\bm{\Lambda}^{-1}.
\end{equation}
Assuming that the Dirichlet interior problem is well posed, we can proceed in the same way for the interior Dirichlet-to-Neumann operator  $\Yiim$,  and we get
\begin{equation}\label{eq:PSY-}
 {\rm PS}(\Yiim)=  \beta^{-1}\bm{\Lambda}^{-1}\left(\frac{1}2\firstReviewer{\bm{I}}+\alpha\bm{H}\right)
 =-\delta \left(\frac{1}2\firstReviewer{\bm{I}}-\alpha\bm{H}\right)^{-1}\bm{\Lambda}^{-1}.
\end{equation}
We note that the matrix operators $\bm{H}$ and $\bm{\Lambda}$ commute, and furthermore the two choices provided in the right hand sides of equations~\eqref{eq:PSY+} and respectively~\eqref{eq:PSY-} for the calculation of the principal symbol operators do produce the same result (this can be   verified using  equation~\eqref{eq:id:alpha:beta:gamma}). Finally, it can be easily seen that 
\begin{equation}
 \label{eq:PSk-Ps}
 {\rm PS}(\Yii_\pm)-\Yii_\pm: \Hiib^s\to \Hiib^{s+1}. 
\end{equation}

We present next some straightforward results that will enable us to establish certain coercivity properties that the operators ${\rm PS}(\Yii_\pm)$ enjoy.

\begin{lemma}\label{lemma:10.8}
It holds, 
\[
 \langle   \bm{H}\bm{\varphi},\overline{\bm{g}}\rangle = -\langle   \bm{\varphi}, \overline{\bm{H} \bm{g}}\rangle,  
 \quad \forall \bm{\varphi}\in \Hiib^{-s},\quad  \bm{g}\in \Hiib^{s}. 
\]
Therefore,
\[
 {\Re} \left\langle  \bm{H} \bm{\varphi},\overline{\bm \varphi}\right\rangle = 0, \quad \forall \bm{\varphi}\in \Hiib^0. 
\]
Besides, if $ c\in(-1/2,1/2)$ then  
 \[
 \left\langle \left(\tfrac12{\bm I}+c i \bm{H}\right)\bm{\varphi},\overline{\bm \varphi}\right\rangle \ge \left(\tfrac12-c\right)\|{\bm \varphi}\|_{\Hiib^0}^{2} .
 \]
\end{lemma}
\begin{proof}
\firstReviewer{The first} result follows from the definition \firstReviewer{of $\Hiib$ which, in addition, implies}  trivially the second result. 
Finally, 
\[
 \left\langle \left(\tfrac12{\bm I}+c i \bm{H}\right)\bm{\varphi},\overline{\bm \varphi}\right\rangle\ge \Big(
 \tfrac{1}2- |c|\underbrace{\|\bm{H}\|_{\Hiib^0\to\Hiib^0}}_{=1} \Big)\|\bm\varphi\|_{\Hiib^0}^{2}
\]
and the result is proven.
 
\end{proof}}

As consequence of the above lemma we can easily derive the following coercivity result that will be useful in what follows:

\begin{proposition}\label{prop:2.5}
It holds
\[
\left\langle  \mathrm{PS}(\Yiim) \bm{g},\overline{\bm{g}}\right\rangle \ge \beta^{-1}(1-|\alpha|)\|\bm{g}\|_{\Hiib^{1/2}}^{2},\quad 
 -\left\langle  \mathrm{PS}(\Yiip) \bm{g},\overline{\bm{g}}\right\rangle \ge \beta^{-1}(1-|\alpha|)\|\bm{g}\|_{\Hiib^{1/2}}^{2}.
\] 
\end{proposition}
\begin{proof}
Since $\bm{\Lambda}^{-1}$ is positive definite we use its square root $\bm{\Lambda}^{-1/2}$ defined explicitly as
\[
 \bm{\Lambda}^{-1/2} = \begin{bmatrix}
                        \Lambda^{-1/2}\\
                        &\Lambda^{-1/2}
                       \end{bmatrix},\quad \Lambda^{-1/2}g = \widehat{g}(0) +\sum_{n\ne 0}|n|^{1/2}\widehat{{g}}(n), 
\]
which turns out to be an isometry from $\Hiib^{s}$ into $\Hiib^{s-1/2}$. It can easily be checked that $\bm{\Lambda}^{-1/2}$ commutes with $\bm{H}$ and hence
\begin{eqnarray*}
{\Re} \left\langle  \mathrm{PS}(\Yiim) \bm{g},\overline{\bm{g}}\right\rangle &=&
\beta^{-1} {\Re} \left\langle  (\tfrac{1}2 I +\alpha\bm{H})\bm{\Lambda}^{-1/2} \bm{g},\overline{\bm{\Lambda}^{-1/2} \bm{g}}\right\rangle
\ge \beta^{-1}(\tfrac12-|\alpha|)\| \bm{\Lambda}^{-1/2} \bm{g}\|_{\Hiib^{0}}^{2}
\\
&=& \beta^{-1}(\tfrac12-|\alpha|)\| \bm{g}\|_{\Hiib^{1/2}}^{2}.
\end{eqnarray*}
The second result follows similarly.
\end{proof}

We will also use an alternative form of the principal symbol operator $\mathrm{PS}({\Yii_\pm})$ that relies on complexified square roots. To this end, we select a complex wavenumber \firstReviewer{
$\kappa$ with $\Re \kappa,\Im \kappa>0$} and we introduce a complexified version of the Fourier multiplier operator defined in equation~\eqref{eq:Lambda1} in the following form
\begin{equation}\label{eq:Lambda:kappa} 
 \Lambda_{\kappa} =  \sum_{n=-\infty}^\infty (n ^2-\kappa^2)^{-1/2} \widehat{g}(n) \exp(in\,\cdot) 
\end{equation}
and then we define the matrix operators $\bm{\Lambda}_\kappa:=\begin{bmatrix}\Lambda_{\kappa} & \\ & \Lambda_\kappa\end{bmatrix}$ and $\bm{\Lambda}^{-1}_\kappa$ accordingly.  \secondReviewer{We mention that the square roots in formula~\eqref{eq:Lambda:kappa} are selected so that $\Im{(n ^2-\kappa^2)^{-1/2}}>0$. The incorporation of these alternative Fourier multipliers in the DtN calculus in connection to BIE has been originally proposed in~\cite{AntoineX}, and has been adopted in the community because these operators provide improved approximations of the Fourier modes in the transition region between propagating and evanescent modes (that is frequencies such that $|n|\approx k$).} Clearly, $\bm{\Lambda}_\kappa:\Hiib^s\to\Hiib^{s+1}$, and, besides, by construction 
\begin{equation}\label{eq:coercivity:LambdaKappa:01}
 \Re \langle \bm{\Lambda}_\kappa \bm{\varphi},\overline{\bm{\varphi}} \rangle \ge c\|\bm{ \varphi}\|_{\Hiib^{-1/2}}^{2},\quad 
 \Re \langle \bm{\Lambda}^{-1}_\kappa \bm{\varphi},\overline{\bm{\varphi}} \rangle \ge c\|\bm{ g}\|_{\Hiib^{1/2}}^{2}
\end{equation}
and 
\begin{equation}\label{eq:coercivity:LambdaKappa:02}
 \Im \langle \bm{\Lambda}_\kappa \bm{\varphi},\overline{\bm{\varphi}} \rangle \ge c\|\bm{ \varphi}\|_{\Hiib^{-3/2}}^{2},\quad 
 -\Im \langle \bm{\Lambda}^{-1}_\kappa \bm{\varphi},\overline{\bm{\varphi}} \rangle \ge c\|\bm{ g}\|_{\Hiib^{-1/2}}^{2}.
\end{equation}
The newly defined operators $\bm{\Lambda}_\kappa$ enjoy the following regularity properties
\[
 \bm{\Lambda}_\kappa-\bm{\Lambda}:\Hiib^s\to \Hiib^{s+3}, \quad \bm{\Lambda}^{-1}_\kappa-\bm{\Lambda}^{-1}:\Hiib^s\to \Hiib^{s+{1}}.
\]
We can thus define complexified principal symbol DtN operators in the following manner
\begin{subequations}\label{eq:PSkappa}
\begin{eqnarray}
 {\rm PS}_\kappa (\Yiim)&:=&  \beta^{-1}\bm{\Lambda}_\kappa^{-1}\left(\frac{1}2\firstReviewer{\bm{I}} +\alpha\bm{H}\right)
 =-\delta \left(\frac{1}2\firstReviewer{\bm{I}} -\alpha\bm{H}\right)^{-1}\bm{\Lambda}_\kappa^{-1},\label{eq:PSkappa:a}\\
 {\rm PS}_\kappa(\Yiip)&:=& -\beta^{-1}\bm{\Lambda}_\kappa^{-1}\left(\frac{1}2\firstReviewer{\bm{I}} -\alpha\bm{H}\right)
 =\delta \left(\frac{1}2\firstReviewer{\bm{I}} +\alpha\bm{H}\right)^{-1}\bm{\Lambda}_\kappa^{-1}.\label{eq:PSkappa:B}
\end{eqnarray}
\end{subequations}
We establish that the complexified operators defined above are indeed principal symbol DtN operators that enjoy the same coercivity properties as the DtN operators themselves
\begin{proposition}\label{prop:Y:complexificated}
 It holds $ {\rm PS}  (\Yii_\pm)- {\rm PS}_\kappa (\Yii_\pm):\Hiib^s\to \Hiib^{s+1}$. Furthermore, there exist  $c_1>0$ and, provided that $\Re(n^2-\kappa^2)\ne 0$, $c_2>0$ so that 
\begin{align*}
 \Im \left\langle  \mathrm{PS}_\kappa (\Yiip) \bm{g},\overline{\bm{g}}\right\rangle  &\ge c_1\|\bm{g}\|^{2}_{\Hiib^{-1/2}},\quad & -\Im \left\langle  \mathrm{PS}_\kappa (\Yiim) \bm{g},\overline{\bm{g}}\right\rangle& \ge c_1 \|\bm{g}\|^{2}_{\Hiib^{-1/2}},\\
 -\Re \left\langle  \mathrm{PS}_\kappa(\Yiip) \bm{g},\overline{\bm{g}}\right\rangle &\ge c_2\|\bm{g}\|^{2}_{\Hiib^{1/2}},\quad &
  \Re \left\langle  \mathrm{PS}_\kappa(\Yiim) \bm{g},\overline{\bm{g}}\right\rangle &\ge c_2\|\bm{g}\|^{2}_{\Hiib^{1/2}}.
 \end{align*}
\end{proposition}
\begin{proof}
Using that
\[
 (n^2-\kappa^2)^{1/2} = a_n^2-i b^2_n, \quad a_n,b_n>0,\ a_n\approx n,\quad  b_n\approx n^{-1}, \ \text{as }n\to\infty, 
\]
we can easily show that
\[
  \Lambda_{\kappa} ^{-1} = {\Theta}^2_r - i{\Theta}^2_i 
\]
with 
\[
  {\Theta} _rg := \sum_{n=-\infty}^{\infty} a_n \widehat{g}(n)\exp(i\, n\,\cdot\,),\quad
  {\Theta} _ig := \sum_{n=-\infty}^{\infty} b_n \widehat{g}(n)\exp(i\, n\,\cdot\,).
\]
Clearly ${\Theta}_i:{H}^{s}\to {H}^{s+1/2}$ and, provided that $\Re (n^2-\kappa^2)\ne 0$,  ${\Theta}_r:{H}^{s}\to {H}^{s-1/2}$ are continuous invertible operators. 

We then define the matrix operators $\bm{\Theta}_r:=\begin{bmatrix}\Theta_r & \\ & \Theta_r\end{bmatrix}$ and $\bm{\Theta}_i:=\begin{bmatrix}\Theta_i & \\ & \Theta_i\end{bmatrix}$ accordingly and we notice that
\[
  {\rm PS}_\kappa (\Yii_\pm) = \mp \beta^{-1}\bm{\Theta}_r^{2}\left(\frac{1}2\firstReviewer{\bm{I}} \mp \alpha\bm{H}\right)\pm i\beta^{-1}\bm{\Theta}_i^{2}\left(\frac{1}2\firstReviewer{\bm{I}} \mp\alpha\bm{H}\right)
\]
so that
\[
  \langle {\rm PS}_\kappa (\Yii_\pm)\bm{g},\overline{\bm{g}}\rangle = 
  \mp \beta^{-1} \langle\left(\tfrac{1}2\firstReviewer{\bm{I}} \mp \alpha\bm{H}\right)\bm{\Theta}_r \bm{g},\overline{\bm{\Theta}_r\bm{g}}\rangle   \pm 
  {i}
  \beta^{-1} \langle\left(\tfrac{1}2\firstReviewer{\bm{I}} \mp \alpha\bm{H}\right)\bm{\Theta}_i \bm{g},\overline{\bm{\Theta}_i\bm{g}}\rangle.
\]
The result follows now from Lemma \ref{lemma:10.8}. 

\end{proof}
\firstReviewer{
\begin{remark}
\begin{equation}\label{eq:pSV}
{\rm PS}_{\kappa_p,\kappa_s}(\bm{V})(n):=\frac{1}{2\omega^2}\begin{bmatrix}n^2(n^2-\kappa_p^2)^{-1/2}-(n^2-\kappa_s^2)^{1/2} & \\  & n^2(n^2-\kappa_s^2)^{-1/2}-(n^2-\kappa_p^2)^{1/2}\end{bmatrix}
\end{equation}
and respectively
\begin{equation}\label{eq:pSK}
{\rm PS}_{\kappa_p,\kappa_s}(\bm{K})(n):=\frac{i\mu}{2(\lambda+2\mu)}
\begin{bmatrix} 
& n(n^2-\kappa_p^2)^{-1/2}   \\ 
-n(n^2-\kappa_s^2)^{-1/2}&
\end{bmatrix}
\end{equation}
with $\Im{\kappa_p}>0,\ \Im{\kappa_s}>0$. The precise definition of the Fourier multiplier operator ${\rm PS}_{\kappa_p,\kappa_s}(\bm{V})$ is given below, and it requires the use of the tangent and normal fields on $\Gamma$
\[
{\rm PS}_{\kappa_p,\kappa_s}(\bm{V}){\bm\varphi}:=\sum_{n=-\infty}^\infty  \begin{bmatrix}{\bm t}^\top \\ \nn^\top\end{bmatrix}{\rm PS}_{\kappa_p,\kappa_s}(\bm{V})(n)\begin{bmatrix}\widehat{\varphi_t}(n) \\ \widehat{\varphi_n}(n) \end{bmatrix} \exp(in\,\cdot)\quad \varphi_t:={\bm \varphi}\cdot{\bm t},\quad \varphi_n:={\bm \varphi}\cdot\nn.
\]
(Here we assume that the length of the  curve is $2\pi$ and that an arc length parameterization is being used; the required modification for curves of any length is straightforward).  
The principal symbols defined in equations~\eqref{eq:pSV} and~\eqref{eq:pSK} lead via the Somigliana's identities to alternative principal symbol approximations ${\rm PS}_{\kappa_p,\kappa_s} (\Yii_\pm)$ of DtN operators
\begin{equation}\label{eq:PSYpm}
{\rm PS}_{\kappa_p,\kappa_s}(\Yii_\pm)(n)=\mp({\rm PS}_{\kappa_p,\kappa_s}(\bm{V})(n))^{-1}\left(\frac{1}{2}\firstReviewer{\bm{I}} \mp {\rm PS}_{\kappa_p,\kappa_s}(\bm{K})(n)\right).
\end{equation}
\end{remark}
}

\subsection{Duality in product spaces}\label{dual}

In this short subsection we  present  certain properties of matrix operators in connection to duality in product spaces. We are interested in  deriving an explicit formula for the  dual of \firstReviewer{the} Calder\'on operators ${\cal C}:\Hiib^{s_1}\times \Hiib^{s_2}\to \Hiib^{s_1}\times \Hiib^{s_2}$ with an emphasis on the case where $s_1=1/2=-s_2$. 
 To this end, for  general matrix operators 
\[
 {\cal A}:\begin{bmatrix}
           \firstReviewer{\textbf{A}}_{11}&\firstReviewer{\textbf{A}}_{12}\\
           \firstReviewer{\textbf{A}}_{21}&\firstReviewer{\textbf{A}}_{22}\\
          \end{bmatrix}:\Hiib^{s_1}\times \Hiib^{s_2}\to \Hiib^{t_1}\times \Hiib^{t_2},\quad \text{(that is, $\firstReviewer{\textbf{A}}_{ij}:\Hiib^{s_j}\to \Hiib^{t_i}$)}
\] 
 we work with the non-standard representation: $(\Hiib^{s }\times \Hiib^{t})' = \Hiib^{-t }\times \Hiib^{-s}$ via the non-standard duality product
\begin{equation}\label{eq:dualproductTwo}
 [(\bm{g}_1,\bm{\varphi}_1), (\bm{g}_{2},\bm{\varphi}_2) ] =  \langle \bm{g}_1,\bm{\varphi}_2\rangle
 - \langle \bm{g}_2,\bm{\varphi}_1\rangle.
\end{equation}
Notice that, with this convention:   
\[
 \left[ {\cal A}\begin{bmatrix}\bm{g}_1\\ \bm{\varphi}_1 \end{bmatrix}, \begin{bmatrix}\bm{g}_2\\ \bm{\varphi}_2\end{bmatrix} \right] =    \left[ \begin{bmatrix}\bm{g}_1\\ \bm{\varphi}_1 \end{bmatrix}, {\cal A}^\top \begin{bmatrix}\bm{g}_2\\ \bm{\varphi}_2\end{bmatrix}  \right],\quad  \text{with\ } {\cal A}^\top :=\begin{bmatrix}
            \firstReviewer{\textbf{A}}^\top_{22} &-\firstReviewer{\textbf{A}}^\top_{12} \\
           -\firstReviewer{\textbf{A}}^\top_{21} & \firstReviewer{\textbf{A}}^\top_{11}
          \end{bmatrix}.
\]
In particular,  we find that 
\[
 {\cal C}^\top = -{\cal C},\quad \left(\frac{1}2 {\cal I} \pm  {\cal C} \right)^\top =
 \left(\frac{1}2 {\cal I} \mp  {\cal C} \right)^\top 
\]
i.e., the transpose of the interior   \firstReviewer{Calderón} operator is the exterior one and vice versa. This property will simplify the analysis of boundary integral formulations for transmission problems as we will \firstReviewer{see in section 5. 
}

\section{Boundary integral formulations. Impenetrable case}\label{BIEf}

     
     We present in what follows various strategies to derive BIE formulations of scattering problems. Besides the classical combined field formulations CFIE we derive regularized formulations that rely on the use of approximations of the DtN operators. The design of the regularized formulations for elastodynamic scattering and transmission problems follows the blueprint from the Helmholtz case~\cite{turc1}.
\subsection{Dirichlet boundary conditions}
 
Let us consider $\Omega_+$ the exterior of a closed, smooth, simply connect curve $\Gamma$. 
\begin{equation}\label{eq:Dirichlet}
 \left|
 \begin{array}{rcl}
  {\bf u} \in \Hiib_{\rm loc}^1(\Omega_+)\\
   {\nabla}\cdot\bm{\sigma}({\bf u} )+\omega^2{\bf u}  &=&0,  \quad \text{in $\Omega_+$},\\ 
 \gamma_{\Gamma}{\bf u} &=&  {\bf f}_\Gamma \\
 \multicolumn{3}{l}{\text{+ RC }}
 \end{array}
 \right.
\end{equation}
%
 
Just like in the Helmholtz case~\cite{BrackhageWerner}, the classical approach~\firstReviewer{\cite{chaillat2008fast,chaillat2017fast,CoSt:1990}} in the case of Dirichlet boundary conditions is to look for a scattered field in the form of a Combined Field representation
\[
{\bf u}(\x):=(\DLii{\bm{\varphi}})(\x)-i\eta_D (\SLii{\bm{\varphi}})(\x),\quad \x\in\Omega_+
\]
where the coupling parameter $\eta_D\neq 0$ and \otherCorrections{the $2\pi$-periodic function $\bm{\varphi}:\mathbb{R}\to\mathbb{C}^2$ is the solution of} the Combined Field Integral Equation (CFIE)
\begin{equation}\label{eq:CFIE_D}
  \frac{1}{2}{\bm{\varphi}}+\Kii{\bm{\varphi}}-i\eta_D \Vii{\bm{\varphi}}={\bf f}.
\end{equation}
 (With, as taken throughout this paper, ${\bf f}:={\bf f}_\Gamma\circ{\bf x}$). 
The question of selecting a value of the coupling parameter $\eta_D$ that leads to formulations with superior spectral properties (and thus faster convergence rates for iterative solver solutions) can be settled via Dirichlet-to-Neumann arguments~\cite{chaillat2017fast}. We begin with Somigliana's identities \eqref{eq:Somigliana:2} which  we rewrite considering as ${\bf u}|_\Gamma$ as the primary unknown boundary density and incorporating the DtN operator in the form
\[
{\bf u}(\x)= \DLii[\gamma {\bf u}](\x)-  \SLii [ \Yiip \gamma {\bf u}] (\x),\quad \x\in\Omega_+.
\]
The main idea in constructing regularized formulations is to use easily constructable approximations $\RD$ of the DtN operator $\Yiip
$ and look for combined field representations in the form
\begin{equation}\label{eq:CFIER_D_potential}
  {\bf u}(\x):=(\DLii{\bm g})(\x)-(\SLii[\RD{\bm g}])(\x),\quad \x\in\Omega_+
\end{equation}
leading to the Combined Field Regularized Integral Equation (CFIER)
\begin{equation}\label{eq:CFIER_D}
  \frac{1}{2}{\bm g}+\Kii{\bm g}-\Vii\RD{\bm g}={\bf f}\quad {\rm on}\ \Gamma.
\end{equation}
We show in what follows that under certain assumptions on the regularizing operator $\RD$, the CFIER {is} well posed: 

\begin{theorem} \label{theo:9.1} Consider a preconditioner $\RD$ that satisfies the following two properties: (a) $\RD-\Yiip:\Hiib^s\to \Hiib^{s}$ {is continuous} and (b) the  non-null  condition {holds} 
\[
\Im \langle \RD \bm{\varphi},\overline{\bm{\varphi}}\rangle\ne 0,\quad 
{ \bm{\varphi}\ne 0.}
\]
Then the CFIER operator $\frac12 \firstReviewer{\bm{I}}+\Kii-\Vii[\RD] $
is an invertible compact perturbation of the identity with continuous inverse.

In particular, \otherCorrections{in view of \eqref{eq:PSk-Ps} and Proposition \ref{prop:Y:complexificated}}, ${\rm PS}_\kappa(\Yiip)$ is a valid election. 
\end{theorem}

\begin{proof} Notice that by construction: 
\[
\frac12 \firstReviewer{\bm{I}}+\Kii -\Vii\RD -\left(\frac12 \firstReviewer{\bm{I}}+\alpha\bm{H}-\beta\bm{\Lambda} \mathrm{PS}(\Yiip)\right): \Hiib^s \to \Hiib^{s+1}.
\]
Furthermore, by \eqref{eq:PSY+}
\[
\frac12 \firstReviewer{\bm{I}}+\alpha\bm{H}-\beta\bm{\Lambda}\, \mathrm{PS}(\Yiip) = \frac12 \firstReviewer{\bm{I}}+\alpha\bm{H}
\otherCorrections{+}\beta\bm{\Lambda}\left[\beta^{-1}\bm{\Lambda}^{-1}\left(\frac12 \firstReviewer{\bm{I}}-\alpha\bm{H}\right)\right]=I 
\]
%
Consequently, the operator is a compact perturbation of the identity, and, as such, owing to Fredholm alternative, it suffices to establish the  injectivity in order to complete the proof. To this end, let us assume that $\bm{g}\in \Hiib^s$ is in the kernel of the operator  and define
\[
 {\bf v}(\x):=(\DLii\bm{g})(\x)-(\SLii[\RD\bm{g}])(\x),\quad \x\in\mathbb{R}^2\setminus\Gamma.
\]
Clearly ${\bf v}_+={\bf v}|_{\Omega_+}$ is a  \firstReviewer{radiating} solution of the Navier equation in $\Omega_+$ with zero Dirichlet boundary conditions on $\Gamma$, which implies that ${\bf v}_+$ is identically zero in $\Omega_+$. Hence,  
\[
\left( \frac{1}2{\cal I}+{\cal C}\right)\begin{bmatrix}\bm{g}\\
                              \RD\bm{g} 
                             \end{bmatrix}={\bf 0},\quad \text{and so }\quad
\left(\frac{1}2{\cal I}-{\cal C}\right)\begin{bmatrix}\bm{g}\\
                              \RD\bm{g} 
                             \end{bmatrix}=\begin{bmatrix}\bm{g}\\
                              \RD\bm{g} 
                             \end{bmatrix}.\quad 
\]
In other words, ${\bf v}_-={\bf v}|_{\Omega_-}$ is a solution of the Navier equation in the domain $\Omega_-$ that satisfies 
\[
 \gamma{\bf v}_- = \bm{g}, \quad \Tiim{\bf v}_- = \RD\bm{g}  
\]
and so
\begin{eqnarray*}
\langle \RD{\bm g},\overline{{\bm g}}\rangle&=&
\langle T{\bf v}_-,\overline{{\gamma}{\bf v}_-}\rangle\in\mathbb{R}
\end{eqnarray*}
by Lemma \ref{lemma:im:Tpm}.
The proof is concluded on the basis of the assumption (b).
\end{proof}

The regularizing operator $\RD= {\rm PS}_\kappa(\Yiip)$ is a pseudodifferential operator of order $1$ whose numerical evaluation is consequently more involved. Using the high-frequency approximation $|\kappa|\to\infty$ in the definition of the Fourier multiplier $\Lambda_{\kappa}^{-1}$, we can construct a simple regularizing operator
\[
\RD_1=-\frac{2\mu(\lambda+2\mu)}{\lambda+3\mu}i\kappa,
\]
which, incidentally, can be interpreted as delivering a quasi-optimal choice for the coupling parameter $\eta_D$ in the CFIE formulation
\begin{equation}\label{eq:eta_D}
\eta_D^{\rm opt}=\frac{2\mu(\lambda+2\mu)}{\lambda+3\mu}k_s
\end{equation}
if we choose $\kappa=k_s$. \firstReviewer{We observed in practice that this choice of the coupling constant $\eta_D^{\rm opt}$ appears to deliver consistently superior iterative behavior of GMRES solvers}. We remark that a similar, easily implementable low-order approximation of the DtN operator was proposed in~\cite{chaillat2017fast} as a regularizing operator in the Dirichlet case.

\subsection{Neumann boundary conditions}
In the case of Neumann boundary conditions, i.e., 
\begin{equation}\label{eq:Neumann}
 \left|
 \begin{array}{rcl}
  {\bf u} \in \Hiib_{\rm loc}^1(\Omega_+)\\
   {\nabla}\cdot\bm{\sigma}({\bf u} )+\omega^2{\bf u}  &=&0,\quad \text{in $\Omega_+$},\\ 
 \Tii_{+,\Gamma}{\bf u}  &=& {\bm \lambda}_\Gamma\\
 \multicolumn{3}{l}{\text{+ RC }}
 \end{array}
 \right.
\end{equation}
we can look for a scattered field in the form of a Combined Field representation akin to the Burton-Miller formulation in the Helmholtz case~\cite{BurtonMiller}
\[
{\bf u}(\x):=-(\SLii\bm{\varphi})(\x) + i\eta_N (\DLii\bm{\varphi})(\x),\quad \x\in\Omega_+
\]
where the coupling parameter $\eta_N\neq 0$, leading to the Combined Field Integral Equation (CFIE)
\begin{equation}\label{eq:CFIE_N}
  \frac{1}{2}\bm{\varphi}-\Kii^\top \bm{\varphi}+i\eta_N \Wii\bm{\varphi}=\bm{\lambda}.
\end{equation}
Here again we start by recasting the Somigliana's identities looking at $ T{\bf u} $ as the primary unknown boundary density and making use of the the Neumann-to-Dirichlet (NtD) operator (which is the inverse of the DtN operator)
\[
{\bf u}(\x)= (\DLii [\Yiip^{-1}\{T  {\bf u}\} ])(\x)- (\SLii \{T {\bf u}\} )(\x) ,\quad \x\in\Omega_+\firstReviewer{.}
\]
The construction of regularized formulations relies again on available approximations $\RN $ of the NtD operator $Y^{-1}$ via looking for combined field representations in the form
\begin{equation}\label{eq:CFIER_N_potential}
  {\bf u}(\x):=(\DLii[\RN \bm{\varphi}])(\x)-(\SLii\bm{\varphi})(\x),\quad \x\in\Omega_+
\end{equation}
leading to the Combined Field Regularized Integral Equation (CFIER)
\begin{equation}\label{eq:CFIER_N}
  \frac{1}{2}\bm{\varphi}-\Kii^\top\bm{\varphi}+\Wii\RN \bm{\varphi}={\bm\lambda}\quad {\rm on}\ \Gamma.
\end{equation}
Again here, the choice $\RN = [{\rm PS}_\kappa(\Yiip)]^{-1}$  leads to well posed CFIER formulations, at least in the case when the boundary $\Gamma$ is smooth enough.  Indeed, the following result can be established analogously to that in Theorem~\ref{theo:9.1}
\begin{theorem} \label{theo:9.2} Consider a preconditioner $\RN $ that satisfies the following two properties: (a) $\RN -[\Yiip]^{-1}:\Hiib^s\to \Hiib^{s+1}$ and (b) the  non-null  condition 
\[
\Im \langle \RN  \bm{\bm{\varphi}},\overline{ \bm{\varphi}}\rangle\ne 0,\quad 
\bm{\varphi}\ne 0.
\]
Then the CFIER operator 
\[
 \frac12 \firstReviewer{\bm{I}}-\Kii^\top+\Wii[\RN ] 
\]
is an invertible  compact perturbation of the identity. 

In particular, $\RN =({\rm PS}_\kappa(\Yiip))^{-1}$ is a valid choice. 
\end{theorem}
\begin{proof}
From 
\[
  \frac12 \firstReviewer{\bm{I}}-\alpha \bm{H}+\delta \bm{\Lambda}^{-1}\left[\delta \Big(\frac{1}2\firstReviewer{\bm{I}} +\alpha\bm{H}\Big)^{-1}\bm{\Lambda}^{-1} \right]^{-1} = I 
\]
we can easily see that  $\frac12 \firstReviewer{\bm{I}}-\Kii^\top+\Wii[{\rm PS}_\kappa(\Yiip)]^{-1}$ is a compact perturbation of the identity. 

The unicity follows from a similar argument as in Theorem \ref{theo:9.1}: take $\bm{\varphi}$ in the kernel of the operator, define
\[
  {\bf w}(\x):=(\DLii[\RN \bm{\varphi}])(\x)-(\SLii\bm{\varphi})(\x),\quad \x\in\mathbb{R}^2\setminus\Gamma 
\]
 and  derive next ${\bf w}$ vanishes in $\Omega_+$. Therefore, with ${\bf w}_-= {\bf w}|_{\Omega_-}$  and $\bm{g}={\rm PS}_\kappa(\Yiip) ^{-1}\bm{\varphi}$ 
\[
 \langle {\rm PS}_\kappa(\Yiip) \bm{g},\overline{\bm{g}}\rangle = \langle \bm{\varphi},\overline{\RN \bm{\varphi}}\rangle = 
\langle T{\bf w}_-,\overline{\gamma {\bf w}_-}\rangle
\]
and the result is consequence of Lemma \ref{lemma:im:Tpm} and Proposition \ref{prop:Y:complexificated}.

\end{proof}

Similarly, using high-frequency approximations we can construct a simple regularizing operator
\[
\RN _1=i\frac{\lambda+3\mu}{2\mu(\lambda+2\mu)}\kappa^{-1},
\]
which, delivers a quasi-optimal choice for the coupling parameter $\eta_N$ in the CFIE formulation
\begin{equation}\label{eq:eta_N}
\eta_N^{\rm opt}=\frac{\lambda+3\mu}{2\mu(\lambda+2\mu)}k_s^{-1}.
\end{equation}
We remark that similar low-order approximations of NtD operators have been used in~\cite{chaillat2020analytical} to construct CFIE formulations with superior spectral properties.

Similarly, we can construct \emph{direct} regularized formulations in the case of Neumann boundary conditions following the ideas in~\cite{turc_corner_N}. Assuming that a smooth incident field ${\bf u}^{\rm inc}$ (which is a solution of the Navier equation in the whole $\mathbb{R}^2$) impinges on the obstacle $\Omega$, we will derive these BIEs in terms of unknown boundary quantity %
 $\gamma {\bf u}^{\rm tot}=\gamma ({\bf u}+{\bf u}^{\rm inc})$. We obtain from Somigliana's identities by taking into account the fact that $T{\bf u}^{\rm tot}=0$ on $\Gamma$
\begin{equation}\label{eq:S1}
  {\bf u}(\x)=(\DLii{\bf u}^{\rm tot})(\x),\quad \x\in\Omega_+.
\end{equation}
Applying the Dirichlet trace on $\Gamma$ to formula~\eqref{eq:S1} we obtain
\begin{equation}\label{eq:S11}
  \frac{1}{2}\gamma {\bf u}^{\rm tot} - \Kii\gamma {\bf u}^{\rm tot} =\gamma {\bf u}^{\rm inc} 
\end{equation}
while applying the traction operator to formula~\eqref{eq:S1} we get
\begin{equation}\label{eq:S12}
   \Wii\gamma {\bf u}^{\rm tot}  =-T{\bf u}^{\rm inc}.
\end{equation}
We combine BIE~\eqref{eq:S11} and a preconditioned (on the left) version of the BIE~\eqref{eq:S12} to arrive at the DCFIER
\begin{equation}\label{eq:S22}
  \frac{1}{2}\gamma {\bf u}^{\rm tot}(\x)- \Kii\gamma  {\bf u}^{\rm tot}  +([\RN \Wii]\gamma  {\bf u}^{\rm tot})  =\gamma {\bf u}^{\rm inc} -\RN  T{\bf u}^{\rm inc}  , 
\end{equation}
We note that the operators on the left hand side of the DCFIER formulation is the  transpose
of the operator in the CFIER formulation, a situation that is similar to that in the Helmholtz case~\cite{turc_corner_N}.  Direct formulations can be then used  in the case when $\Omega$ is a Lipschitz domain in order to take advantage of the increased regularity of $\gamma {\bf u}^{\rm tot}$.

\section{The penetrable case}\label{BIEt}
 
We consider in this section the penetrable case with transmission conditions, as formulated in  \eqref{eq:Transmission:0}: 
\begin{equation}\label{eq:Transmission}
 \left|
 \begin{array}{rcl}
 \multicolumn{3}{l}{  {\bf u}_- \in \Hiib^1(\Omega_-),\quad {\bf u}_+ \in \Hiib_{\rm loc}^1(\Omega_+),}\\
   {\nabla}\cdot\bm{\sigma}_-({\bf u}_-)+\omega_-^2{\bf u}_-  &=&0,\quad \mbox{in }\Omega_-,\\ 
   {\nabla}\cdot\bm{\sigma}_+({\bf u}_+ )+\omega_+^2{\bf u}_-  &=&0,\quad \mbox{in }\Omega_+,\\ 
 \gamma {\bf u}_+-\gamma {\bf u}_- &=& -\gamma {\bf u}^{\rm inc},\\ 
 \Tiip{\bf u}_+-\Tiim{\bf u}_- &=& -\Tiip{\bf u}^{\rm inc},\\
 \multicolumn{3}{l}{\text{+ RC}\firstReviewer{.}}
 \end{array}
 \right.
\end{equation}

We present two approaches. In the first one we analyze two classical formulations, the Costabel-Stephan \eqref{eq:stephan:form} and Kress-Roach formulation \eqref{eq:KiptmanKress} and introduce two new ones, extensions of that firstly introduced in \cite{dominguez2016well} for Helmholtz equation.  

In the second subsection we analyze the Optimized Schwarz formulation.

\subsection{Regularized formulations}

In the first approach, we seek to solve for the Cauchy data $(\gamma {\bf u}_\pm,T_\pm{\bf u})$  and reconstruct the solution via
\[
 {\bf u}_{\pm} = \pm \DLii_{\pm}\gamma {\bf u}_\pm  \mp \SLii_{\mp}\otherCorrections{T_\pm{\bf u}},
\]
 where the $\pm$ notation is used for denoting the corresponding boundary layer potential operators for the exterior/interior problem.  Hence, denoting by
\[
 {\cal C}_{\pm} := \begin{bmatrix}
              \Kii_{\pm}  & -\Vii_{\pm}\\
              \Wii_{\pm}  & -\Kii_{\pm}^\top   
            \end{bmatrix}
\] 
we have that
\[
\left(\frac{1}2 {\cal I} + {\cal C}_+\right)\begin{bmatrix}
                                         \gamma {\bf u}_+\\
                                         \Tiip{\bf u}_+ 
                                        \end{bmatrix}=\begin{bmatrix}
                                         \gamma {\bf u}_+\\
                                         \Tiip{\bf u}_+ 
                                        \end{bmatrix},\quad
\left(\frac{1}2 {\cal I} - {\cal C}_-\right)\begin{bmatrix}
                                         \gamma {\bf u}_-\\
                                         \Tiim{\bf u}_- 
                                        \end{bmatrix}=\begin{bmatrix}
                                         \gamma {\bf u}_-\\
                                         \Tiim{\bf u}_-
                                        \end{bmatrix}.
\]
These last two properties together with the fact that  
\[ \left(\frac{1}2 {\cal I} + {\cal C}_+\right)\begin{bmatrix}
                                         \gamma {\bf u}^{\rm inc}\\
                                         \Tiip{\bf u}^{\rm inc} 
                                        \end{bmatrix}={\bf 0} 
\]
can be combined to derive a system of BIE equations for the solution of penetrable Navier problems 
\begin{equation}\label{eq:stephan:form}
{\cal L}_{\rm SC} \begin{bmatrix}\gamma {\bf u}_-\\
                   \Tiim{\bf u}_-
               \end{bmatrix} := -({\cal C}_+ + {\cal C}_-)
               \begin{bmatrix}
                  \gamma {\bf u}_-\\
                   \Tiim{\bf u}_-
               \end{bmatrix}=\begin{bmatrix}
                  \gamma {\bf u}^{\rm inc}\\
                   \Tiip{\bf u}^{\rm inc} 
               \end{bmatrix}\firstReviewer{.}
\end{equation}
We refer in what follows to the formulation~\eqref{eq:stephan:form} as the Stephan-Costabel  \firstReviewer{as first introduced in cf. \cite{CoSt:1990} 
which is  in turn the counterpart of the acoustic formulation introduced by these two authors in \cite{CoSt:1985}.}
 

 Besides, since
\[
  \left(\frac{1}2 {\cal I} + {\cal C}_-\right)\begin{bmatrix}
                  \gamma {\bf u}_-\\
                   \Tiim{\bf u}_-
               \end{bmatrix}={\bf 0}
\]
we find also that 
\begin{equation}\label{eq:DCFIER}
   \left(\frac{1}2 {\cal I} + {\cal C}_-  
   -{\cal R}^{\top }({\cal C}_+ + {\cal C}_-)\right)\begin{bmatrix}
                  \gamma {\bf u}_-\\
                   \Tiim{\bf u}_-
               \end{bmatrix}={\cal R}^\top \begin{bmatrix}
                  \gamma {\bf u}^{\rm inc}\\
                   \Tiip{\bf u}^{\rm inc} 
               \end{bmatrix} 
\end{equation}
 for any suitable operator
 \[
{\cal R}=\begin{bmatrix} 
          \firstReviewer{\bm{R}}_{11} & \firstReviewer{\bm{R}}_{12}\\
          \firstReviewer{\bm{R}}_{21} & \firstReviewer{\bm{R}}_{22}
         \end{bmatrix}
         \quad (\text{recall that } {\cal R}^\top =\begin{bmatrix} 
           \firstReviewer{\bm{R}}^\top_{22} & -\firstReviewer{\bm{R}}^\top _{12}\\
          -\firstReviewer{\bm{R}}^\top_{21} &  \firstReviewer{\bm{R}}^\top_{11}
         \end{bmatrix}).
 \]
Alternatively, we can work with the adjoint of the formulation~\eqref{eq:DCFIER}, in the sense discussed in Section~\ref{dual}
 \begin{equation}\label{eq:DCFIERt}
   \Big(\frac{1}2 {\cal I} - {\cal C}_- +({\cal C}_+ + {\cal C}_-){\cal R}\Big)\begin{bmatrix}
                  \bm{g}\\
                    \bm{\varphi}
               \end{bmatrix}=\begin{bmatrix}
                  \gamma {\bf u}^{\rm inc}\\
                   \Tiip{\bf u}^{\rm inc} 
               \end{bmatrix}
\end{equation}
Clearly, owing to duality arguments, the first formulation~\eqref{eq:DCFIER} is uniquely solvable if and only if so is the second (dual) formulation~\eqref{eq:DCFIERt}. The connection between the boundary densities $( \bm{g},
\bm{\varphi})$ which are the solution of the adjoint formulation~\eqref{eq:DCFIERt} and the fields ${\bf u}_\pm$ is as follows
\begin{equation}\label{eq:CFIERsolutions}
\begin{aligned}
 {\bf u}_+ &= \begin{bmatrix} \DLii_+  & -\SLii_+ \end{bmatrix} {\cal R}  \begin{bmatrix}
                   \bm{g}\\
                   \bm{\varphi}
               \end{bmatrix} &&= \DLii_+(\firstReviewer{\bm{R}}_{11} \bm{g}+\firstReviewer{\bm{R}}_{12}\bm{\varphi})
                     -\SLii_+(\firstReviewer{\bm{R}}_{21} \bm{g}+ \firstReviewer{\bm{R}}_{22} \bm{\varphi})\\
 {\bf u}_- &= \begin{bmatrix}
                 -\DLii_- & \SLii_- 
                \end{bmatrix}( {\cal R}  -I)\begin{bmatrix}
                   \bm{g}\\
                   \bm{\varphi}
               \end{bmatrix}&&= \DLii_-(\bm{g} -\firstReviewer{\bm{R}}_{11} \bm{g}-\firstReviewer{\bm{R}}_{12}\bm{\varphi})
                     - \SLii_-(\bm{\varphi}-\firstReviewer{\bm{R}}_{21}  \bm{g}-\firstReviewer{\bm{R}}_{22}\bm{\varphi}).\\
\end{aligned}
\end{equation}
In other words, formulations~\eqref{eq:DCFIER} and respectively~\eqref{eq:DCFIERt} can be viewed as regularized direct and indirect BIE formulations of elastodynamic transmission problems.

 In particular, selecting ${\cal R} =\frac{1}2 {\cal I}$ in either of the formulations~\eqref{eq:DCFIER} or~\eqref{eq:DCFIERt} we obtain the Kress-Roach type formulation
\begin{equation}\label{eq:KiptmanKress}
  {\cal L}_{\rm KR} \begin{bmatrix}
                  \gamma {\bf u}_-\\
                   \Tiim{\bf u}_-
               \end{bmatrix}:= ( {\cal I} + {\cal C}_-- {\cal C}_+  )\begin{bmatrix}
                  \gamma {\bf u}_-\\
                   \Tiim{\bf u}_-
               \end{bmatrix}= 2\begin{bmatrix}
                  \gamma {\bf u}^{\rm inc}\\
                   \Tiip{\bf u}^{\rm inc}
               \end{bmatrix}. 
\end{equation}
\firstReviewer{This formulation was first introduced, for acoustic problems, by Kress and Roach in \cite{KressRoach}. We also refer   the reader to \cite[Theorem 2.1]{Lo:2015} for the analysis of this formulation in 3D elastodynamics problems. As we will show below, Proposition~\ref{propLSC0}  and Theorem~\ref{th:stephan-costabel}, the well posedness of  \eqref{eq:stephan:form}--\eqref{eq:KiptmanKress} can be stated in an easier way for 2D problems, obtaining, in addition, crucial information about the eigenvalues distribution of the associated operator in the complex plane.}  We stress that, unlike the Kress-Roach formulations for the Helmholtz transmission problems,  the BIE formulation~\eqref{eq:KiptmanKress} is not a second kind formulation since the double layer BIOs are no longer compact in this setting.

We will establish next that the Stephan-\firstReviewer{Costabel} and the Kress-Roach type formulations are well posed.  The proof follows the ideas used in the seminal work \cite{CoSt:1985,KressRoach} for which the well-posedness of the transmission problem with reversed material properties
\begin{equation}\label{eq:Adjoint}
 \left|
 \begin{array}{rcl}
 \multicolumn{3}{l}{  {\bf v}_- \in \Hiib^1(\Omega_-),\quad {\bf v}_+ \in \Hiib_{\rm loc}^1(\Omega_+)}\\
   {\nabla}\cdot\bm{\sigma}_-({\bf v}_+)+\omega_-^2{\bf v}_+  &=&0,\quad \mbox{in }\Omega_+,\\ 
   {\nabla}\cdot\bm{\sigma}_+({\bf v}_- )+\omega_+^2{\bf v}_-  &=&0,\quad \mbox{in }\Omega_-,\\ 
 \gamma {\bf v}_+  -\gamma{\bf v}_- &=& -{\bf u}^{\rm inc},\\ 
 \Tiip{\bf v}_+ -\Tiim{\bf v}_-&=& -\Tiip{\bf u}^{\rm inc}, \\
 \multicolumn{3}{l}{\otherCorrections{\text{+ RC}}}
 \end{array}
 \right.
\end{equation}
will be essential. 

Let us introduce (with $\alpha_{\pm}\in i\mathbb{R}$, $\beta_\pm>0>\delta_{\pm}$ as in \eqref{eq:C0}) operators
 \begin{eqnarray*}
{\cal L}_{0,\rm SC} &:=&  -( {\cal C}_{+,0}+  {\cal C}_{-,0})= -\begin{bmatrix}
               (\alpha_++\alpha_-) \bm{H} & -(\beta_++\beta_-)\bm{\Lambda}  \\
              (\delta_++\delta_-)  \bm{\Lambda}^{-1}  & -(\alpha_++\alpha_-) \bm{H}\\
              \end{bmatrix},\\
{\cal L}_{0,{\rm KR}} &:=& {\cal I} + {\cal C}_{-,0} -{\cal C}_{+,0}= \otherCorrections{{\cal I}}+\begin{bmatrix}
               (\alpha_--\alpha_+) \bm{H} & -(\beta_--\beta_+)\bm{\Lambda}  \\
              (\delta_--\delta_+)  \bm{\Lambda}^{-1}  & -(\alpha_--\alpha_+) \bm{H}\\
              \end{bmatrix}
\end{eqnarray*}
which are simply the principal parts  for ${\cal L}_{\rm SC}$ and ${\cal L}_{\rm KR}$.

We can show easily that 
\begin{equation}\label{eq:10:23}
 {\cal L}_{0,\rm SC}^2 = \rho {\cal I}
\end{equation}
where 
 \begin{eqnarray}
 \rho &=&  -{ (\beta_++\beta_-)(\delta_++\delta_-)-(\alpha_++\alpha_-)^2}\nonumber
  \\
 &=&
 \frac{(\lambda_+ (\mu_++\mu_-)+\mu_+ (\mu_++3 \mu_-)) (\lambda_- (\mu_++\mu_-)+\mu_- (3 \mu_++\mu_-))}{4 \mu_+ \mu_- (\lambda_++2 \mu_+) (\lambda_-+2 \mu_-)}\label{eq:10:23b}\\
 &=&  F({\lambda_+},{\mu_+},\frac{{\lambda_-}}{{\lambda_+}},\frac{{\mu_-}}{ {\mu_+}} )=
  F({\lambda_-},{\mu_-},\frac{{\lambda_+}}{{\lambda_-}},\frac{{\mu_+}}{{\mu_-}}),\nonumber
\end{eqnarray}
with 
\[
 F(x,y,r,s) := \frac{\left((s+1) x y+(3 s+1) y^2\right) \left(r (s+1) x y+s (s+3) y^2\right)}{4 s y^2 (x+2 y) (r x+2 s y)}.
\]
We prove in the next lemma that $\rho\ge 1$ with $\rho =1$ if and only if $\mu_+=\mu_-$.

\begin{lemma}\label{lemma:F}
For all $(x,y,r,s)\in\mathbb{R}^4_+$, $
 F(x,y,r,s)\ge 1$ with $F(x,y,r,s)=1$ if and only if $s=1$.
%
\end{lemma}
\begin{proof} 
Clearly $ F(x,y,r,1)=1$. Besides, 
 \[
  \partial_r F(x,y,r,s) = \frac{(s-1) x y (s x+3 s y+x+y)}{4 (x+2 y) (r x+2 s y)^2}. 
 \]
Hence  the result follows readily since for $(x,y,r,s)\in\mathbb{R}^4_+$, $  \partial_r F(x,y,r,s)>0$, respectively  $  \partial_r F(x,y,r,s)<0$,   if  $s>1$, respectively $s<1$. 

\end{proof}

\begin{proposition}\label{propLSC0}
${\cal L}_{0,\rm SC},  {\cal L}^\top_{0,\otherCorrections{{\rm KR}}}:\Hiib^s \times \Hiib^{s-1}  \to  \Hiib^s \times \Hiib^{s-1} $ are normal and invertible operators. Moreover,  the eigenvalues of ${\cal L}_{0,\rm SC}$ and $ {\cal L}^\top_{0,\otherCorrections{{\rm KR}}}$  are $\{\pm\sqrt{\rho}\}$ and $\{1\pm \sqrt{1-\rho}\}$ respectively where  $\rho>1$ if $\mu_+\ne \mu_-$ and $1$ otherwise. 
\end{proposition}
\begin{proof}
It follows from equation~\eqref{eq:10:23} that ${\cal L}_{0,\rm SC}$ is invertible and its eigenvalues  are $ \pm \sqrt{\rho}$. The last claim can be established by considering the matrix operator
 \[
  {\cal Q} = \begin{bmatrix}
 & \varepsilon i \bm{\Lambda}\\
\otherCorrections{-}\varepsilon^{-1} i \bm{\Lambda}^{-1}&
\end{bmatrix}\quad \text{where}\ \varepsilon=\left(-\frac{\beta_++\beta-}{\delta_++\delta_-}\right)^{1/2}>0
 \]
which satisfies 
 \[
  {\cal Q}^{-1} = {\cal Q},\quad 
 {\cal Q}
{\cal L} _{0,\rm SC} 
{\cal Q}=
-{\cal L}_{0,\rm SC}. 
 \] 
 Thus, if $({\bm\lambda},\bm{g})$ is an eigenvector of ${\cal L}_{0,\rm SC}$ associated with the eigenvalue $\pm \sqrt{\rho}$, then 
   ${\cal Q}({\bm\lambda},\bm{g})^\top $  is also an eigenvector of ${\cal L}_{0,\rm SC}$ associated with the eigenvalue $\mp   \sqrt{\rho}$.

 Regarding $ {\cal L}^\top_{0,\otherCorrections{{\rm KR}}}$, we see first, using \eqref{eq:10:23}, 
 \[
  {\cal L}^\top_{0,\otherCorrections{{\rm KR}}} {\cal L}^\top_{0,\otherCorrections{{\rm KR}}}= \big(\firstReviewer{\bm{I}} - ({\cal C}_{-,0}-{\cal C}_{+,0})\big)
  \big( \firstReviewer{\bm{I}} + ({\cal C}_{-,0}-{\cal C}_{+,0})\big) =   {\cal L}^2_{0,\rm SC}=\rho^2\firstReviewer{\bm{I}} 
 \]
 from where the invertibility follows readily. Furthermore, from 
 \begin{equation}\label{eq:10:24}
 ({\cal C}_{-,0}-{\cal C}_{+,0})^2 =  2  {\cal C}_{-,0}^2+2{\cal C}_{+,0}^2-({\cal C}_{-,0}+{\cal C}_{+,0})^2 =(1-\rho )\firstReviewer{\bm{I}}  
 \end{equation}
 we conclude that the eigenvalues of $ {\cal L}^\top_{0,\otherCorrections{{\rm KR}}}$ are $1\pm \sqrt{1-\rho}$. Using a similar argument to the one applied above establishes the fact that both signs are actually achieved in the formula for the eigenvalues of the operator $ {\cal L}^\top_{0,\otherCorrections{{\rm KR}}}$. 
\end{proof}

Next theorem extends this result to ${\cal L}_{\rm SC}$ and ${\cal L}_{\rm KR}$: 

\begin{theorem}\label{th:stephan-costabel}
The operators 
${\cal L}_{\rm SC}, {\cal L}_{\rm KR}:\Hiib^s \times \Hiib^{s-1}  \to  \Hiib^s \times \Hiib^{s-1} $ are invertible.

Furthermore, the eigenvalues of ${\cal L}_{\rm SC}$ and ${\cal L}_{\rm KR}$  are clustered around $\pm \sqrt{\rho}$ and $1\pm \sqrt{1-\rho}$ respectively. 
\end{theorem}
\begin{proof} From the Proposition~\ref{propLSC0} we conclude that ${\cal L}_{\rm SC}$ and ${\cal L}_{\rm KR}$ are compact perturbations of invertible operators. Hence, according to the Fredholm theory, their invertibility follows once we establish their injectivity, which, in turn, will be proved  using the same ideas as for Helmholtz transmission problems. 

To this end, take first  $(\bm{g},\bm{\varphi})\in\mathop{\rm Ker}{\cal L}_{\rm SC}$ and define the fields in $\mathbb{R}^2\setminus\Gamma$  
\[
 {\bf u}_\pm  = \pm\DLii_{\pm }\bm{g}\mp\SLii_{\pm }\bm{\varphi}.
\]
Clearly $({\bf u}_+|_{\Omega_+},{{\bf u}_-}|_{\Omega_-})$  are  solutions of the original transmission problem with ${\bf u}^{\rm inc}={0}$. Therefore ${\bf u}_+=0$ in $\Omega_+$ and ${\bf u}_-=0$ in $\Omega_-$. Since  
\[
\left(\frac{1}2{\cal I} - {\cal C}_+\right)\begin{bmatrix}\bm{g}\\
                       \bm{\varphi} 
                      \end{bmatrix}=
                      \begin{bmatrix}\bm{g}\\
                       \bm{\varphi}
                      \end{bmatrix}=
\left(\frac{1}2{\cal I} + {\cal C}_-\right)\begin{bmatrix}\bm{g}\\
                       \bm{\varphi}
                      \end{bmatrix}
\]
we conclude that 
 (${\bf u}_-|_{\Omega_+},{\bf u}_+|_{\Omega_-})$   solve the adjoint problem \eqref{eq:Adjoint}, again with ${\bf u}^{\rm inc} =0$. Then by hypothesis,  ${\bf u}_-|_{\Omega_+}$ and ${\bf u}_-|_{\Omega_+}$ vanish too. Hence ${\bf u}^{\pm}=0 $ and by the jump properties for the potentials,  we have also  $(\bm{g},\bm{\varphi})={\bf 0}$.

Consider now $(\bm{g},\bm{\varphi})\in\mathop{\rm Ker}{\cal L}_{\rm KR}$ and define in this case
\begin{eqnarray*}
 {\bf v}^{\pm} = \DLii_{\mp }\bm{g}-\SLii_{\mp }\bm{\varphi}\Big|_{\Omega_\pm}. 
\end{eqnarray*}
Then,
  ${({\bf v}_-,{\bf v}_+)}$ solves the adjoint problem \eqref{eq:Adjoint}, as above with ${\bf u}^{\rm inc}=0$. Thus ${{\bf v}_{\pm}}$ vanish in $\Omega_\pm$. Hence 
\[
\left(   \frac12 {\cal I} + {\cal C}_-\right)\begin{bmatrix}
                        \bm{g}\\
                        \bm{\varphi}
                       \end{bmatrix} = \left( -\frac12 \otherCorrections{{\cal I}} +{\cal C}_-\right)\begin{bmatrix}
                        \bm{g}\\
                        \bm{\varphi}
                       \end{bmatrix} = {\bf 0}
\]
which implies  $( \bm{g},\bm{\varphi})\in {\rm Ker}({\cal L}_{\rm SC})$ and  therefore,  $( \bm{g},\bm{\varphi})={\bf 0}$.

%

Regarding the eigenvalues distribution, let first $(\lambda_n)_n$ be the eigenvalues for ${\cal L}_{\rm SC}$. Then, $\{\lambda^2_n\}_n$ are  eigenvalues for  ${\cal L}^2_{\rm SC} = {\cal L}_{0,\rm SC}^2 + \firstReviewer{\bm{K}} = \rho I +{\cal K}$.  Since ${\cal K}$ when acting from $\Hiib^s\times \Hiib^{s-1}$ into itself is compact, $\rho-\lambda^2_n$ can only accumulate at zero, from where this result follows. 

The proof for ${\cal L}_{\rm KR}$ follows similarly. Note, as in \eqref{eq:10:24}, 
\[
 ({\cal C}_--{\cal C}_+)^2 =   {\cal I}-({\cal C}_-+{\cal C}_+)^2 = {\cal I}-{\cal L}^2_{\rm SC}.
\]
Hence, if $\{\mu_n\}_n$ are the eigenvalues for ${\cal L}_{\rm KR}$,   $\{(\mu_n-1)^2\}_n$ are eigenvalues for ${\cal I}-{\cal L}^2_{\rm SC}$ which can only accumulate at $1-\rho$. Then, $\{\mu_n\}_n$ are clustered around $1\pm\sqrt{1-\rho}$. 

\end{proof}

 We are ready to analyze the CFIER formulation~\eqref{eq:DCFIER} and~\eqref{eq:DCFIERt} under suitable assumptions on the regularizing operator ${\cal R}$. 
 Specifically, we assume 
 
 \paragraph{Assumption 1} We assume ${\cal R}=
 \begin{bmatrix}
          \firstReviewer{\bm{R}}_{11}&\firstReviewer{\bm{R}}_{12}\\
          \firstReviewer{\bm{R}}_{21}&\firstReviewer{\bm{R}}_{22}                           
                           \end{bmatrix}$ satisfies the following properties:
 \begin{enumerate}
  \item[(a)]  There exists ${\cal K}:\Hiib^{s+1}\times\Hiib^{s}\to \Hiib^{s+1+\ell}\times\Hiib^{s+\ell}$ with $\ell\ge 1$ so that 
 \begin{equation}\label{eq:cond:a}
     \frac{1}2 {\cal I} - {\cal C}_-     +({\cal C}_+ + {\cal C}_-){\cal R} =  {\cal A}  +{\cal K},
     \end{equation}
with ${\cal A}:{\Hiib^{s+1}\times\Hiib^{s}\to\Hiib^{s+1}\times\Hiib^{s}}$ invertible with   the eigenvalues {\em clustered} around a few accumulation points. 
\item[(b)] We have 
 \begin{equation}\label{eq:cond:b}
 \begin{aligned}
 &\Im\big( \langle  \bm{g}, \overline{\bm{\varphi}}\rangle -\langle \bm{g}, \overline{ {\firstReviewer{\bm{R}}_{22}\bm{\varphi}}}\rangle-
\langle  {\firstReviewer{\bm{R}}_{12} \bm{\varphi}}, \overline{\bm{\varphi}}\rangle-
\langle  {\bm{g}}, \overline{\firstReviewer{\bm{R}}_{21}\bm{g}}\rangle- \langle \firstReviewer{\bm{R}}_{11}\bm{g} , \overline{\bm{\varphi}}\rangle \big)< 0\\
 & \hspace{5cm}\forall {\bf 0 }\ne \big(\bm{g},\bm{\varphi}\big)\in\Hiib^{1/2}\times \Hiib^{-1/2}.
 \end{aligned}
 \end{equation}
 \end{enumerate}

We establish
\begin{theorem}\label{th:10.6}
Under the assumptions {\rm (a)--(b)} stated above,  for the following operators
 \begin{eqnarray} 
   {\cal L}_{\rm dir}&:=& \frac{1}2 {\cal I} + {\cal C}_- -{\cal R}^\top ({\cal C}_+ + {\cal C}_-),\quad
   {\cal L}_{\rm ind}:= \frac{1}2 {\cal I} - {\cal C}_- + ({\cal C}_+ + {\cal C}_-){\cal R}.
 \end{eqnarray}
 there exists ${\cal K}:\Hiib^{s+1}\times \Hiib^{s}\to \Hiib^{s+3}\times \Hiib^{s+2}$ so that ${\cal L}_{\rm ind}= {\cal A}+{\cal K}$,  ${\cal L}_{\rm dir}= {\cal A}^\top+{\cal K}^\top$ and both operators ${\cal L}_{\rm dir}$ and ${\cal L}_{\rm ind}$ are invertible. 
\end{theorem}
\begin{proof}
Clearly,  \eqref{eq:cond:a} implies the decomposition stated in the statement of the theorem. In particular, both operators ${\cal L}_{\rm dir}$ and ${\cal L}_{\rm ind}$ are also Fredholm, and therefore we just have to prove the invertibility. 

Assume then that
\[
 {\cal L}_{\rm ind}\begin{bmatrix}
                   \bm{g}\\
                   \bm{\varphi}
               \end{bmatrix}={\bf 0}, 
\]
and define by \eqref{eq:CFIERsolutions} $({\bf u}_+,{\bf u}_-)$ in all $\mathbb{R}^2\setminus\Gamma$. Proceeding as in the proof of Proposition of \ref{th:stephan-costabel} we can infer that
\begin{eqnarray*}
 {\bf v}_+ &=& \begin{bmatrix} -\DLii_-  & \SLii_-\end{bmatrix}  ({\cal I} - {\cal R})\begin{bmatrix}
                   \bm{g}\\
                   \bm{\varphi}
               \end{bmatrix}\\ 
  {\bf v}_- &=& \begin{bmatrix} \DLii_+  & -\SLii_+\end{bmatrix}  {\cal R}\begin{bmatrix}
                   \bm{g}\\
                   \bm{\varphi}
               \end{bmatrix}
\end{eqnarray*}
which, by unicity of solution for the transmission problem,  ${\bf v}_+ |_{\Omega_-} ={\bf 0}={\bf v}_- |_{\Omega_+}$. Then 
\[
 \begin{bmatrix}
  \gamma  {\bf v}_+\\
  \Tiip{\bf v}_+
 \end{bmatrix}
=({\cal I} - {\cal R} )\begin{bmatrix}
                   \bm{g}\\
                   \bm{\varphi}
                   \end{bmatrix},\quad  \begin{bmatrix}
  \gamma{\bf v}_-\\
  \Tiim{\bf v}_-
 \end{bmatrix}
 =-{\cal R} \begin{bmatrix}
                   \bm{g}\\
                   \bm{\varphi}
                   \end{bmatrix}.
\]
Then, by Lemma \ref{lemma:im:Tpm}, 
\begin{equation}
 \langle \gamma  {\bf v}_+ ,  \overline{\Tiip{\bf v}_+}\rangle - 
\langle \gamma  {\bf v}_- ,  \overline{\Tiim{\bf v}_-}\rangle
= \langle  \bm{g}, \overline{\bm{\varphi}}\rangle -\langle \bm{g}, \overline{ {\firstReviewer{\bm{R}}_{22}\bm{\varphi}}}\rangle-
\langle  {\firstReviewer{\bm{R}}_{12} \bm{\varphi}}, \overline{\bm{\varphi}}\rangle-
\langle  {\bm{g}}, \overline{\firstReviewer{\bm{R}}_{21}\bm{g}}\rangle- \langle \firstReviewer{\bm{R}}_{11}\bm{g} , \overline{\bm{\varphi}}\rangle \label{eq:debussy:02}
%
%
%
%
\end{equation}
The result follows then by assumption (b).  
\end{proof}

Obviously, the optimal choice of the regularizing operator
\[
 {\cal R }= ({\cal C}_++{\cal C}_-)^{-1}\Big(\frac{1}2 {\cal I} +{\cal C}_-\Big)
\]
which makes
\[
{\cal L}_{\rm dir} = 
{\cal L}_{\rm ind} ={\cal I}.  
\]
However, although the optimal operator ${\cal R}$ defined above satisfies condition (a) in Assumption 1, it fails to satisfy condition (b). Furthermore, the definition of the regularizing operator involves operator inversion. We bypass these issues by resorting to the principal part calculus. Indeed, for a wavenumber $\kappa$ with \otherCorrections{$\Im\kappa>0$} we consider the principal part operators \otherCorrections{(see \eqref{eq:C0})}
\[
 {\cal C}^{\kappa}_{\pm } =\begin{bmatrix}
               \alpha_\pm \bm{H} & -\beta_\pm \bm{\Lambda}_\kappa  \\
              \delta_\pm  \bm{\Lambda}_\kappa^{-1}  & -\alpha_\pm \bm{H}\\
              \end{bmatrix},\quad \otherCorrections{\alpha_\pm = \frac{i \mu_\pm }{2 (\lambda_\pm +2 \mu_\pm )}},\ \otherCorrections{\beta_\pm = \frac{\lambda_\pm +3 \mu_\pm }{4 \mu_\pm  (\lambda_\pm +2 \mu_\pm )}}, \  \otherCorrections{\delta_\pm =  -\frac{\mu_\pm  (\lambda_\pm +\mu_\pm )}{\lambda_\pm +2 \mu_\pm }},
\] 
where the Fourier multiplier operator $\bm{\Lambda}_\kappa$ was defined in equation~\eqref{eq:Lambda:kappa}. Using the result established in equation~\eqref{eq:10:23}
\[
 \left({\cal C}^{\kappa}_+ +
 {\cal C}^{\kappa}_-\right)^{2} = \rho{\cal I}, \quad \text{that is, } \left({\cal C}^{\kappa}_+ +
 {\cal C}^{\kappa}_-\right)^{-1} = \frac{1}\rho \left({\cal C}^{\kappa}_+ +
 {\cal C}^{\kappa}_-\right) 
\]
we define instead the regularizing operators
\begin{equation}\label{def:Rkappa}
  {\cal R }_{\kappa} =  \frac{1}{\rho} \left({\cal C}^{\kappa}_+ +
 {\cal C}^{\kappa}_-\right) \left(\frac12{\cal I}+{\cal C}_-^\kappa \right) = \begin{bmatrix*}
                                 \firstReviewer{\bm{R}}_{11}^{\kappa} &  \firstReviewer{\bm{R}}_{12}^{\kappa} \\
                                 \firstReviewer{\bm{R}}_{21}^{\kappa} &  \firstReviewer{\bm{R}}_{22}^{\kappa}  
                                \end{bmatrix*}.
\end{equation}
The choice of regularizing operators presented in equation~\eqref{def:Rkappa} still renders the operators
\begin{eqnarray}
  {\cal L}^\kappa_{\rm dir}&:=& \left(\frac{1}2 {\cal I} + {\cal C}_-\right) -{\cal R}_\kappa^\top ({\cal C}_+ + {\cal C}_-) \label{eq:iDCFIER:02}\\
   {\cal L}^\kappa_{\rm ind}&:=& \left(\frac{1}2 {\cal I} - {\cal C}_-\right) + ({\cal C}_+ + {\cal C}_-){\cal R}_\kappa   \label {eq:DCFIER:02}    
\end{eqnarray}
compact perturbations of the identity. In what follows we will refer to these formulations
as {\sl Direct Regularized Combined Field Integral Equations} (DCFIER) and respectively {\sl Indirect Regularized Combined Field Integral Equations} (ICFIER). Indeed, with ${\cal R}_\kappa$ selected as in equation~\eqref{def:Rkappa} above we have 
\[
 {\cal L}^\kappa_{\rm dir} = {\cal I}+{\cal K}
\]
with ${\cal K}:\Hiib^{s+1}\times\Hiib^{s}\to \Hiib^{s+3}\times\Hiib^{s+2}$ so that (a) in Assumption 1 is satisfied with $\ell =2$.  A careful computation delivers the explicit form of the four entries of the regularizing operator ${\cal R}$
where
\begin{eqnarray*}
 \firstReviewer{\bm{R}}_{11}^\kappa &=&  \frac{1}{2\rho}(\alpha_++\alpha_-)\bm{H} 
 -\frac{1}{\rho}(\alpha_-(\alpha_++\alpha_-)+\delta_-(\beta_++\beta_-)) \bm{I}\\
 \firstReviewer{\bm{R}}_{12}^\kappa &=&  -\frac{1}{2\rho}(\beta_++\beta_-)\bm{\Lambda}_\kappa 
 -\frac{1}{\rho}(
  \beta_-(\alpha_++\alpha_-)
 -\alpha_-(\beta_++\beta_-)) \bm{\Lambda}_\kappa \bm{H}\\ 
 &=& -\frac{1}{2\rho}(\beta_++\beta_-)\bm{\Lambda}_\kappa -\frac{1}{\rho}(
  \beta_-\alpha_+-\alpha_-\beta_+) \bm{\Lambda}_\kappa \bm{H}\\
  &=&-\frac{\beta_++\beta_-}{\rho} \bm{\Lambda}_\kappa\left(\frac12\firstReviewer{\bm{I}}  +\frac{
  \alpha_+\beta_-
 -\alpha_-\beta_+}{\beta_++\beta_-} \bm{H}\right)\\ 
 \firstReviewer{\bm{R}}_{21}^\kappa &=&\frac{1}{2\rho}(\delta_++\delta_-)\bm{\Lambda}_\kappa^{-1}+\frac{1}{\rho}(\alpha_-(\delta_++\delta_-)-
  \delta_-(\alpha_++\alpha_-)
) \bm{\Lambda}_\kappa^{-1} \bm{H}\\
 \\
 &=&     \frac{\delta_++\delta_-}{\rho} \bm{\Lambda}_\kappa^{-1}\left(\frac12\firstReviewer{\bm{I}}  -\frac{
  \alpha_+\delta_-
 -\alpha_-\delta_+}{\delta_++\delta_-} \bm{H}\right)\\
 \firstReviewer{\bm{R}}_{22}^\kappa &=& -\frac{1}{2\rho}(\alpha_++\alpha_-)\bm{H} -\frac{1}{\rho}(\alpha_-(\alpha_++\alpha_-)+\beta_-(\delta_++\delta_-)) \bm{I}.\\
\end{eqnarray*}

\begin{lemma}\label{lemma:10.9}
There exists $c_1,c_2>0$ so that for any $(\bm{\varphi},\ \bm{g})\in\Hiib^{1/2}\times \Hiib^{-1/2}$ it holds 
 \[
-\Re \langle \firstReviewer{\bm{R}}^\kappa_{12}\bm{\varphi},\overline{\bm{\varphi}}\rangle\ge c_1\|\bm{\varphi}\|^2_{\Hiib^{-1/2}}  \quad  -\Re \langle \firstReviewer{\bm{R}}^\kappa_{21} \bm{g},\overline{\bm{g}}\rangle  \ge c_1 \|\bm{g}\|^2_{\Hiib^{1/2}}.
 \] 
 and 
 \[
-\Im \langle \firstReviewer{\bm{R}}^\kappa_{12}\bm{\varphi},\overline{\bm{\varphi}}\rangle\ge c_1\|\bm{\varphi}\|^2_{\Hiib^{-3/2}}  \quad   \Im \langle \firstReviewer{\bm{R}}^\kappa_{21} \bm{g},\overline{\bm{g}}\rangle  \ge c_1 \|\bm{g}\|^2_{\Hiib^{-1/2}}. 
 \]
\end{lemma}

\begin{proof}
The proof  it is very similar to that of Proposition~\ref{prop:Y:complexificated}. Notice that we can take 
 \[
  \bm{\Lambda}_{\kappa} = 
  \bm{\Theta}^2_r  + i\bm{\Theta}^2_i, \qquad   {\bm{\Lambda}^{-1}_{\kappa}} = 
  \bm{\Omega}^2_r  - i\bm{\Omega}^2_i  
 \]
 where $\bm{\Theta}_r:\Hiib^s\to\Hiib^{s-1/2}$,  $\bm{\Theta}_i:\Hiib^s\to\Hiib^{s+1/2}$, 
 $\bm{\Omega}_r:\Hiib^s\to\Hiib^{s+1/2}$,  $\bm{\Omega}_i:\Hiib^s\to\Hiib^{s+3/2}$
 are  invertible Fourier multiplier  operators that commute with $\bm{H}$.  Given that $  \bm{\Theta}_r$ and   $\bm{\Theta}_i$ commute with $\bm{H}$ we have 
\begin{eqnarray*}
\langle \firstReviewer{\bm{R}}^\kappa_{12}\bm{\varphi},\overline{\bm{\varphi}}\rangle &=& -\frac{\beta_++\beta_-}{\rho} \left(
\left\langle \left(\tfrac12\bm{I}-c_{12} i \bm{H}\right)\bm{\Theta}_r\bm{\varphi},\overline{\bm{\Theta}_r\bm{\varphi}}\right\rangle
+ i \left\langle \left(\tfrac12\bm{I}-c_{12} i \bm{H}\right)\bm{\Theta}_i\bm{\varphi},\overline{\bm{\Theta}_i\bm{\varphi}}\right\rangle
\right)\\
\langle \firstReviewer{\bm{R}}^\kappa_{21}\bm{\varphi},\overline{\bm{\varphi}}\rangle &=& \frac{\delta_++\delta_-}{\rho} \left(
\left\langle \left(\tfrac12\bm{I}+c_{21} i \bm{H}\right) \bm{\Omega}_r\bm{\varphi},\overline{ \bm{\Omega}_r\bm{\varphi}}\right\rangle
- i \left\langle \left(\tfrac12\bm{I}+c_{21} i \bm{H}\right) \bm{\Omega}_i\bm{\varphi},\overline{ \bm{\Omega}_i\bm{\varphi}}\right\rangle
\right) 
\end{eqnarray*}
where
\[
 c_{12}:=  \frac{
  \alpha_+\beta_-
 -\alpha_-\beta_+}{\beta_++\beta_-}i,\quad   c_{21}:= \frac{
  \alpha_+\delta_-
 -\alpha_-\delta_+}{\delta_++\delta_-}i .
\]
Since  $-i\alpha_{\pm}\in(0,1/2)$  and  $\beta_\pm>0>\delta_{\pm}$,  $c_{12},c_{21}\in(-1/2,1/2)$. 
The result is therefore a consequence of Lemma \ref{lemma:10.8}. 
\end{proof}

Finally, we need an additional result
\begin{lemma}\label{lemma:10.10}
 For any $ (\bm{g},\bm{\varphi})\in \Hiib^0\times \Hiib^0$ it holds
 \[
 \langle \bm{g}, \overline{\bm{\varphi}}\rangle -\langle \bm{g},\overline{\firstReviewer{\bm{R}}_{22}^\kappa  {\bm \varphi}}\rangle-
\langle \firstReviewer{\bm{R}}_{11}^\kappa \bm{g} , \overline{\bm{\varphi}}\rangle   = 0.
 \]
\end{lemma}
\begin{proof} Note that since $\alpha_\pm \in i{\mathbb R}$   
 \begin{eqnarray*}
\langle  {\firstReviewer{\bm{R}}_{11}^\kappa\bm{g}},\overline{\bm{\varphi}}\rangle&=&    \frac{1}{2\rho}(\alpha_++\alpha_-) \langle 
  \bm{H}\bm{g},\overline{\bm{\varphi}}\rangle 
-\frac{1}{\rho}(\alpha_-(\alpha_++\alpha_-)+\delta_-(\beta_++\beta_-))\langle \bm{g}, \overline{\bm{\varphi}}\rangle\\
\langle  \bm{g} , \overline{\firstReviewer{\bm{R}}_{22}^\kappa{\bm \varphi} }\rangle&=&   -\frac{1}{2\rho}(\alpha_++\alpha_-) \langle 
  \bm{H}\bm{g},\overline{\bm{\varphi}}\rangle 
-\frac{1}{\rho}(\alpha_-(\alpha_++\alpha_-)
+\beta_-(\delta_++\delta_-)) 
\langle{\bm{g}}, \overline{\bm{\varphi}}\rangle. 
 \end{eqnarray*}
Therefore 
\[
\langle \bm{g},\overline{\firstReviewer{\bm{R}}_{22}^\kappa  {\bm \varphi}}\rangle+
\langle \firstReviewer{\bm{R}}_{11}^\kappa \bm{g} , \overline{\bm{\varphi}}\rangle
 =   \langle \bm{g},\overline{\bm{\varphi}} \rangle
\]
because, after some   calculations, one can show the following identity (cf \eqref{eq:id:alpha:beta:gamma}) 
\begin{eqnarray*}
&& 2  \alpha_-(\alpha_++\alpha_-)+  \delta_-(\beta_++\beta_-)+\beta_-(\delta_++\delta_-) \\
&&\qquad = (\alpha_++\alpha_-)^2+(\beta_++\beta_-)(\delta_++\delta_-)+(\underbrace{\alpha_-^2+\beta_-\delta_--\alpha_+^2-\beta_+\delta_+}_{=0}) = -\rho.
\end{eqnarray*} 
\end{proof}

We are now in the position to prove the main result on the well posedness of the regularized BIE formulations for transmission problems

\begin{theorem}\label{thm_cfier}
 With ${\cal R}_\kappa$ defined in \eqref{def:Rkappa} there exists ${\cal K}:\Hiib^{s+1}\times \Hiib^{s}\to \Hiib^{s+3}\times \Hiib^{s+2}$ such  that the operators ${\cal L}^\kappa_{\rm ind}= \firstReviewer{\bm{I}} +{\cal K}$,  ${\cal L}^\kappa_{\rm dir}= \firstReviewer{\bm{I}} +{\cal K}^\top:\Hiib^{s+1}\times \Hiib^{s}\to \Hiib^{s+1}\times \Hiib^{s}$ are both   invertible. 
\end{theorem}
\begin{proof}
We have already seen that (a) in Assumption 1 is verified with ${\cal A}={\cal I}$. 
On the other hand,  condition (b) for $(\bm{g},\bm{\varphi})\ne {\bf 0}$ is a consequence of Lemma~\ref{lemma:10.9}, which implies
 \[
  \Im  
\langle  {\firstReviewer{\bm{R}}_{12}^\kappa \bm{\varphi}}, \overline{\bm{\varphi}}\rangle,\quad 
  \Im \langle  {\bm{g}}, \overline{\firstReviewer{\bm{R}}_{21}^\kappa\bm{g}}\rangle <0,
 \]
 and Lemma \ref{lemma:10.10}, which established
 \[
 \langle \bm{g}, \overline{\bm{\varphi}}\rangle -\langle \bm{g},\overline{\firstReviewer{\bm{R}}_{22}^\kappa  {\bm \varphi}}\rangle-
\langle \firstReviewer{\bm{R}}_{11}^\kappa \bm{g} , \overline{\bm{\varphi}}\rangle =0.
 \]
We invoke the result  in Theorem~\ref{th:10.6} to conclude the proof.  
\end{proof}

It is relatively straightforward to refine the regularity property of the operators ${\cal K}$ in Theorem~\ref{thm_cfier}:
\begin{corollary}
 Under the same hypothesis of  Theorem~\ref{thm_cfier}, it holds that ${\cal K}:\Hiib^{s}\times \Hiib^{s}\to \Hiib^{s+1}\times \Hiib^{s+1}$ and  the operators ${\cal L}^\kappa_{\rm ind}= I+{\cal K}$,  ${\cal L}^\kappa_{\rm dir}= I+{\cal K}^\top:\Hiib^{s}\times \Hiib^{s}\to \Hiib^{s}\times \Hiib^{s}$ are both   invertible. 
\end{corollary}

\subsection{Optimized Schwarz Domain Decomposition methods}\label{OS}

We finally consider Optimized Schwarz Domain Decomposition Methods~\cite{boubendir2017domain} \otherCorrections{(OS)} for the solution of the transmission problem whereby we connect the \firstReviewer{solutions ${\bf u}_+$ and ${\bf u}_-$ of the exterior and interior time harmonic problems Navier equations via certain transmission operators $\firstReviewer{\bm{\Upsilon}}_\pm:\Hiib^{1/2}\to\Hiib^{-1/2}$:}
\begin{eqnarray}\label{eq:ddm1}
  {\nabla}\cdot\bm{\sigma}_+({\bf u}_+)+\omega^2{\bf u}_+&=&0\quad{\rm in}\ \Omega_+\nonumber\\
  \Tiip[{\bf u}_+ + {\bf u}^{\rm inc}]+\firstReviewer{\bm{\Upsilon}}_+\gamma [{\bf u}_+ + {\bf u}^{\rm inc}]&=&\Tiim {\bf u}_- +\firstReviewer{\bm{\Upsilon}}_+ \otherCorrections{\gamma}{\bf u}_- 
\end{eqnarray}
and
\begin{eqnarray}\label{eq:ddm2}
  {\nabla}\cdot\bm{\sigma}_-({\bf u}_-)+\omega^2{\bf u}_-&=&0\quad{\rm in}\ \Omega\nonumber\\
  \Tiim{\bf u}_- +\firstReviewer{\bm{\Upsilon}}_- \otherCorrections{\gamma} {\bf u}_- &=&\Tiip[{\bf u}_++ {\bf u}^{\rm inc}]+\firstReviewer{\bm{\Upsilon}}_-\gamma  [{\bf u}_+ + {\bf u}^{\rm inc}].
\end{eqnarray}
\firstReviewer{We assume from now on that $\firstReviewer{\bm{\Upsilon}}_{\pm}$ satisfies the following property: 
\begin{equation}\label{eq:cond_coerc}
\Im{\langle\firstReviewer{\bm{\Upsilon}}_+\bm{g},\overline{\bm{g}}\rangle}>0\qquad  - \Im{\langle\firstReviewer{\bm{\Upsilon}}_-\bm{g},\overline{\bm{g}}\rangle}>0.
\end{equation}  }

\begin{proposition}\label{prop:OSM:01}
 Under assumption \eqref{eq:cond_coerc}, the exterior and interior Navier problems with generalized Robin conditions 
 \begin{subequations}\label{eq:01:prop:OSM:01}
\begin{equation}\label{eq:01a:prop:OSM:01}
\left|
 \begin{array}{rcl}
   \multicolumn{3}{l}{{\bf u}_+ \in \Hiib_{\rm loc}^1(\Omega_+),} \\
   {\nabla}\cdot\bm{\sigma}({\bf u}_+ )+\omega^2{\bf u}_+  &=&0,\quad \text{in }\Omega_+,\\  
   \Tiip {\bf u}_+   +\firstReviewer{\bm{\Upsilon}}_+\gamma  {\bf u}_+   &=& {\bm\lambda}_+,\\
   \multicolumn{3}{l}{\otherCorrections{\rm +RC}} 
 \end{array}
 \right. 
\end{equation}
 and
\begin{equation}\label{eq:01b:prop:OSM:01}
\left|
 \begin{array}{rcl}
   \multicolumn{3}{l}{{\bf u}_- \in \Hiib^1(\Omega_-), }  \\
   {\nabla}\cdot\bm{\sigma}({\bf u}_- )+\omega^2{\bf u}_-  &=&0,\quad \text{in }\Omega_-, \\  
   \Tiim {\bf u}_-   +\firstReviewer{\bm{\Upsilon}}_-\otherCorrections{\gamma} {\bf u}_-   &=&  {\bm\lambda}_-,
 \end{array}
 \right.
\end{equation}
\end{subequations}
have a unique solution for any $\otherCorrections{\bm{\lambda}}_{\pm}\in \Hiib^{-1/2}(\Gamma)$.  

Furthermore, \eqref{eq:ddm1}--\eqref{eq:ddm2} are equivalent to the original transmission problem \eqref{eq:NavD}. 
\end{proposition}
\begin{proof}
 The unicity of solution, and by ellipticity the existence, follows from standard arguments in boundary problems for elliptic partial differential equations. 
 
 Notice that it suffices $\firstReviewer{\bm{\Upsilon}}_+-\firstReviewer{\bm{\Upsilon}}_-$ to be injective to show the equivalence between \eqref{eq:ddm1}--\eqref{eq:ddm2} and the transmission problem. But, by hypothesis
 \begin{equation}\label{eq:cond_coercF}
\Im{\langle(\firstReviewer{\bm{\Upsilon}}_+-\firstReviewer{\bm{\Upsilon}}_-)\bm{g},\overline{\bm{g}}\rangle}>0,\quad \forall \bm{g}\ne 0, 
\end{equation} 
which implies as byproduct the injectivity. 
\end{proof}


The \otherCorrections{OS} system~\eqref{eq:ddm1} and~\eqref{eq:ddm2} is typically recast in terms of generalized elastodynamic Robin data on $\Gamma$
\begin{equation}\label{eq:ddm_data}
  \bm{\lambda}_+:= \Tiip{\bf u}_++\firstReviewer{\bm{\Upsilon}}_+\gamma  {\bf u}_+ ,  \quad \bm{\lambda}_-:= \Tiim{\bf u}_-+\firstReviewer{\bm{\Upsilon}}_- \otherCorrections{\gamma}{\bf u}_-  
\end{equation}
in a form that involves the Schwarz iteration operator
\begin{equation}\label{eq:DDM_rtr}
  \begin{bmatrix}
                   I & -\Sii_-
  \\ -\Sii_+ & I 
  \end{bmatrix}\begin{bmatrix}\bm{\lambda}_+\\ \bm{\lambda}_-\end{bmatrix}=
  \begin{bmatrix}
  -\left(T {\bf u}^{\rm inc} +\firstReviewer{\bm{\Upsilon}}_+ \gamma {\bf u}^{\rm inc} \right) \\ 
  \left(T {\bf u}^{\rm inc} +\firstReviewer{\bm{\Upsilon}}_- \gamma {\bf u}^{\rm inc} \right)
  \end{bmatrix}.
  \end{equation}
  \firstReviewer{
In equation~\eqref{eq:DDM_rtr} the operators $\Sii_{\pm}$ are the Robin-to-Robin (RtR) operators associated to Navier equations with generalized Robin boundary conditions whose precise definition is given  by
\begin{subequations}\label{eq:Sii}
\begin{eqnarray}\label{eq:Siia}
 \Sii_+ \bm{\lambda}_+&:=& \Tiip {\bf u}_+ +\firstReviewer{\bm{\Upsilon}}_+\gamma  {\bf u}_+,\quad \\
 \label{eq:Siib}
 \Sii_- \bm{\lambda}_-&:=& \Tiim {\bf u}_-  +\firstReviewer{\bm{\Upsilon}}_-\gamma {\bf u}_-  ,\quad 
\end{eqnarray}
\end{subequations}
 where ${\bf u}_{\pm}$ 
are the solutions of the exterior and interior problems in 
\eqref{eq:01:prop:OSM:01} in Proposition \ref{prop:OSM:01} with data $\bm{\lambda}_\pm$.}

The optimal choice of transmission operators with respect to the iterative solution of the \otherCorrections{OS} formulation~\eqref{eq:DDM_rtr} is given by $\firstReviewer{\bm{\Upsilon}}_+=-\Yiim$ and $\firstReviewer{\bm{\Upsilon}}_-=-\Yiip$ (that is the DtN operators corresponding to each domain $\Omega_\pm$), a ubiquitous pattern in the case of Optimized Schwarz methods involving two subdomains~\cite{Nataf,boubendir2017domain}. Optimized Schwarz methods employ approximations of the DtN operators in the formulation above. As such, we use the following transmission operators
\begin{equation}\label{eq:opt_choice}
  \firstReviewer{\bm{\Upsilon}}_{\mp}:=-\mathrm{PS}_{\kappa}(Y^{\pm}) = 
 \pm \frac{1}{\beta_\pm }\bm{\Lambda}_\kappa^{-1} \left(\frac{1}{2}\firstReviewer{\bm{I}} \mp\alpha_{\pm}\Hiib\right) = 
 \mp \gamma_{\pm }\bm{\Lambda}_\kappa^{-1} \left(\frac{1}{2}\firstReviewer{\bm{I}} \pm\alpha_{\pm}\Hiib\right)^{-1}.  
\end{equation}

\firstReviewer{The Optimized Schwarz algorithm is straightforward to implement as it amounts to the  evaluation of $\Sii_+ \bm{\lambda}_+$ operators. To this end, first we find ${\bf u}_+$, the solution the generalized Robin problem \eqref{eq:01a:prop:OSM:01} with data $\bm{\lambda}_+$; second plug ${T}_+{\bf u}_+$ and $\gamma {\bf u}_+$ into the formula \eqref{eq:Siia} with \eqref{eq:opt_choice}. The  evaluation of $\Sii_- \bm{\lambda}_-$ is identical with problem $\eqref{eq:01b:prop:OSM:01}$ and  \eqref{eq:Siib}
 instead. RtR maps, on the other hand, can be computed in a stable manner using boundary integral equations.}

 We present in what follows: (a)  a simple \firstReviewer{procedure} to compute in a robust manner RtR maps, and therefore to accomplish the evaluation of $\Sii_\pm \bm{\lambda}_\pm$ (Theorems \ref{thm:form_1}-\ref{thm:form_3}); (b) a proof of the well-posedness of \eqref{eq:DDM_rtr} (Theorem \ref{thm:SpSm}).
 
 \firstReviewer{We start with $\Sii_-$, i.e.,  the case of the bounded domain $\Omega_{-}$, as the situation is somewhat simpler. Hence, given $\bm{\lambda}_-$ we write the boundary condition in \eqref{eq:01b:prop:OSM:01} as }
\[
  \begin{bmatrix}
  0&0&\\
 -\firstReviewer{\bm{\Upsilon}}_-&-\firstReviewer{\bm{I}}
 \end{bmatrix}
 \begin{bmatrix}
 \gamma {\bf u}_-\\
 T{\bf u}_-
 \end{bmatrix}=-
 \begin{bmatrix}
  {\bf 0}\\
  \bm{\lambda}_-
 \end{bmatrix}.
\]
On the other hand, we use the Calder\'on projection in the following way
\[
 \begin{bmatrix}
  -\tfrac{1}2 \firstReviewer{\bm{I}} -\Kii_-&\Vii_-\\
   \Wii_-&\tfrac{1}2 \firstReviewer{\bm{I}}-\Kii^\top_- 
 \end{bmatrix}
 \begin{bmatrix}
 \gamma {\bf u}_-\\
 T{\bf u}_-
 \end{bmatrix}=
 \begin{bmatrix}
  {\bf 0}\\
  {\bf 0}
 \end{bmatrix}.
\]
Adding these two matrix equations we derive the following direct BIE formulation of elastodynamic equations with generalized Robin boundary values: 
\begin{equation}\label{eq:RT_formulation} 
  {\cal A}_-\begin{bmatrix}\gamma {\bf u}_-\\ T{\bf u}_-\end{bmatrix}:=  \begin{bmatrix}
 -\frac{1}{2}\firstReviewer{\bm{I}}-\Kii_-   &     \Vii_-\\ 
-\firstReviewer{\bm{\Upsilon}}_-  + \Wii_-&   -\frac{1}{2}\firstReviewer{\bm{I}}-\Kii_-^\top  \end{bmatrix} 
\begin{bmatrix}
\gamma {\bf u}_-\\ 
T{\bf u}_-
\end{bmatrix} 
= 
-\begin{bmatrix} 
0\\
\bm{\lambda}_-
\end{bmatrix}.
  %
  %
\end{equation}

Notice that ${\cal A}_-:\Hiib^{s }\times \Hiib^{s+1}\to \Hiib^{s }\times \Hiib^{s+1}$.
For the following results, recall the duality product $[\cdot,\cdot]$ cf.  \eqref{eq:dualproductTwo} which makes possible the realisation 
 $(\Hiib^{s }\times \Hiib^{t})' = \Hiib^{-t }\times \Hiib^{-s}$. 

\begin{theorem}\label{thm:form_1} The matrix operator $\mathcal{A}_-: \Hiib^{s}\times\Hiib^{s-1}\to \Hiib^{s}\times\Hiib^{s-1}$ is a compact perturbation of 
\[
 \mathrm{PS} ({\cal A}_-) =  \begin{bmatrix}
                            -\frac12 \firstReviewer{\bm{I}} -\alpha_-\bm{H} & \beta_-\bm{\Lambda} \\
                                -\firstReviewer{\bm{\Upsilon}}_-+\otherCorrections{\delta_-}\bm{\Lambda}^{-1}   & -\frac12 \firstReviewer{\bm{I}} -\alpha_-\bm{H}&\\
                           \end{bmatrix}
\]
which is a coercive operator: There exists $c>0$ such that   
\[
 \Re [ \mathrm{PS} ({\cal A}_-)(\bm{g},\bm{\psi}) ,\overline{(\bm{g}, \bm{\psi})}]\ge c \left[
 \|\bm{g}\|_{\Hiib^{1/2}}+ \|\bm{\varphi}\|_{\Hiib^{-1/2}}
 \right]. 
\]
 Furthermore, the operator
$\mathcal{A}_-$ is invertible. 
\end{theorem}
\begin{proof}
Clearly, ${\cal A}_--\mathrm{PS} ({\cal A}_-): \Hiib^{s}\times\Hiib^{s-1}\to
 \Hiib^{s+2}\times\Hiib^{s+1}$ is continuous. Notice also 
\begin{eqnarray*}
[\mathrm{PS} ({\cal A}_-)(\bm{g},{\bm \varphi}),\overline{(\bm{g},{\bm \varphi})}]&=&
-\langle \tfrac{1}2\bm{g}+\alpha_-\bm{H}\bm{g},\overline{\bm{\varphi}}\rangle +
\beta_-\langle  \bm{\Lambda}\bm{\varphi},\overline{\bm{\varphi}}\rangle\\
&&
-\otherCorrections{\delta_-}\langle \bm{\Lambda}^{-1}\bm{g},\overline{\bm{g}}\rangle - \langle \mathrm{PS}_{\kappa}(\Yiip) \bm{g},\overline{\bm{g}}\rangle+
\langle  \tfrac{1}2\bm{\varphi}+\alpha_-\bm{H}\bm{\varphi},\overline{\bm{g}}\rangle
\end{eqnarray*}
and so
\[
 \Re [\mathrm{PS} ({\cal A}_-)(\bm{g},{\bm \varphi}),\overline{(\bm{g},{\bm \varphi})}] =  
  \beta_- \Re \langle  \bm{\Lambda}\bm{\varphi},\overline{\bm{\varphi}}\rangle 
-\otherCorrections{\delta_-} \Re \langle \bm{\Lambda}^{-1}\bm{g},\overline{\bm{g}}\rangle - \Re \langle \mathrm{PS}_{\kappa}(\Yiip) \bm{g},\overline{\bm{g}}\rangle.
\]
On account of relation~\eqref{eq:coercivity:lambda} established in Lemma~\ref{lemma:10.8} (recall that $\beta_->0>\delta_-$) and the result  in Proposition~\ref{prop:Y:complexificated} we conclude the proof of the first result in the theorem. 

Following the Fredholm paradigm, since $\mathrm{PS}( {\cal A}_-): \Hiib^{s}\times\Hiib^{s-1}\to \Hiib^{s}\times\Hiib^{s-1}$ is invertible  we just have to show that $\mathcal{A}_- :\Hiib^{s}\times\Hiib^{s-1}\to \Hiib^{s}\times\Hiib^{s-1}$ is injective.  Let $(\bm{g}_0,{\bm \varphi}_0)\in \mathop{\rm Ker}(\mathcal{A}_-)$ and define the field
\[
{\bf w}(\x):=(\DLii_- \bm{g}_0)(\x)- (\SLii_-{\bm \varphi}_0)(\x),\quad \x\in\mathbb{R}^2\setminus\Gamma.
\]
Clearly $ {\bf w}|_{\Omega_+}$  is a  \firstReviewer{radiating} solution of the Navier equation in $\Omega_+$ with zero Dirichlet boundary conditions on $\Gamma$, and therefore $ {\bf w}$ vanishes in $\Omega_+$. Also, ${\bf w}_-={\bf w}|_{\Omega_-}$ is a solution of the Navier equation in $\Omega_-$ whose Cauchy boundary data on $\Gamma$ \firstReviewer{are}
\[
\otherCorrections{\gamma {\bf w}_-}  = -\bm{g}_0,\qquad \Tiim {\bf w}_-   = -{\bm \varphi}_0.
\]
However, we derive from its definition that
\[
\Tiim {\bf w}_-   =\Wii_-\bm{g}_0+\frac{1}{2}{\bm \varphi}_0-\Kii_-^\top{\bm \varphi}_0
\]
and thus we obtain
\[
\Wii_-\bm{g}_0=\frac{1}{2}{\bm \varphi}_0+\Kii_-^\top{\bm \varphi}_0 .
\]
Given that $(\bm{g}_0,{\bm \varphi}_0)\in \mathop{\rm Ker}(\mathcal{A}_-)$, we also have that
\[
\Wii_-\bm{g}_0-\firstReviewer{\bm{\Upsilon}}_-\bm{g}_0=\frac{1}{2}{\bm \varphi}_0+\Kii_-^\top{\bm \varphi}_0 .
\]
We derive from the last two relations that $\firstReviewer{\bm{\Upsilon}}_-\bm{g}_0=0$, which, in the light of the coercivity properties established in Proposition \ref{prop:Y:complexificated} implies that $\bm{g}_0=0$. This, in turn, implies that the Cauchy data of ${\bf w}_-$ vanishes on $\Gamma$, and thus ${\bf w}_-=0$ in $\Omega_-$. Consequently, $ \bm{\varphi}_0=0$, which completes the proof.
\end{proof}

\firstReviewer{Regarding $\Sii_+\bm{\lambda}$, i.e.,}  the exterior Robin-to-Robin problem, we can derive a similar formulation as in \eqref{eq:RT_formulation} proceeding in the same way, namely 
 
 \begin{equation}\label{eq:RT_formulation_p} 
  {\cal A}_+\begin{bmatrix}\gamma {\bf u}_-\\ T{\bf u}_-\end{bmatrix}:=  \begin{bmatrix}
  \frac12 \firstReviewer{\bm{I}}-\Kii_+   &      \Vii_+\\ 
  \firstReviewer{\bm{\Upsilon}}_+  + \Wii_+&    \frac12 \firstReviewer{\bm{I}}-\Kii_+^\top  \end{bmatrix} \begin{bmatrix}\gamma {\bf u}_-\\ T{\bf u}_-\end{bmatrix} = 
 \begin{bmatrix} 0\\\bm{\lambda}_+\end{bmatrix}.
\end{equation}

\begin{theorem} 
The matrix operator $\mathcal{A}_+: \Hiib^{s}\times\Hiib^{s-1}\to \Hiib^{s}\times\Hiib^{s-1}$ is a compact perturbation of 
\[
 \mathrm{PS} ({\cal A}_+) =  \begin{bmatrix}
                             \frac12 \firstReviewer{\bm{I}} -\alpha_+\bm{H} & \beta_+\bm{\Lambda} \\
                                 \firstReviewer{\bm{\Upsilon}}_+-\otherCorrections{\delta_+}\bm{\Lambda}^{-1}   &  \frac12 \firstReviewer{\bm{I}} -\alpha_+\bm{H}&\\
                           \end{bmatrix}
\]
which is a coercive operator: There exists $c>0$ such \firstReviewer{that}  
\[
 \Re [  \mathrm{PS} ({\cal A}_+)(\bm{g},\bm{\psi}),\overline{(\bm{g}, \bm{\psi})}]\ge c \left[
 \|\bm{g}\|_{\Hiib^{1/2}}+ \|\bm{\varphi}\|_{\Hiib^{-1/2}}
 \right],\qquad (\bm{g},\bm{\psi})^\top\ne 0. 
\]
 Furthermore, if $\omega^2$ is not an eigenvalue of the Dirichlet Navier operator with material parameters $(\mu_+,\lambda_+)$ in the domain $ \Omega_-$, the operator $\mathcal{A}_+$ is invertible. 
\end{theorem}
\begin{proof} The first part of the theorem follows along the same lines as that in Theorem~\ref{thm:form_1}. Let ${({\bm g_0},{\bm \varphi_0})}\in \mathop{\rm Ker}(\mathcal{A}_+)$ and define the field
\[
{\bf w}(\x):=-(\DLii_+  \bm{g}_0)(\x)+(\SLii_+\bm{\varphi}_0)(\x),\quad \x\in\mathbb{R}^2\setminus\Gamma
\]
and denote \otherCorrections{${\bf w}_\pm ={\bf w}|_{\Omega_\pm}$}.
It follows immediately that $\otherCorrections{\gamma{\bf w}_-}  ={\bf 0}$ and $ \Tiip {\bf w}   =\firstReviewer{\bm{\Upsilon}}_+\bm{g}_0$. Given the jump relations of Navier layer potentials it follows that $\otherCorrections{\gamma{\bf w}_+}=-\bm{g}_0$ and thus
\[
 \Tiip {\bf w}_+   +  \firstReviewer{\bm{\Upsilon}}_+ \gamma{\bf w}_+  ={\bf 0}\qquad{\rm on}\ \Gamma.
\]
Now, $\otherCorrections{{\bf w}_+}$ is a  \firstReviewer{radiating} solution of Navier equations in $\Omega_+$, and given that $\Im{\langle\firstReviewer{\bm{\Upsilon}}_+{\bm {\varphi}},{\bm{\varphi}}\rangle}>0$, 
it follows that $\otherCorrections{{\bf w}_+}=0$ in $\otherCorrections{{\bf w}_+}$, and hence $\bm{g}_0={\bf 0}$. Under the  assumption in the theorem, we also have that ${{\bf w}|_{\Omega_-}}={\bf 0}$, and thus 
$\bm{\varphi}_0=\otherCorrections{\Tiim {\bf w}_-}={\bf 0}$ as well.

\end{proof}

%

\firstReviewer{The restriction on $\omega$ not being an eigenvalue of the Dirichlet problem can be however overcome, that is,  } it is possible to derive BIE formulations to solve exterior Navier equations with generalized Robin boundary conditions that are uniquely solvable for all frequencies $\omega$. Indeed, let \otherCorrections{cf. \eqref{eq:PSkappa:a}}
\begin{equation}\label{eq:regeps}
\otherCorrections{ \widetilde{\Vii} = -\frac{1}{\delta_-}\bm{\Lambda}_\kappa \left(\frac{1}2 \bm{I} -\alpha_-\bm{H}\right)
 =
\left(  {\rm PS}_\kappa (\Yiim)\right)^{-1}},
      \end{equation} 
 and consider alternative formulations obtained by adding the regularized boundary condition $\varepsilon \widetilde{\Vii}(T{\bf u}_+)+\varepsilon \widetilde{\Vii}\firstReviewer{\bm{\Upsilon}}_+{\bf u}_+=\varepsilon \widetilde{\Vii}\bm{\lambda}_+$ to the second equation in formulation~\eqref{eq:RT_formulation_p}   
 \begin{equation}\label{eq:RT_formulation_pe}
 \mathcal{A}_+^{\varepsilon}:= \begin{bmatrix}
  \frac12 \firstReviewer{\bm{I}}-\Kii_+  + \varepsilon \widetilde{\Vii}\firstReviewer{\bm{\Upsilon}}_+  &      \Vii_++\varepsilon \widetilde{\Vii}\\ 
  \firstReviewer{\bm{\Upsilon}}_+  + \Wii_+&    \frac12 \firstReviewer{\bm{I}}-\Kii_+^\top  \end{bmatrix} \begin{bmatrix}\gamma {\bf u}_-\\ T{\bf u}_-\end{bmatrix},\qquad
   \mathcal{A}_+^{\varepsilon} \begin{bmatrix}\gamma {\bf u}_-\\ T{\bf u}_-\end{bmatrix}  =\begin{bmatrix}\varepsilon \widetilde{\Vii}{\bm{\lambda}_+} \\\bm{\lambda}_+\end{bmatrix}.
\end{equation}
We establish

\begin{theorem} \label{thm:form_2}The matrix operator $\mathcal{A}_+^\varepsilon:\Hiib^{s}\times\Hiib^{s-1}\to \Hiib^{s}\times\Hiib^{s-1}$ in the left hand side of equation~\eqref{eq:RT_formulation_pe} is a compact perturbation of a coercive operator in the space $\Hiib^{1/2}   \times\Hiib^{-1/2} $ provided $\varepsilon>0$ is small enough. Furthermore, under the same assumption, the operator $\mathcal{A}_+^\varepsilon$ is invertible. 
\end{theorem}
\begin{proof}
Clearly
\[
  \mathrm{PS} ({\cal A}_+^{\varepsilon})=  \mathrm{PS} ({\cal A}_+)  + {\varepsilon} \begin{bmatrix}
                              \widetilde{\Vii}\firstReviewer{\bm{\Upsilon}}_+  &        \widetilde{\Vii}\\ 
                             0&0
                           \end{bmatrix}
\]
is coervice, in the sense we have stated in previous results, for $\varepsilon$ small enough since so is $\mathrm{PS} ({\cal A}_+)$. 

Now let $ (\bm{g}_0,\bm{\varphi}_0) \in \mathop{\rm Ker}(\mathcal{A}_+^\varepsilon)$ and define the field
\[
{\bf w}(\x):=-(\DLii_+ \bm{g}_0)(\x)+(\SLii_+\bm{\varphi}_0)(\x),\quad \x\in\mathbb{R}^2\setminus\Gamma,\qquad
\otherCorrections{{\bf w}_\pm = {\bf w}|_{\Omega_\pm}}.
\]
It follows immediately that $\gamma{\bf w}_- =-\varepsilon \widetilde{\Vii}(\firstReviewer{\bm{\Upsilon}}_+\bm{g}_0+\bm{\varphi}_0)$ and $\Tiim{\bf w}_-=\firstReviewer{\bm{\Upsilon}}_+\bm{g}_0+\bm{\varphi}_0$. 
Therefore ${\bf w}_-$ is a solution of Navier equations in $\Omega_-$ with generalized Robin boundary conditions 
\[
\Tiim{\bf w}_- +(\varepsilon \widetilde{\Vii})^{-1}  \gamma{\bf w}_-=0, 
\]
and therefore 
\[
 \Im\langle \Tiim{\bf w}_- ,\overline{\gamma {\bf w}}_-\rangle + \Im\langle (\varepsilon \widetilde{\Vii})^{-1}{\gamma{\bf w}_-},\overline{\gamma{\bf w}}_-\rangle=0\firstReviewer{.}
\]
\otherCorrections{This implies that ${\bf w}=0$, since  by Lemma \ref{lemma:im:Tpm} and Proposition \ref{prop:Y:complexificated} we have
\[
 \Im \langle \Tiim{\bf w} ,\overline{\gamma {\bf w}}_-\rangle =0,\quad 
 -\Im \langle (\varepsilon \widetilde{\Vii})^{-1}{\gamma{\bf w}_-},\overline{\gamma{\bf w}_-}\rangle
 = -\Im \langle  \varepsilon^{-1}    {\rm PS}_\kappa (\Yiim)   {\gamma{\bf w}_-},\overline{\gamma{\bf w}_-}\rangle
>c \|{\gamma{\bf w}_-}\|^2_{\Hiib^{1/2}}.
\] 
Therefore, 
\[
\bm{\varphi}_0=-\firstReviewer{\bm{\Upsilon}}_+\bm{g}_0. 
\]
}
The jump conditions of Navier layer potentials yield $ \gamma{\bf w}_+ =-\bm{g}_0$ and $\Tiip {\bf w} =-\bm{\varphi}_0=\firstReviewer{\bm{\Upsilon}}_+\bm{g}_0$. Consequently, ${\bf w}_+$ is a  \firstReviewer{radiating} solution of Navier equations in $\Omega_+$ with generalized Robin boundary conditions 
\[
\Tiip {\bf w}_+ + \firstReviewer{\bm{\Upsilon}}_+ \gamma {\bf w}_+  =0.
\]
Thus,  from Lemma \ref{lemma:im:Tpm}
\[
 0= \Im\langle  \Tiip_+ {\bf w}_+  ,\overline{\gamma{\bf w}_+}\rangle +\Im{\langle\firstReviewer{\bm{\Upsilon}}_+  \gamma{\bf w}_+ ,\overline{  \gamma{\bf w}_+ }\rangle}\ge
 -\Im{\langle {\rm PS}_\kappa (\Yiim) \gamma{\bf w}_+ ,\overline{  \gamma{\bf w}_+ }\rangle} 
\]
from where we infer, Proposition \ref{prop:Y:complexificated}, that $\gamma{\bf w}_+=0$  and thus ${\bm \varphi}_0=\bm{g}_0=0$.
\end{proof}

 Alternatively, instead of working with a full system of BIEs , it is also possible to construct a  single  robust regularized BIE formulations for the solution of exterior Navier equations with generalized Robin boundary conditions. Defining first the operator
\[
\firstReviewer{\bm{R}^{\rm os}}:=({\rm PS}_\kappa(\Yiip)-{\rm PS}_\kappa(\Yiim))^{-1}
\]
we look for a  \firstReviewer{radiating} solution of the Navier equation in $\Omega_+$ of the form
\begin{equation}\label{eq:reg_ext}
{\bf u}_+(\x):=(\DLii_+ [\firstReviewer{\bm{R}^{\rm os}}{\bm \varphi]})(\x)-(\SLii_+[{\rm PS}_\kappa(\Yiip)\firstReviewer{\bm{R}^{\rm os}}{\bm \varphi}])(\x),\quad \x\in\Omega_+.
\end{equation}
We have
\begin{eqnarray*}
\gamma{\bf u}_+&=&\left[\left(\frac{1}{2}+\Kii_+-\Vii_+{\rm PS}_\kappa(\Yiip)\right)\firstReviewer{\bm{R}^{\rm os}}\right]{\bm \varphi}\\
 \Tiip{\bf u}_+ &=&\left[\left(\Wii_++\frac{1}{2}{\rm PS}_\kappa(\Yiip)-\Kii^\top_+{\rm PS}_\kappa(\Yiip)\right)\firstReviewer{\bm{R}^{\rm os}}\right]{\bm \varphi}.
\end{eqnarray*}
It follows than that ${\bf u}_+$ satisfies the boundary condition $T{\bf u}_++\firstReviewer{\bm{\Upsilon}}_+{\bf u}_+=T{\bf u}_+-{\rm PS}_\kappa(\Yiim){\bf u}_+=\bm{\lambda}_+$  if and only if  ${\bm \varphi}$ satisfies the BIE
\begin{equation}
\begin{aligned}
&\firstReviewer{\bm{B}}_+{\bm \varphi}=\bm{\lambda}_+,\\
&\firstReviewer{\bm{B}}_+:=\left[\left(\Wii_++\frac{1}{2}{\rm PS}_\kappa(\Yiip)-\Kii^\top_+{\rm PS}_\kappa(\Yiip)\right)-{\rm PS}_\kappa(\Yiim)\left(\frac{1}{2}\firstReviewer{\bm{I}} +\Kii_+-\Vii_+{\rm PS}_\kappa(\Yiip)\right)\right]\firstReviewer{\bm{R}^{\rm os}}.
\end{aligned}
\end{equation}

\begin{theorem} \label{thm:form_3}The operator $\firstReviewer{\bm{B}}_+:\Hiib^s\to \Hiib^s$ is a compact perturbation of the identity. Furthermore, the operator $\firstReviewer{\bm{B}}_+$ is invertible. 
\end{theorem}
\begin{proof}
Notice that
\begin{equation}
 \label{eq:R-1}
\firstReviewer{\bm{R}}^{-1}  = - \left(\frac{\beta_++\beta_-}{2\beta_+\beta_-}\firstReviewer{\bm{I}} -
\frac{\alpha_+\beta_--\alpha_-\beta_+}{\beta_-\beta_+}\bm{H}\right) \bm{\Lambda}^{-1}_\kappa, 
\end{equation}
that 
\[
\firstReviewer{\bm{B}}_+=
\firstReviewer{\bm{B}}_{+,0} +\firstReviewer{\bm{K}}
\]
with ${\cal K }:\Hiib^s\to \Hiib^{s+2}$ 
\begin{eqnarray*}
\firstReviewer{\bm{B}}_{+,0} &=&\bigg[\delta_+\Lambda_\kappa^{-1}-
 \left(\tfrac12\firstReviewer{\bm{I}} -\alpha_+\bm{H}\right) \frac{1}{\beta_+}\bm{\Lambda}_\kappa^{-1}\left(\tfrac12\firstReviewer{\bm{I}} -\alpha_+\bm{H}\right)\\
 && -\frac{1}{\beta_-} \bm{\Lambda}_\kappa^{-1}(\tfrac12 \firstReviewer{\bm{I}} +\alpha_-\bm{H})
 \left(\tfrac12\firstReviewer{\bm{I}}  +\alpha_+\bm{H}-\beta_+\bm{\Lambda}_\kappa \left(-\tfrac{1}{\beta_+}\right)\bm{\Lambda}_\kappa^{-1}
 \left(\tfrac12\firstReviewer{\bm{I}}  -\alpha_+\bm{H}\right) 
 \right)\bigg]\firstReviewer{\bm{R}}\\
 &=& \Big[\left(\delta_+-\tfrac{1}{\beta_+}\left(\tfrac14-\alpha_+^2\right)-\tfrac{1}{2\beta_-}
 \right) \firstReviewer{\bm{I}}  +\left(\tfrac{\alpha_+}{\beta_+}-\tfrac{\alpha_-}{\beta_-}\right)\bm{H}\Big]\bm{\Lambda}_\kappa^{-1}\bigg]\firstReviewer{\bm{R}}
\end{eqnarray*}
But since
\[
 \delta_+-\frac{1}{\beta_+}\left(\frac14-\alpha_+^2\right)=-\frac{1}{2\beta_+}  
\]
we can conclude 
\[
\firstReviewer{\bm{B}}_{+,0} = - \Big[\left(\frac{\beta_++\beta_-}{2\beta_-\beta_+}
 \right)\firstReviewer{\bm{I}}  -\left(\frac{\alpha_+\beta_--\alpha_-\beta_+}{\beta_+\beta_-}\right)\bm{H}\Big]\bm{\Lambda}_\kappa^{-1}\firstReviewer{\bm{R}} = \firstReviewer{\bm{I}} 
\]
in view of \eqref{eq:R-1}.
%
%
%
%
 Clearly, the invertibility of the operator $\firstReviewer{\bm{B}}_+$ is then equivalent to its injectivity. Let ${\bm \varphi}_0\in \mathop{\rm Ker}(\firstReviewer{\bm{B}}_+)$ and define the field
 \begin{equation}\label{eq:reg_ext0}
 {\bf w}(\x):=(\DLii_+ [\firstReviewer{\bm{R}^{\rm os}}{\bm \varphi_0]})(\x)-(\SLii_+[{\rm PS}_\kappa(\Yiip)\firstReviewer{\bm{R}^{\rm os}}{\bm \varphi_0}])(\x),\quad \x\in\mathbb{R}^2\setminus\Gamma, \quad\otherCorrections{{\bf w}_\pm ={\bf w}|_{\Omega_\pm}}.
 \end{equation}
 It follows that $ {\bf w}|_{\Omega_+}$ is a  \firstReviewer{radiating} solution of the Navier equations in $\Omega_+$ with generalized Robin conditions $\Tiip {\bf w} +\firstReviewer{\bm{\Upsilon}}_+ \gamma{\bf w}_+=0$ on $\Gamma$, and thus $ {\bf w}_{+}=0$ in $\Omega_+$. Given the jump properties of Navier layer potentials, we obtain
 \[
\gamma{\bf w} =-\firstReviewer{\bm{R}^{\rm os}}{\bm \varphi_0}\qquad \Tiim{\bf w} =-{\rm PS}_\kappa(\Yiip)\firstReviewer{\bm{R}^{\rm os}}{\bm \varphi_0}=\firstReviewer{\bm{\Upsilon}}_-\firstReviewer{\bm{R}^{\rm os}}{\bm \varphi_0}.
 \]
 We have that $ {\bf w}|_{\Omega_-}$ is a solution of the Navier equation with material parameters $(\mu_+,\lambda_+)$ in $\Omega$ that satisfies the generalized Robin boundary condition on $\Gamma$
 \[
 \Tiim{\bf w} +\firstReviewer{\bm{\Upsilon}}_- \otherCorrections{\gamma{\bf w}_-}  =0.
 \]
 We derive then that $ {\bf w}_-={\bf w}|_{\Omega_-}=0$ in $\Omega$ which in turn implies that $\firstReviewer{\bm{R}^{\rm os}}{\bm \varphi_0}=0$ on $\Gamma$. Given that the operator $\firstReviewer{\bm{R}^{\rm os}}$ is invertible, we get that ${\bm \varphi_0}=0$, which was to be proved.
\end{proof}

Having show some different  formulations to 
\firstReviewer{evaluate $\Sii_\pm$},  we are now in the position to prove the main result concerning Schwarz iteration operators:

%

  \begin{theorem} \label{thm:SpSm} The RtR operators $\firstReviewer{\bm{S}}_\pm:\Hiib^{s}\to \Hiib^{s+2}$ are continuous, and so compact when viewed acting in $\Hiib^{s}$ into itself. Furthermore, the operator $I-\firstReviewer{\bm{S}}_-\firstReviewer{\bm{S}}_+:\Hiib^s\to \Hiib^{s}$ is invertible with continuous inverse. As a consequence, the underlying Schwarz iteration operator in the left hand side of equation~\eqref{eq:DDM_rtr} is a invertible compact perturbation of the identity in $\Hiib^s\times\Hiib^s$.   
   \end{theorem}
  \begin{proof}
  Notice that $I-\firstReviewer{\bm{S}}_-\firstReviewer{\bm{S}}_+$ is the Schur complement of the matrix operator in the left hand side of equation~\eqref{eq:DDM_rtr}. Hence, the last result stated in the theorem is a consequence of the previous statements. Let us recall that in view of~\eqref{eq:opt_choice} we have
  \[
  \firstReviewer{\bm{S}}_{\pm}{\bf f}  = T_\pm {\bf u}_\pm- {\firstReviewer{\bm{\Upsilon}}}_\pm   \gamma_\pm {\bf u}_\pm, 
  \]
  where   ${\bf u}_\pm$ are the solutions of~\eqref{eq:01:prop:OSM:01}.  
The following estimate basically follows from the definition of the operators involved 
  \[ 
\| T_\pm {\bf u}_\pm -  \firstReviewer{\bm{\Upsilon}}_\pm  \gamma_\pm {\bf u}_\pm \|_{\Hiib^{s+1}}\le C
\|   \gamma_\pm {\bf u}_\pm\|_{\Hiib^{s}}    
  \]
  for some $C>0$ independent of ${\bf u}_\pm$. 
Besides, from \eqref{eq:RT_formulation} and \eqref{eq:RT_formulation_p} (or \eqref{eq:RT_formulation_pe}) we derive the estimate
\[
 \|\gamma_\pm {\bf u}_\pm\|_{\Hiib^{s}}\le C '\|\bm{\lambda}_\pm\|_{\Hiib^{s-1}}
\]
with $C'>0$ independent of $\bm{g}$. Gathering these properties we can conclude that $\firstReviewer{\bm{S}}_\pm:\Hiib^{s}\to \Hiib^{s+2}$ is continuous.

  In order to establish the invertibility of the operator $\firstReviewer{\bm{I}}- \firstReviewer{\bm{S}}_-\firstReviewer{\bm{S}}_+$ it suffices to prove its injectivity. Assume that ${\bm \varphi}$ such as ${\bm \varphi} = \firstReviewer{\bm{S}}_-\firstReviewer{\bm{S}}_+{\bm \varphi}$ and solve the Navier problem in the exterior domain $\Omega_+$
  \begin{equation}
   \label{eq:el:del:medio:de:los:chichos}
\begin{aligned}
  {\nabla}\cdot\bm{\sigma}_+( {\bf w}_+)+\omega^2 {\bf w}_+&=&0  &\quad{\rm in}\ \Omega_+,\\
 \Tiip {\bf w}_+  + \firstReviewer{\bm{\Upsilon}}_+ \otherCorrections{\gamma {\bf w}_+}  &=&{\bm \varphi}&\quad{\rm on}\ \Gamma
\end{aligned}
\end{equation}
where $ {\bf w}_+$ is  \firstReviewer{radiating} at infinity, so that
\begin{equation*}
  \firstReviewer{\bm{S}}_+{\bm \varphi}=\left(\Tiip {\bf w}_+ +\firstReviewer{\bm{\Upsilon}}_-  \otherCorrections{\gamma {\bf w}_+}\right) .
\end{equation*}
We also consider the Navier problem in the bounded domain $\Omega$
\begin{eqnarray*}
  {\nabla}\cdot\bm{\sigma}_-({\bf w}_- )+\omega^2{\bf w}_- &=&0\quad{\rm in}\  \Omega_-\nonumber\\
  \Tiim{\bf w}_-+\firstReviewer{\bm{\Upsilon}}_- \otherCorrections{\gamma {\bf w}_-} &=&\firstReviewer{\bm{S}}_+{\bm \varphi}\quad{\rm on}\ \Gamma
\end{eqnarray*}
so that
\begin{equation}
  \firstReviewer{\bm{S}}_-\firstReviewer{\bm{S}}_+{\bm \varphi}=\left(\Tiim{\bf w}_-+\firstReviewer{\bm{\Upsilon}}_+ \gamma {\bf w}_- \right) .
\end{equation}
Given that ${\bm \varphi} = \firstReviewer{\bm{S}}_-\firstReviewer{\bm{S}}_+{\bm \varphi}$ we obtain
\begin{equation}\label{eq:identity_1}
\Tiip {\bf w}_+ +\firstReviewer{\bm{\Upsilon}}_+ \otherCorrections{\gamma} {\bf w}_+= 
\Tiim{\bf w}_-+\firstReviewer{\bm{\Upsilon}}_+ \gamma {\bf w}_-.
\end{equation}
Furthermore, the boundary condition of ${\mathbf{w}_-}$ on $\Gamma$ leads to
\begin{equation}\label{eq:identity_2}
\Tiip {\bf w}_+ +\firstReviewer{\bm{\Upsilon}}_-  \otherCorrections{\gamma}{\bf w}=\Tiim{\bf w}_-+\firstReviewer{\bm{\Upsilon}}_- \gamma{\bf w}_-.
\end{equation}
We note that it follows from equations~\eqref{eq:identity_1} and~\eqref{eq:identity_2} that
\[
(\firstReviewer{\bm{\Upsilon}}_+-\firstReviewer{\bm{\Upsilon}}_-)[\otherCorrections{\gamma}{\bf w}_+-\otherCorrections{\gamma{\bf w}_-}]=0.
\]
which since the operator $\firstReviewer{\bm{\Upsilon}}_+-\firstReviewer{\bm{\Upsilon}}_-$ is injective, in turn implies 
\[
 \otherCorrections{\gamma}{\bf w}_+ = \gamma  {\bf w}_-   
\]
and then, from \eqref{eq:identity_1} , 
\[
 \Tiip{\bf w}_+   =\Tiim{\bf w}_-
\]
which implies 
\[
\Im\langle \Tiip{\bf w}_+  ,\overline{\gamma {\bf w}}_+ \rangle =\Im \langle    \Tiim{\bf w}_-,\overline{\gamma {\bf w}}_- \rangle=0
\]
and hence, by Lemma \ref{lemma:im:Tpm}, $ \otherCorrections{\gamma}{\bf w}_+ =0=\Tiip{\bf w}_+$ which yields, via \eqref{eq:el:del:medio:de:los:chichos}, that $\bm{\varphi}={\bf 0}$
  \end{proof}

Having presented BIE formulations for the various elastodynamic scattering and transmission problems considered in this paper, we turn next to describing Nystr\"om discretizations of those BIEs.

\section{Numerical experiments}
\label{Nd}

\subsection{Nystr\"om discretizations}

We use Nystr\"om discretizations for the numerical solution of the various elastodynamic BIE that rely on global trigonometric interpolation with $2n$ nodes
\[
t_j=\frac{j\pi}{n},\ j=0,1,\ldots,2n-1
\]
onto the space of trigonometric polynomials
\[
\firstReviewer{\mathbb{T}_n}=\left\{\varphi(t)=\sum_{m=0}^na_m\cos{mt}+\sum_{m=1}^{n-1}b_m\sin{mt}\ :{\ a_m,b_m\in\mathbb{C}}\right\}.
\]
Our BIE formulations {works with}  two types of operators: (a) BIO and (b) Fourier multipliers. The latter are discretized in a straightforward manner using trigonometric interpolation and FFTs. With regards to (a), the majority of the kernels of the BIOs discussed in this paper exhibit a logarithmic singularity which is resolved via the Kusmaul-Martensen quadrature~\cite{kusmaul,martensen}
\[
\frac{1}{2\pi}\int_0^{2\pi}H({\tau},{t})\log\left(4\sin^2\frac{\tau-t}{2}\right)\varphi({t})\,{\rm d}t\approx\sum_{m=0}^{2n-1}\firstReviewer{\bm{R}}_m({\tau})H({\tau},{t}_m)\varphi({t}_m)
\]
where
\[
\firstReviewer{\bm{R}}_m({\tau}):=-\frac{1}{n}\sum_{j=1}^{n-1}\frac{1}{j}\cos{j({\tau}-t_m)}-\frac{1}{2n^2}\cos{n({\tau}-t_m)},\ 0\leq m\leq 2n-1.
\]
This turns to be equivalent to consider the approximation for the underlying integral operator given by
\[
\frac{1}{2\pi}\int_0^{2\pi}\log\left(4\sin^2\frac{{t}-{\tau}}{2}\right) P_n(H({\tau},{\cdot})\varphi) (t)\,{\rm d}t
\]
$P_n:C[0,2\pi]\to \firstReviewer{\mathbb{T}_n}$ appearing above is the trigonometric interpolation operator,
since this integral can be computed analytically.

In the case of smooth integrands in the definition of BIE, we use the trapezoidal rule for their discretizations. The hypersingular part of the operator $\Wii$, on the other hand, is treated via the Kress quadrature~\cite{KressH} for the evaluation of $2\pi$ periodic Hilbert transforms 
\[
\frac{1}{2\pi}{\mathrm{p.v.}\,}\int_0^{2\pi}\cot{\frac{{t-\tau}}{2}}\varphi'({t}){\rm d}t\approx\sum_{m=0}^{2n-1}T_m({\tau})\varphi(t_m)
\]
where
\[
T_m({\tau}):=-\frac{1}{n}\sum_{j=1}^{n-1}j\cos{j({\tau}-t_m)}-\frac{1}{2}\cos{n({\tau}-t_m)},\ 0\leq m\leq 2n-1.
\]
Finally, the discretization of the Cauchy {principal value} integral that enters the definition of the principal part of the double layer operators $\Kii$ and $\Kii^\top$, requires more care. Indeed, as discussed in~\cite{Kress}, the use of trigonometric interpolation in the evaluation of Cauchy {principal value} integrals
\[
\frac{1}{2\pi}{\mathrm{p.v.}\,} \int_0^{2\pi}\cot{\frac{{t}-{\tau}}{2}}\varphi(t){\rm d}t\approx\frac{1}{2\pi}{\mathrm{p.v.}\,}\int_0^{2\pi}\cot{\frac{{t}-{\tau}}{2}}[P_n\varphi](t){\rm d}t
\]
leads to 
\[
\frac{1}{2\pi}{\mathrm{p.v.}\,}\int_0^{2\pi}\cot{\frac{{t}-{\tau}}{2}}\varphi({t}){\rm d}t\approx\sum_{m=0}^{2n-1}Q_m(t)\varphi(t_m)
\]
with
\[
Q_m(\tau):=\frac{1}{n}\sum_{j=1}^{n-1} \sin{j(\tau-t_m)}+\frac{1}{2n}\sin{n(\tau-t_m)},\ 0\leq m\leq 2n-1.
\]
However, cf. \eqref{eq:Lambda1}, 
\[
 \frac{1}{2\pi}{\mathrm{p.v.}\,}\int_0^{2\pi}\cot{\frac{{t}-{\tau}}{2}} \cos nt \,{\rm d}t = \sin n\tau
\]
and thus the Nystr\"om discretization of the Cauchy {principal value} operator is not a bijection from $\firstReviewer{\mathcal{T}_n}$ to itself. For these reasons we use the half grid size shifted quadrature method~\cite{Kress}
\[
\frac{1}{2\pi}{\mathrm{p.v.}\,}\int_0^{2\pi}\cot{\frac{ t -\tau}{2}}\varphi({t}){\rm d}t\approx \frac{\pi}{n}\sum_{m=0}^{2n-1}\cot \frac{  t_{m+\frac{1}2}-\tau}{2}\varphi\left( t_{m+\frac{1}2} \right),\quad  t_{m+\frac{1}2} := t_m+\frac{\pi}{2n} .
\]
The evaluation of the density function at the shifted grid points $ t_{m+\frac{1}{2}} $ is readily achieved via trigonometric interpolation. Given that the Nystr\"om discretizations of the principal parts of the elastodynamic operators are available in the literature, the main difficulty of the overall collocation schemes resides in the logarithmic splitting of various kernels. We present in what follows numerical results concerning the various formulations considered in this paper.

\subsection{Numerical results}

In this section we present various results concerning the accuracy of elastodynamic BIE solvers based on Nystr\"om discretizations as well as the rates of convergence of iterative solvers (e.g. GMRES~\cite{SaadSchultz}) for the solution of the linear systems ensuing from the various BIE formulations considered in this text. With regards to the accuracy of our discretizations, we present results concerning far field data. For a scattered elastic field ${\bf u}$ we can define the associated longitudinal wave ${\bf u}_p$ and the transversal wave ${\bf u}_s$  cf. \eqref{eq:up:us}.
The Kupradze radiation conditions~\cite{KupGeBa:1979} take on the form
\[
  {\bf u}_p(\x)=\frac{e^{ik_p|\x|}}{\sqrt{|\x|}}\left({\bf u}_{p,\infty}(\hat{{\x}})+\mathcal{O}\left(\frac{1}{|\x|}\right)\right)\qquad
  {\bf u}_s(\x)=\frac{e^{ik_s|\x|}}{\sqrt{|\x|}}\left({\bf u}_{s,\infty}(\hat{{\x}})+\mathcal{O}\left(\frac{1}{|\x|}\right)\right)
  \]
  as $|\x|\to\infty$ where $\hat{{\x}}=\x/|\x|$. We present in this section the maximum errors $\varepsilon_\infty$ achieved when computing the quantities ${\bf u}_{p,\infty}$ and ${\bf u}_{s,\infty}$ evaluated at fine enough meshes \secondReviewer{(1024  equidistributed points is typically used)} on the unit circle $|\hat{{\x}}|=1$. We also report the size $n$ of the $2n$ dimensional trigonometric polynomial space $\firstReviewer{\mathbb{T}_n}$. Therefore, in the impenetrable case, the size of the linear systems Nystr\"om discretization linear systems of the various (vectorial) BIE is $(4n)\times(4n)$; in the penetrable case, the size increases to $(8n)\times(8n)$. We consider two types of incident fields: (a) elastodynamic point sources (in the method of manufactured solutions); and (b) plane waves. In case (b) the far field errors are computed with respect to reference solutions produced through 
  \secondReviewer{the next level of discretization of the BIEs (that is, the computed values for $2n_{\rm max}$ where $n_{\rm max}$ is highest value of $n$ shown in the numerical results, are taken as {\em exact} solutions.)}. 
  
  We start with an illustration of the accuracy that can be achieved in the far field when the four elastodynamic BIOs are employed in the solution of elastodynamic BIE.   
  
  The domains considered in this numerical section are: (i) the unit circle centered at origin; (ii) the starfish obstacle~\cite{hao2014high} whose paramaterization is given by
    \[
    \x(t)=\left(1+\tfrac{1}{4}\sin{5t}\right)(\cos{t}, \sin{t}),\ 0\leq t<2\pi; 
    \]
    (iii) the cavity-like geometry  given by
    \begin{eqnarray*}
     \x(t) &=& (x_1(t),x_2(t)),\qquad \ 0\leq t<2\pi; \\
    x_1(t)&=& \tfrac{2}{5}(\cos(t)+2\cos(2t)),\\
     x_2(t)&=& \tfrac12\sin(t)+\tfrac12\sin(2t)+\tfrac14\sin(3t)+\tfrac1{48}( 4\sin(t)-7\sin(2t)+6\sin(3t)-2\sin(4t)).
    \end{eqnarray*}
We depict in Figure \ref{fig:geom} a sketch of these domains. 
Notice that the diameters of all of the scatterers considered in this paper are equal to 2. 

  \subsubsection{The method of manufactured solutions}

  The method of manufactured solutions is a reliable test of the accuracy of PDE solvers. In the Dirichlet case, we consider the time-harmonic Navier equation with boundary value data produced by a point source $\x_0$ placed inside the scatter $ \Omega_-$
  \[
    {\bf f}_\Gamma (\x)=\Phi(\x,\x_0),\quad \x\in\Gamma,\ \x_0\in \Omega_-.
    \]
    In this case, the  \firstReviewer{radiating} solution of the Dirichlet scattering problem in the exterior domain $ {\Omega_+}$ is the point source itself, that is ${\bf u}(\x)=\Phi(\x,\x_0)$ for all $\x\in {\Omega_+}$. We look for the solution ${\bf u}$ in the form of either a single or double layer potential
    \[
    {\bf u}(\x)=( \SLii \bm{\varphi} )(\x)\quad {\rm or}\quad{\bf u}(\x)=( \DLii  \bm{g} )(\x),\ \x\in {\Omega_+}
    \]
    and we solve the corresponding BIEs   (with, as usual, ${\bf f} ={\bf f}_\Gamma\circ{\bf x}$) 
    \[
     \Vii \bm{\varphi} ={\bf f}\quad{\rm and}\quad \frac{1}{2}\bm{g}+ \Kii\bm{g}={\bf f}\firstReviewer{.}
    \]
    Assuming that the frequency $\omega$ is selected such as both BIE formulations above are uniquely solvable, we use Nystr\"om discretizations to solve numerically these BIEs and we compare in the far-field the computed solutions with the exact point source solution. We repeat the experiments in the case of Neumann boundary conditions, that is we choose a boundary data in the form
\[
    \bm{\lambda}_\Gamma(\x)=T_\Gamma \Phi(\x,\x_0),\quad \x\in\Gamma,\ \x_0\in \Omega_-
    \]
    and we solve the BIEs
   \[
    -\frac{1}{2} \bm{\varphi} + \Kii^\top\bm{\varphi} =\bm\lambda\quad{\rm and}\quad  \Wii\bm{g} =\bm\lambda\firstReviewer{.} 
    \]
    Again here, we compare the numerical solutions in the far field against the exact point source solution. 
    
We report errors corresponding to (a) the Dirichlet single layer formulation (under the heading $\Vii$, given that we used the single layer $\Vii$ BIO to solve the time-harmonic Navier equation), (b) Dirichlet double layer formulation (under the heading $\Kii$), and (c) Neumann double layer formulation (under the heading $\Wii$). Given that the BIO $\Kii^\top$ is the (real) $L^2$ adjoint of the BIO $\Kii$, its discretization produces identical levels of accuracy to those of $\Kii$, and therefore we chose not to present them.

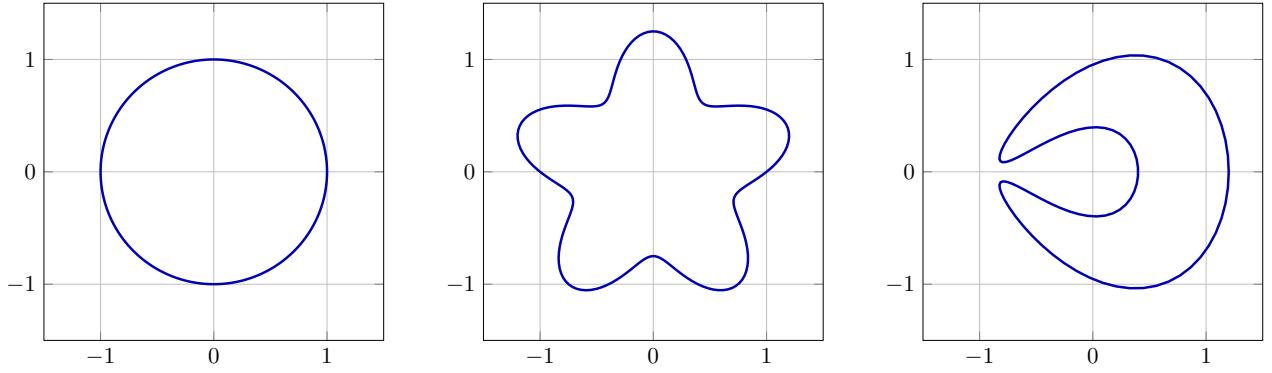
\begin{figure} 
\begin{center}
 \resizebox{0.31\textwidth}{!}{
\begin{tikzpicture} 
[declare function = {f1(\x)=1); }]
\begin{axis}[
  xmin=-1.5 ,xmax=1.5 , 
  ymin=-1.5,ymax=1.5,
  grid=both,  axis equal image,
xtick=    {-2, -1,0, 1,2 },
ytick=    {-2, -1,0, 1,2 },
]
  \addplot[domain=0:360,samples=200,variable=\x,very thick, blue!70!black,data cs=polar](x,{f1(x)});%
\end{axis}
\end{tikzpicture}
}
\quad
\resizebox{0.31\textwidth}{!}{
\begin{tikzpicture}
[declare function = {f1(\x)=1+ 0.25*sin(5*(\x) ); }]
\begin{axis}[
  xmin=-1.5 ,xmax=1.5 , 
  ymin=-1.5,ymax=1.5,
  grid=both,  axis equal image,
xtick=    {-2, -1,0, 1,2 },
ytick=    {-2, -1,0, 1,2 },
]
  \addplot[domain=0:360,samples=200,variable=\x,very thick, blue!70!black,data cs=polar](x,{f1(x)});%
\end{axis}
\end{tikzpicture} 
}
\quad
 \resizebox{0.31\textwidth}{!}{
\begin{tikzpicture}
\begin{axis}[
  xmin=-1.5 ,xmax=1.5 , 
  ymin=-1.5,ymax=1.5,
  grid=both,  axis equal image,
xtick=    {-2, -1,0, 1,2 },
ytick=    {-2, -1,0, 1,2 },
]
  \addplot[domain=0:360,samples=100,variable=\t,very thick, blue!70!black](%
    {(cos( t)+2*cos(2* t))/2.5},
    {(sin( t)+ sin(2*t)+1/2* sin(3*t))/2-(-4* sin(t)+7*sin(2*t)-6*sin(3*t)+2*sin(4*t))/48}%
  ); 
\end{axis}
\end{tikzpicture}
}
\end{center}
\caption{\label{fig:geom} Domains for the experiments: from left to right,  the unit circle,   the starfish domain, and the cavity problem.}
\end{figure}

\begin{table}
   \begin{center}
\begin{tabular}{|c|c|c|c|c|}
\hline
$\omega$ & $n$ &  $\Vii$ & $\Kii$ & $\Wii$ \\
\hline
 & & $\varepsilon_\infty$ & $\varepsilon_\infty$ &  $\varepsilon_\infty$ \\
\hline
16 & 32  & 3.0 $\times$ $10^{-2}$ &9.9$\times$ $10^{-2}$ & 1.3 $\times$ $10^{-1}$\\
16 & 64 & 7.1 $\times$ $10^{-7}$ & 1.1 $\times$ $10^{-3}$ & 1.6 $\times$ $10^{-6}$\\
16 & 128 & 5.5 $\times$ $10^{-14}$ & 4.2 $\times$ $10^{-13}$& 1.6 $\times$ $10^{-9}$\\
\hline
\hline
32 & 64  & 3.0 $\times$ $10^{-2}$ &2.5 $\times$ $10^{-1}$ &3.4 $\times$ $10^{-1}$ \\
32 & 128 & 2.2 $\times$ $10^{-7}$ & 2.4 $\times$ $10^{-4}$ &5.0 $\times$ $10^{-8}$\\
32 & 256 & 1.0 $\times$ $10^{-15}$ & 8.3 $\times$ $10^{-13}$ & 5.4 $\times 10^{-12}$ \\
\hline
\end{tabular}
\caption{Errors in the method of manufactered solution for the smooth starfish geometry for different values of the frequency $\omega$ and parameter values $\lambda=1$, $\mu=1$, at various levels of discretization.\label{comp5}}
\end{center}
\end{table}

\subsubsection{Iterative behavior of BIE formulations in the impenetrable case}

We present in this section various numerical experiments regarding the iterative behavior of the various BIE formulations for the solution of elastodynamic impenetrable scattering problems. We consider incident plane waves of the form
\begin{equation}\label{eq:plwave}
  {\bf u}^{\rm inc}(\x)=\frac{1}{\mu}e^{ik_s\x\cdot{\bm d}}({\bm d}\times{\bm p})\times{\bm d}+\frac{1}{\lambda+2\mu}e^{ik_p\x\cdot{\bm d}}({\bm d}\cdot{\bm p}){\bm d}
\end{equation}
where the direction ${\bm d}$ has unit length $|{\bm d}|=1$. If the vector ${\bm p}$ is chosen such as ${\bm p}=\pm{\bm d}$, the incident plane is a pressure wave or P-wave. In the case when ${\bm p}$ is orthogonal to the direction of propagation ${\bm p}$, the incident plane wave is referred to as a shear wave or S-wave. We considered plane waves of direction ${\bm d}=\begin{bmatrix}0 & -1\end{bmatrix}^\top$ in all of our numerical experiments; in the case of S-wave incidence we selected ${\bm p}=\begin{bmatrix}1 &0\end{bmatrix}^\top$. We observed that other choices of the direction ${\bm d}$ and of the vector ${\bm p}$ lead to qualitatively similar results.

\paragraph{Dirichlet boundary conditions}

We investigated the iterative behavior of Dirichlet integral solvers based on three formulations: (1) the CFIE formulation~\eqref{eq:CFIER_D} with the coupling constant $\eta_D=1$; (2) the CFIE formulation with the optimal coupling constant $\eta_D$ given in equation~\eqref{eq:eta_D}---which we refer to by the acronym ``CFIE $\eta_D$ opt''; and (3) the CFIER formulation~\eqref{eq:CFIER_D} with the choice $\RD={\rm PS}_\kappa(\Yiip)$ with $\kappa=k_s+0.4\ i\ k_s^{1/3}$---similar choices were presented in the literature~\cite{chaillat2020analytical}. We report the number of GMRES iterations required by each of these three BIE formulations to reach relative residuals of $10^{-8}$ and corresponding discretizations that deliver results with accuracies at the $10^{-7}$ level in the case of smooth scatterers. 
As it can be seen from the results presented in Tables~\ref{comp6}-\ref{comp8}, the CFIE formulation with the optimal coupling parameter $\eta_D$ exhibits the best iterative behavior in the high-frequency regime. We note that although the double layer operator $\Kii$ is not compact, all of the three BIE formulations considered in the numerical experiments behave like integral equations of the second kind, that is, the numbers of GMRES iterations required to reach a certain residual do not increase with more refined discretizations. \firstReviewer{ We note that we considered numerical experiments for which the ration $k_s/k_p=2$; qualitatively similar results were observed for other aspect ratios between the shear and pressure wavenumbers.}

\begin{table}
   \begin{center}  
\begin{tabular}{|c|c|c|c|c|}
\hline
$\omega$ & $n$ & \# iter CFIE~\eqref{eq:CFIER_D} $\eta_D=1$ & \# iter CFIE $\eta_D$ opt & \# iter CFIER $\RD={\rm PS}_\kappa(\Yiip)$ \\
\hline
\hline
10 & 64 & 31 & 22 & 21\\
20 & 128 & 50 & 27 & 30\\
40 & 256 & 97 & 28 & 38\\
80 & 512 & 231 & 27 & 39\\
160 & 1024 & 438 & 32 & 43\\
\hline
\end{tabular}
\caption{\firstReviewer{Numbers of GMRES iterations to reach residuals of $10^{-8}$ for various BIE} formulations of Dirichlet elastodynamic scattering problems at high frequencies in the case when $\Gamma$ is a unit circle. The material parameters are $\lambda=2$ and $\mu=1$, and the incidence was S-wave. The discretizations used in these numerical experiments delivered results accurate at the level of $10^{-7}$.\label{comp6}}
\end{center}
\end{table}

\begin{table}
   \begin{center}  
\begin{tabular}{|c|c|c|c|c|}
\hline
$\omega$ & $n$ & \# iter CFIE~\eqref{eq:CFIER_D} $\eta_D=1$ & \# iter CFIE $\eta_D$ opt & \# iter CFIER $\RD={\rm PS}_\kappa(\Yiip)$ \\
\hline
\hline
10 & 64 & 41 & 32 & 26\\
20 & 128 & 85 & 39 & 36\\
40 & 256 & 166 & 46 & 56\\
80 & 512 & 353 & 51 & 76\\
160 & 1024 & 782 & 57 & 100\\
\hline
\end{tabular}
\caption{\firstReviewer{Numbers of GMRES iterations to reach residuals of $10^{-8}$ for various BIE} formulations of Dirichlet elastodynamic scattering problems at high frequencies in the case when $\Gamma$ is the starfish contour. The material parameters are $\lambda=2$ and $\mu=1$, and the incidence was S-wave. The discretizations used in these numerical experiments delivered results accurate at the level of $10^{-7}$.\label{comp7}}
\end{center}
\end{table}

\begin{table}
   \begin{center}  
\begin{tabular}{|c|c|c|c|c|}
\hline
$\omega$ & $\otherCorrections{n}$ & \# iter CFIE~\eqref{eq:CFIER_D} $\eta_D=1$ & \# iter CFIE $\eta_D$ opt & \# iter CFIER $\RD={\rm PS}_\kappa(\Yiip)$\\
\hline
\hline
10 & 64 & 73 & 50 & 52\\
20 & 128 & 137 & 80 & 89\\
40 & 256 & 263 & 144 & 116\\
80 & 512 & 511 & 185 & 235\\
160 & 1024 & 1008 & 320 & 406\\
\hline
\end{tabular}
\caption{\firstReviewer{Numbers of GMRES iterations to reach residuals of $10^{-8}$ for various BIE} formulations of Dirichlet elastodynamic scattering problems at high frequencies in the case when $\Gamma$ is the cavity contour. The material parameters are $\lambda=2$ and $\mu=1$, and the incidence was an S-wave. The discretizations used in these numerical experiments delivered results accurate at the level of $10^{-7}$.\label{comp8}}
\end{center}
\end{table}

\paragraph{Neumann boundary conditions}

We present next in Tables~\ref{comp9} and~\ref{comp10} results related to the iterative behavior of the BIE formulations considered in this paper for the solution of elastodynamic problems with Neumann boundary conditions in the case of smooth scatterers. Specifically, we consider (a) the classical CFIE formulation~\eqref{eq:CFIE_N} with the coupling parameter $\eta_N=1$, (b) the same formulation with the optimal choice of the coupling constant $\eta_N$ given in equation~\eqref{eq:eta_N}, and (b) the CFIER formulation with the regularizing operator $\RN =({\rm PS}_\kappa(\Yiip))^{-1}$ as well as  $\RN =({\rm PS}_{\kappa_p,\kappa_s}(\Yiip))^{-1}$  defined in equation~\eqref{eq:PSYpm}. For the latter regularizer we used $\kappa_p=k_p+0.4 i H^{2/3}k_p^{1/3}$  and $\kappa_s=k_s+0.4 i H^{2/3}k_s^{1/3}$ where $H$ is the maximum absolute value of the curvature of the boundary curve $\Gamma$. We present results for two level of discretizations, a coarser one that delivers results within the $10^{-6}$ range accuracy, and a finer one that produces results in the $10^{-7}$ accuracy range. We observe that given that the BIO $\Wii$ is a pseudodifferential operator of order one, the CFIE formulations require more GMRES iterations to reach the same residual for finer discretizations, whereas the CFIER formulations appear to behave like an integral equation of the second kind. The use of an optimized coupling constant $\eta_N$ is beneficial for CFIE formulations, and the use of the regularizing operator $\RN =({\rm PS}_\kappa(\Yiip))^{-1}$ in the CFIER formulations gives rise to savings of a factor of $1.7$ in the numbers of GMRES iterations when finer discretizations are applied. Given the fact that the Fourier multiplier $({\rm PS}_\kappa(\Yiip))^{-1}$ can be effected efficiently with FFTs, the savings in computational times are of the same magnitude.

\begin{table}
   \begin{center}  
   {

\begin{tabular}{|c|c|c|c|c|c|}
\hline
$\omega$ & $n$ & \specialcell[t]{\# iter CFIE~\eqref{eq:CFIE_N}\\ $\eta_N=1$} &
\specialcell[t]{\# iter CFIE\\ $\eta_N$ opt}& 
\specialcell[t]{\# iter CFIER \\ $\RN =({\rm PS}_\kappa(\Yiip))^{-1}$} &
\specialcell[t]{\# iter CFIER \\ \firstReviewer{$\RN= {\rm PS}_\kappa(\Yiip))_{\kappa_p,\kappa_s}^{-1}$ }}\\
\hline
\hline
10 & 64/128 & 86/114 & 49/63          & 29/29   &\firstReviewer{34/34}\\
20 & 128/256 & 166/224 & 64/81        & 41/41   &\firstReviewer{46/46}\\
40 & 256/512 & 300/407 & 90/117       & 63/63   &\firstReviewer{71/71}\\
80 & 512/1024 & 646/890 & 164/223     & 109/109 &\firstReviewer{113/113}\\
160 & 1024/2048 & 1358/3079 & 319/420 & 198/198 &\firstReviewer{230/230}\\
\hline
\end{tabular}
}
\caption{\firstReviewer{Numbers of GMRES iterations to reach residuals of $10^{-8}$ for various BIE} formulations of Neumann elastodynamic scattering problems at high frequencies in the case when $\Gamma$ is the starfish contour. The material parameters are $\lambda=2$ and $\mu=1$, and the incidence was an S-wave. The coarser/finer discretizations used in these numerical experiments delivered results accurate at the level of $10^{-6}$ and $10^{-8}$ respectively.\label{comp9}}
\end{center}
 \end{table}

\begin{table}
   \begin{center}  
   {
\begin{tabular}{|c|c|c|c|c|c|}
\hline
$\omega$ & $n$ & \specialcell[t]{\# iter CFIE~\eqref{eq:CFIE_N}\\ $\eta_N=1$} &
\specialcell[t]{\# iter CFIE\\ $\eta_N$ opt}& 
\specialcell[t]{\# iter CFIER \\ $\RN =({\rm PS}_\kappa(\Yiip))^{-1}$} &
\specialcell[t]{\# iter CFIER \\ \firstReviewer{$\RN= {\rm PS}_\kappa(\Yiip))_{\kappa_p,\kappa_s}^{-1}$ }}\\
\hline
\hline
10 & 64/128 & 136/185 & 84/107          & 62/62  &\firstReviewer{57/57} \\
20 & 128/256 & 239/324 & 133/170        & 101/101&\firstReviewer{86/86}\\
40 & 256/512 & 462/629 & 242/312        & 181/181&\firstReviewer{146/146}\\
80 & 512/1024 & 864/1177 & 389/504      & 294/294&\firstReviewer{255/255}\\
160 & 1024/2048 & 2480/4101 & 707/914   & 514/514&\firstReviewer{453/453}\\
\hline
\end{tabular}
}
\caption{\firstReviewer{Numbers of GMRES iterations to reach residuals of $10^{-8}$ for various BIE} formulations of Neumann elastodynamic scattering problems at high frequencies in the case when $\Gamma$ is the cavity contour. The material parameters are $\lambda=2$ and $\mu=1$, and the incidence was an S-wave. The coarser/finer discretizations used in these numerical experiments delivered results accurate at the level of $10^{-6}$ and $10^{-8}$ respectively.\label{comp10}}
\end{center}
 \end{table}


\subsubsection{Transmission problems}

\firstReviewer{We conclude the numerical results section with five  experiments concerning elastodynamic transmission problems involving four BIE formulations, namely (a) the  Kress-Roach  
type CFIE formulation~\eqref{eq:KiptmanKress}; (b) the Costabel-Stephan formulation \eqref{eq:stephan:form};  (c) the CFIER formulation~\eqref{eq:DCFIER} that incorporates the regularizing operators defined in equation~\eqref{def:Rkappa}; (d) the Optimized Schwarz formulation~\eqref{eq:DDM_rtr} with transmission operators defined in equation~\eqref{eq:opt_choice} and (e) the Optimized Schwarz formulation~\eqref{eq:DDM_rtr} with transmission operators  ${\bm{\Upsilon}}_{\mp}:=-\mathrm{PS}_{\kappa_p^\pm,\kappa_s^\pm}(Y^{\pm})$. These formulations are tested in the starfish domain   and the cavity problem. We observe that they behave in practice like formulations of the second kind, that is the numbers of GMRES iterations required to reach a certain residual do not grow with discretization size. As it can be seen from the data presented in Tables~\ref{comp15} and~\ref{comp16}, while the CFIER formulation delivers important reductions in the numbers of GMRES iterations over the CFIE and the Costabel-Stephan formulation, it is the Optimized \firstReviewer{Schwarz} (OS) formulation~\eqref{eq:DDM_rtr} that has the best iterative behavior in the high-frequency high-contrast regime.}

\firstReviewer{This better performance of the OS formulations could  (partial and/or empirically) be explained by the distribution of the eigenvalues of the matrices of the respective linear systems in the complex plane as can be observed in Figures \ref{fig:02} and \ref{fig:03}. Indeed, in these Figures the OS formulation seems to be the one which  yields the most compact distribution of the eigenvalues around the accumulation point(s).}

\begin{table}  
\begin{center}  
\resizebox{\textwidth}{!}
   {
\begin{tabular}{|c|c|c|c|c|c|c|c|}
\hline
$\omega$ & $n$ & 
\specialcell[t]{ \# iter CFIE \\ \eqref{eq:KiptmanKress}} &
\specialcell[t]{\# iter CoSt  \\ \otherCorrections{\eqref{eq:stephan:form}} }& 
\specialcell[t]{\# iter CFIER  \\ \otherCorrections{\eqref{eq:DCFIER:02}}} &
\specialcell[t]{\# iter OS\\ \eqref{eq:DDM_rtr}  $\mathrm{PS}_{\kappa}(Y^{\pm})$}&
\specialcell[t]{\# iter OS \\ \eqref{eq:DDM_rtr} $\firstReviewer{\mathrm{PS}_{\kappa_p^\pm,\kappa_s^\pm}(Y^{\pm})}$} 
\\
\hline
\hline
10 & 64/128    & 90/62   & 63/63     & 44/44    &  27/27   & 34/34 \\
20 & 128/256   & 146/114 & 114/116   & 71/72    &  36/36   & 43/43\\
40 & 256/512   & 236/207 & 213/215   & 133/133  &  60/60   & 66/66\\
80 & 512/1024  & 412/372 & 377/380   & 220/220  &  77/77   & 74/74\\
160 & 1024/2048& 647/594 & 615/617   & 358/358  & 106/106  & 87/87\\
\hline
\end{tabular}
}
\caption{Numbers of GMRES iterations of various formulations to reach residuals of $10^{-6}$ for various BIE formulations of transmission elastodynamic problems at high frequencies in the case when $\Gamma$ is the starfish geometry. The material parameters are $\otherCorrections{\lambda_+}=2$, $\otherCorrections{\mu_+}=8$ and $\otherCorrections{\lambda_-}=1$, $\otherCorrections{\mu_-}=1$, and the incidence was an P-wave. The discretizations used in these numerical experiments delivered results accurate at the level of $10^{-6}$.\label{comp15}}
\end{center}
\end{table}

\begin{table}
   \begin{center}  
\resizebox{\textwidth}{!}
   {
\begin{tabular}{|c|c|c|c|c|c|c|}
\hline
$\omega$ & $n$ & 
\specialcell[t]{ \# iter CFIE \\ \eqref{eq:KiptmanKress}} &
\specialcell[t]{\# iter CoSt  \\ \otherCorrections{\eqref{eq:stephan:form}} }& 
\specialcell[t]{\# iter CFIER  \\ \otherCorrections{\eqref{eq:DCFIER:02}}} &
\specialcell[t]{\# iter OS\\ \eqref{eq:DDM_rtr}  $\mathrm{PS}_{\kappa}(Y^{\pm})$}&
\specialcell[t]{\# iter OS \\ \eqref{eq:DDM_rtr} $\firstReviewer{\mathrm{PS}_{\kappa_p^\pm,\kappa_s^\pm}(Y^{\pm})}$} 
\\
\hline     
\hline
10 & 64/128     & 122/122   & 125/125    & 79/80   &  34/ 34 &  46 /46  \\
20 & 128/256    & 212/212   & 219/219    & 132/134 &  48/ 48 &  58 /58  \\
40 & 256/512    & 342/345   & 355/355    & 203/204 &  64/ 64 &  75 /75  \\
80 & 512/1024   & 598/609   & 623/623    & 336/342 &  91/ 91 &  99 /99  \\
160 & 1024/2048 & 1000/1005 &  1041/1041 & 493/491 & 108/108 &  125/125 \\
\hline
\end{tabular}
}
\caption{Numbers of GMRES iterations of various formulations to reach residuals of $10^{-6}$ for various BIE formulations of transmission elastodynamic problems at high frequencies in the case when $\Gamma$ is the cavity geometry. The material parameters are $\otherCorrections{\lambda_+}=2$, $\otherCorrections{\mu_+}=8$ and $\otherCorrections{\lambda_-}=1$, $\otherCorrections{\mu_-}=1$, and the incidence was an P-wave. The discretizations used in these numerical experiments delivered results accurate at the level of $10^{-5}$.\label{comp16}}  
\end{center} 
\end{table}

\begin{figure}[h]
\begin{center}
\includegraphics[width=0.45\textwidth]{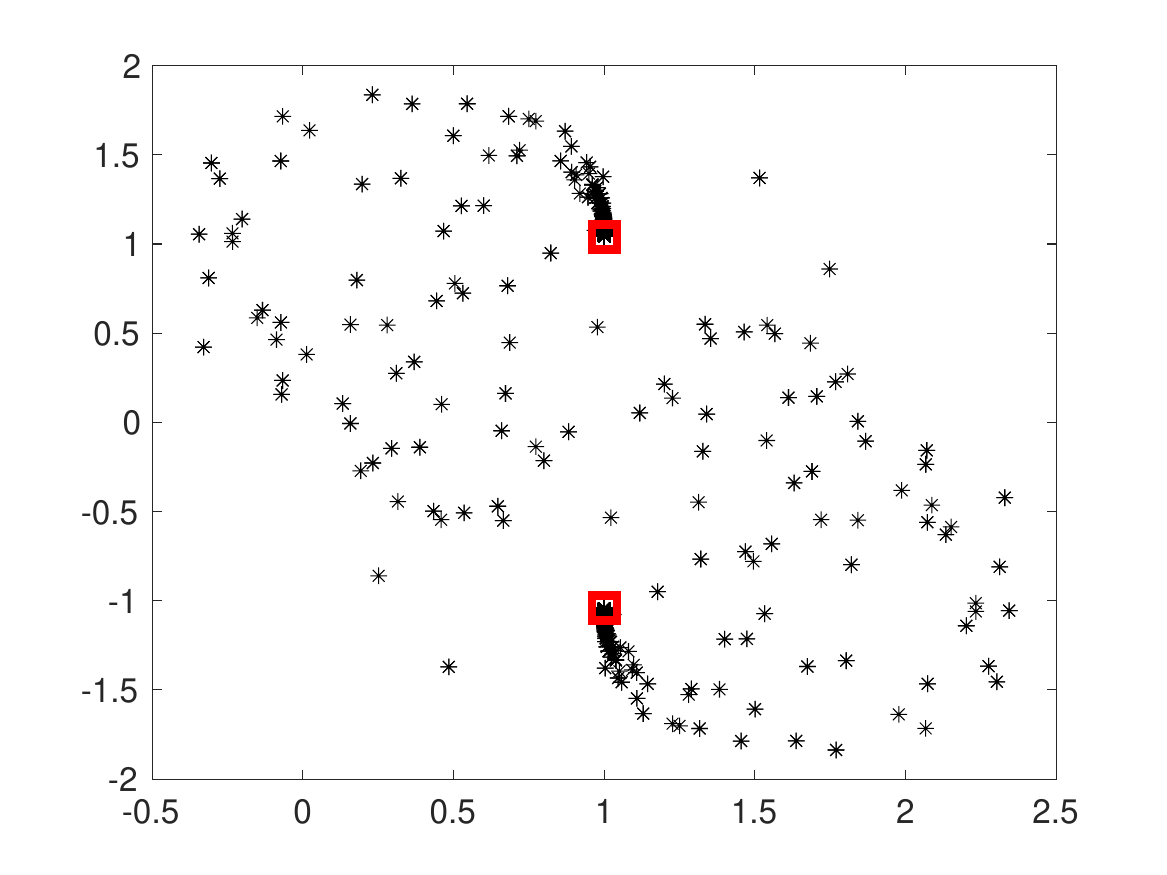}
\includegraphics[width=0.45\textwidth]{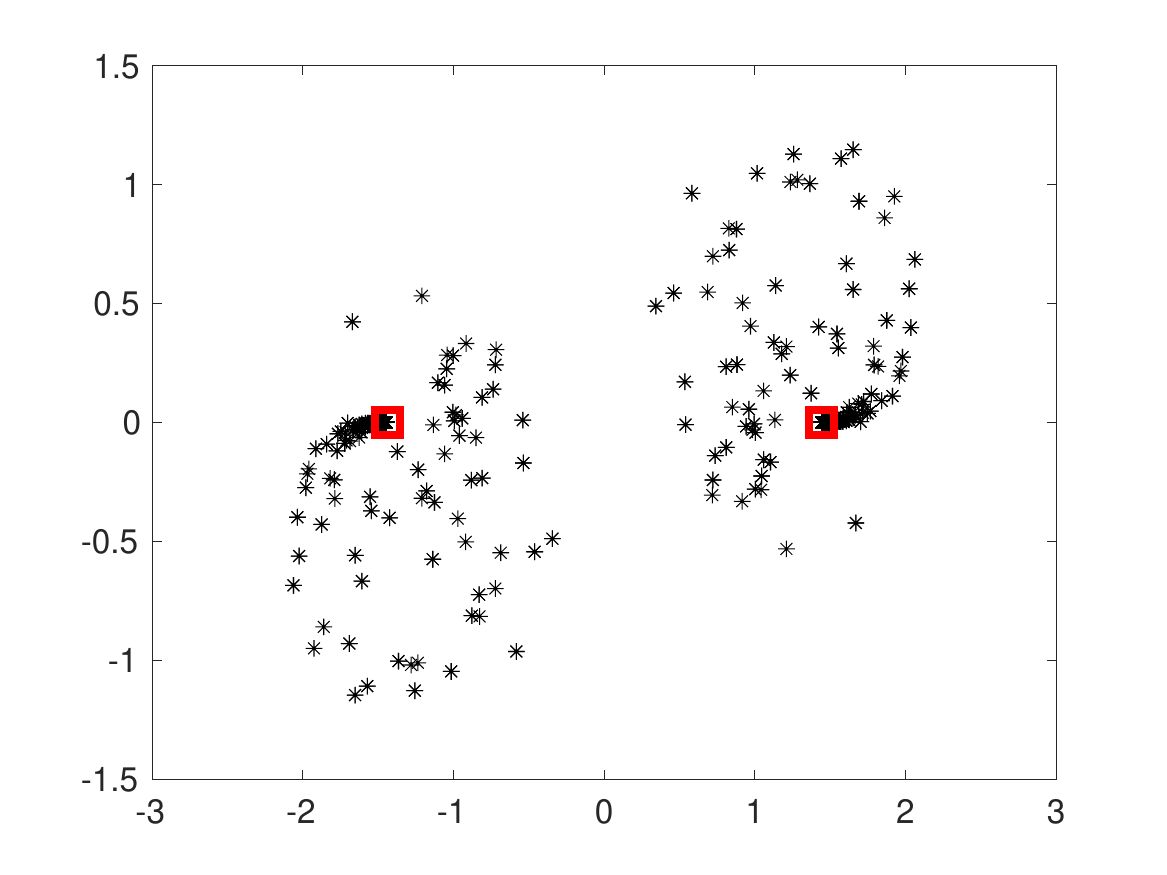}\\
\includegraphics[width=0.45\textwidth]{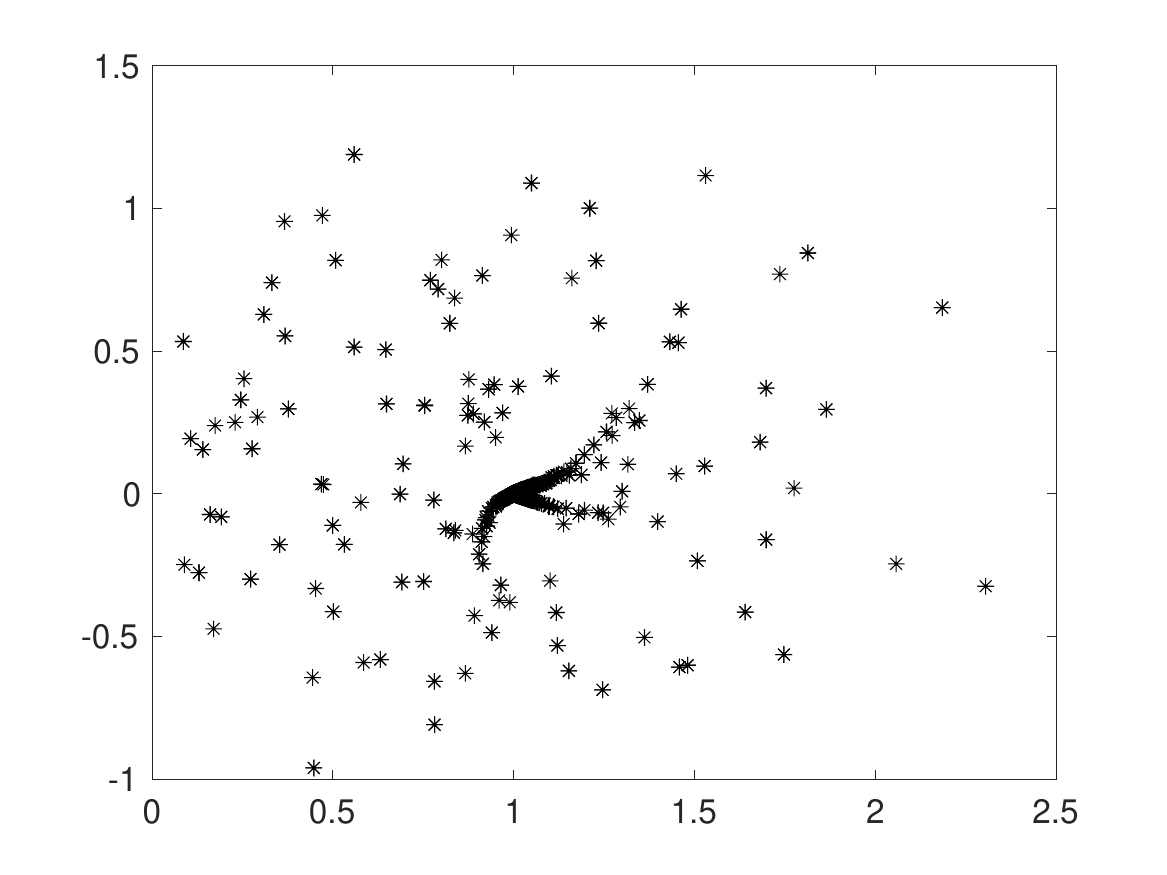} 
\includegraphics[width=0.45\textwidth]{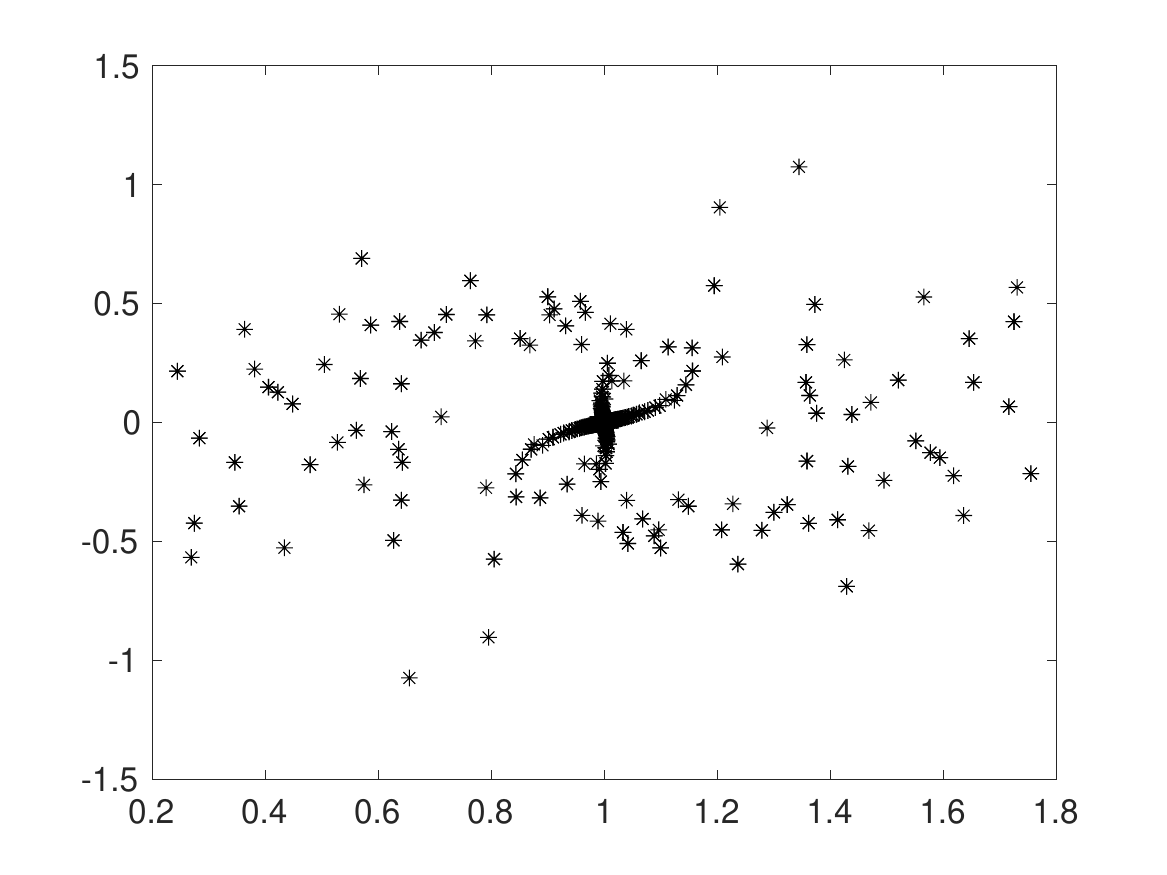}\\
\end{center}
\caption{\label{fig:02}Eigenvalue distributions in the complex plane for the matrix of the Kress-Roach, Costabel-Stephan (top, from left to right) and CFIER and Schwarz (below) formulations with \eqref{eq:opt_choice} for $\omega=20$, $\lambda_+=2$, $\mu_+=8$, $\lambda_-=\mu_-=1$. According to Theorem \ref{th:stephan-costabel}, and since $\rho= 901/432$, the eigenvalues accumulate at  $1\pm 1.04 i$  in case of the Kress-Roach formulation, at $\pm 1.44$ in the case of the  Costabel-Stephan formulation and at 1 for the CFIER and Schwarz formulations respectively. Notice that the eigenvalues are clustered closely to $1$ for the Optimized Schwarz formulations  which could explain the faster convergence of GMRES in this case.}
\end{figure}
\begin{figure}[h]
\begin{center}
\includegraphics[width=0.45\textwidth]{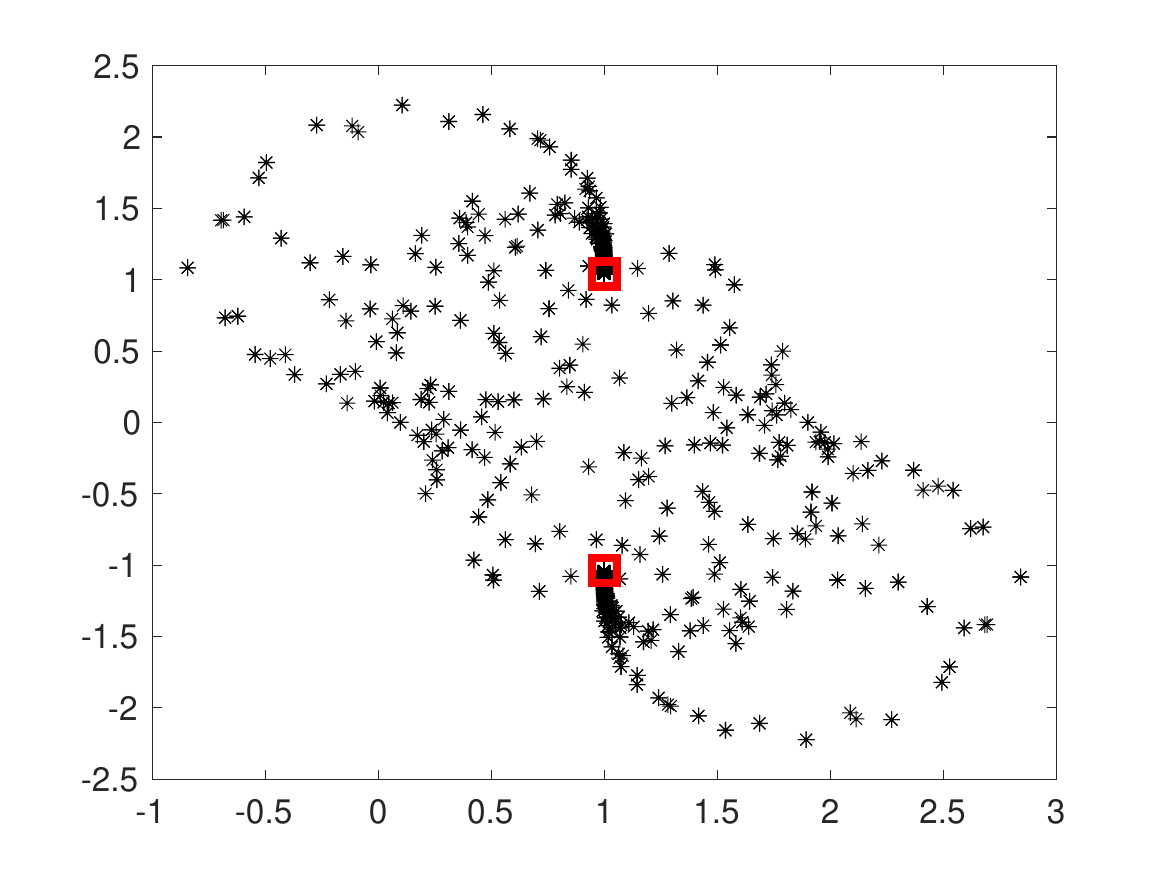}
\includegraphics[width=0.45\textwidth]{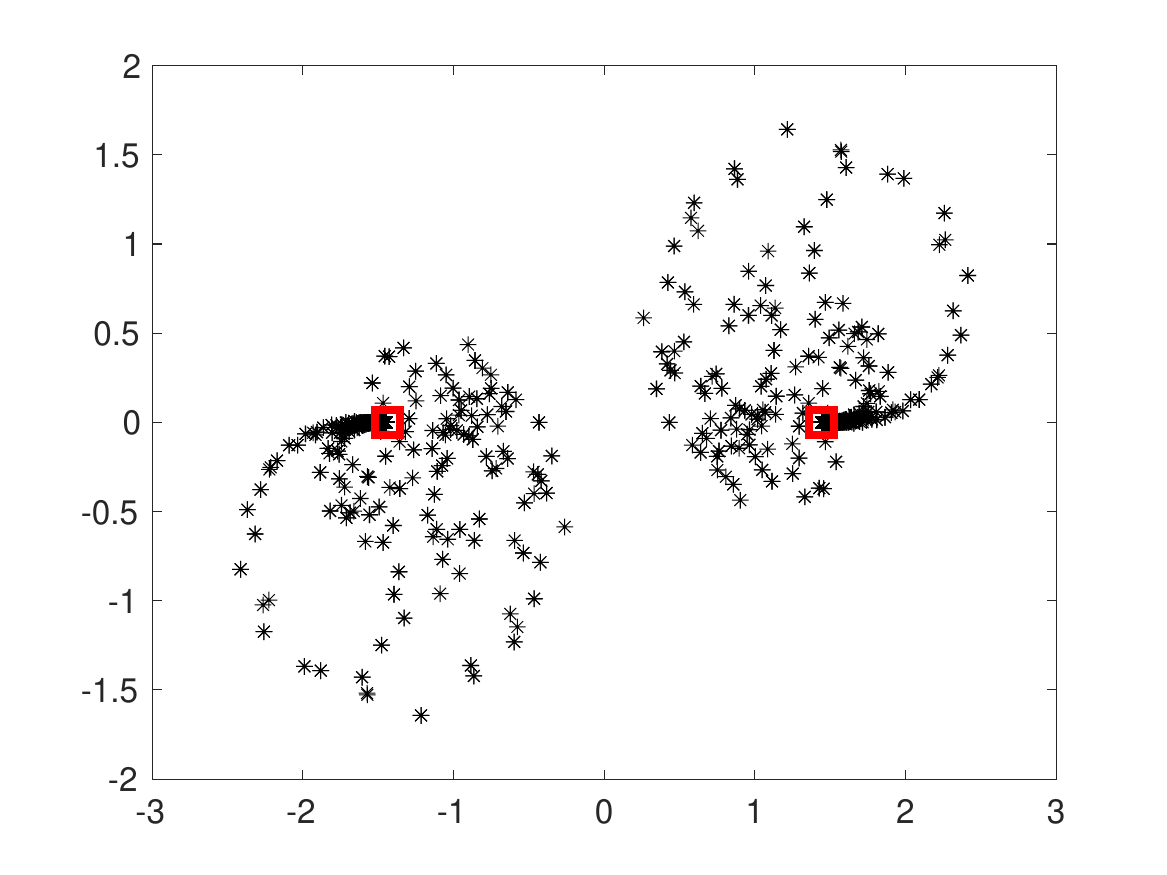}\\
\includegraphics[width=0.45\textwidth]{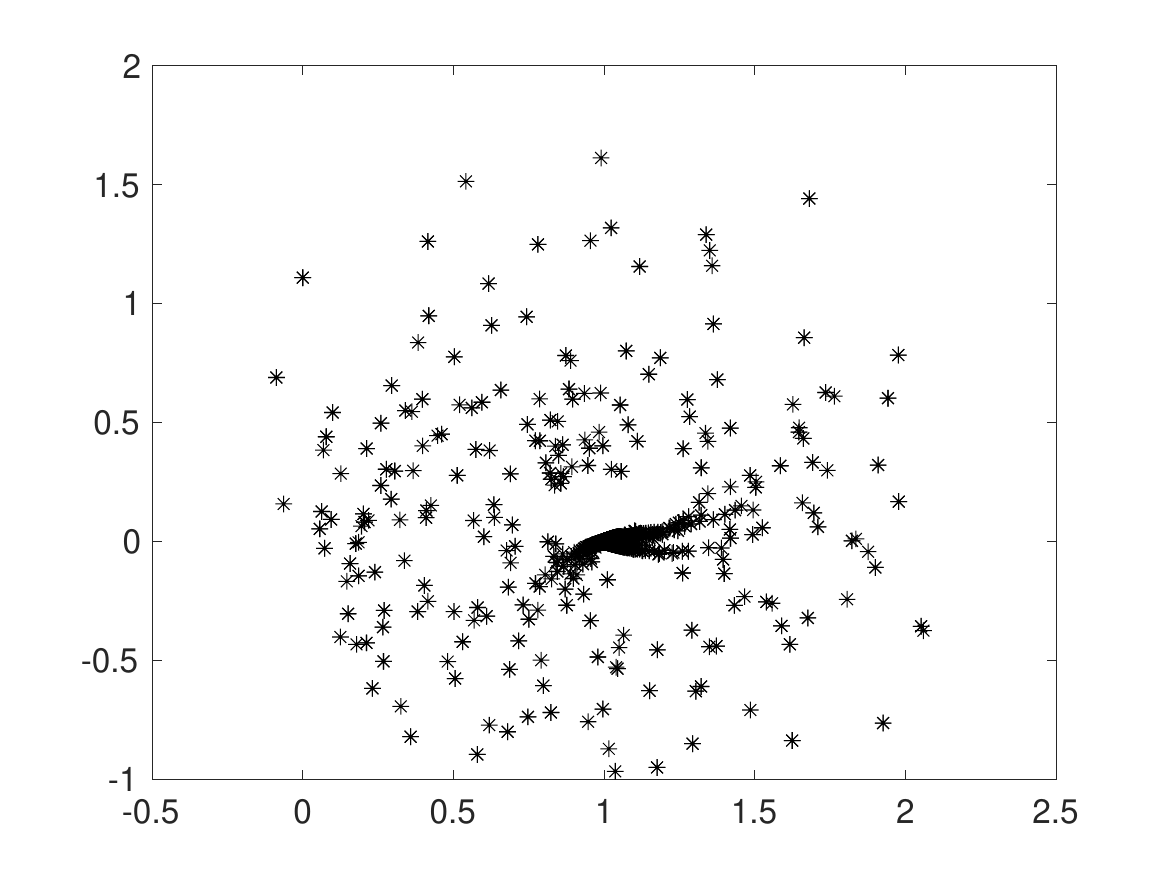} 
\includegraphics[width=0.45\textwidth]{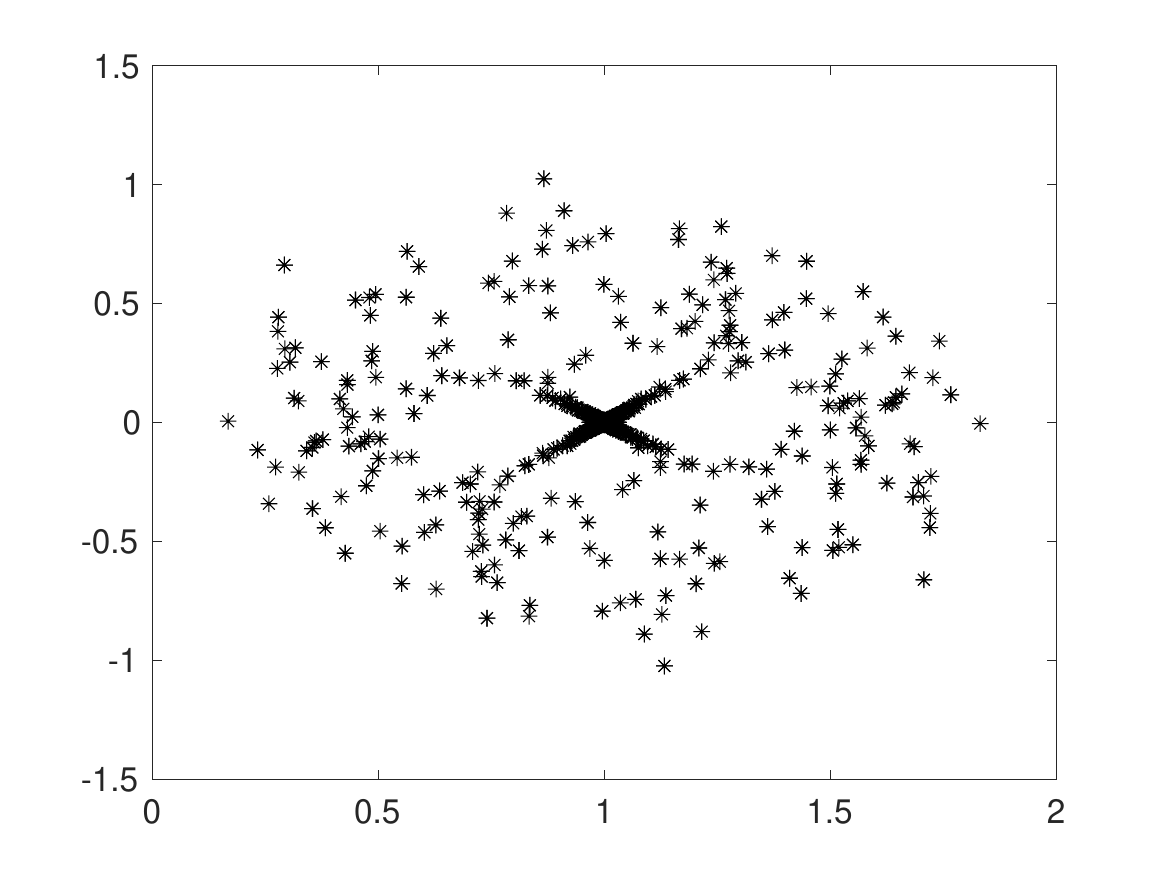}\\
\end{center}
\caption{\label{fig:03}Eigenvalue distributions in the complex plane for the matrix of the Kress-Roach, Costabel-Stephan (top, from left to right) and CFIER and Schwarz (below) formulations with \eqref{eq:opt_choice} for $\omega=40$, $\lambda_+=2$, $\mu_+=8$, $\lambda_-=\mu_-=1$. According to Theorem \ref{th:stephan-costabel}, and since $\rho= 901/432$, the eigenvalues accumulate at $1\pm 1.04 i$  in case of the Kress-Roach formulation, at $\pm 1.44$ in the case of the  Costabel-Stephan formulation and at 1 for the CFIER and Schwarz formulations respectively. Notice that the eigenvalues are clustered closely to $1$ for the Optimized Schwarz formulations which could explain the faster convergence of GMRES in this case. }
\end{figure}

\section{Conclusions}
   
We analyzed various boundary integral formulations of elastodynamic scattering and transmission problems in smooth two dimensional domains, including regularized formulations and Optimized Schwarz methods that rely on approximations of Dirichlet to Neumann (DtN) operators. We presented a singularity splitting based high-order Nystr\"om framework for the discretization of the four boundary integral operators associated with time-harmonic Navier equations in two dimensions. We provided extensive numerical evidence that the regularized formulations lead to faster convergence rates when iterative solvers such as GMRES are used for the solution of the numerical solution of the linear systems that result from the Nystr\"om discretization of elastodynamic scattering and transmission problems. Extensions to Lipschitz domains and three dimensional configurations are currently underway. {With regards to 3D extensions, the authors plan to use the kernel independent density interpolation strategy recently introduced in~\cite{faria2021general} which resolves weakly singular, singular, and strongly singular kernels encountered in the elastodynamic BIOs. The implementation of the Fourier principal calculus in the case of surfaces in 3D, on the other hand, is subject of current research investigation.}

\section*{Acknowledgments}
 Catalin Turc gratefully acknowledges support from NSF through contract DMS-1614270 and DMS-1908602. 
 \otherCorrections{We would like to thank the reviewers for their thorough work in reviewing this manuscript and for their valuable comments and suggestions that certainly have helped us to improve both  the analysis and exposition of the results.}

\bibliography{biblioEL}

 \appendix
\section{Appendix: Detailed factorization of the kernels of the elastodynamic boundary integral operators}

We show in this section explicit factorizations of the kernels of the \otherCorrections{BIOs} for elastodynamic. The aim is twofold. First, it supports the implementation of the Nystr\"om method {sketched in Section 5}, and, secondly, provides justification of the derivation of the principal part of the operators. 

  For these purposes, let us assume that $\Gamma$ is parameterized, such as it has been assumed  throughout  this paper, by a smooth $2\pi-$periodic 
parameterization ${\bf x}$.  Recall that  
\[
  {\bm r}:={\bm r}(\tau,t)={\bf x}(\tau)-{\bf x}(t),\quad  r=r(\tau,t)=|{\bf r}|. 
\]
The fundamental solution of the Navier equations
\begin{eqnarray*}
 \Phi(\x,\y)&:=&\frac{1}{\mu}\phi_0(k_s |\x-\y|)\firstReviewer{\bm{I}} +\frac{1}{\omega^2}\nabla_{\x}\nabla_{\x}^\top(\phi_0(k_s |\x-\y|)-\phi_0(k_p | \x-\y|))
\end{eqnarray*}
can be factorized as 
 \[ 
\Phi(\x,\y)=\Phi_1(|\x-\y|)I_2+\Phi_2(|\x-\y|) G( \x-\y )
\] 
where $I_2$ is the identity matrix of order 2, 
\begin{eqnarray*}
\Phi_1(z)&:=& \frac{k_s^2}{\omega^2}\phi_0(k_s z)-\frac{1}{\omega^2} \psi(z)\\
\Phi_2(z)&:=& \frac{1}{\omega^2}\left(k_p^2\phi_0(k_p z)-k_s^2\phi_0(k_s z)\right)+\frac{2}{\omega^2}\psi(z)\\
\psi(z)  &:=&\left[k_s^2 \frac{\phi_1(k_s z)}{k_s z}-k_p^2 \frac{\phi_1(k_p z)}{k_p z}\right]\\
G({\bm x})&:=&\frac{1}{|\bm{x}|^2}\bm x \bm{x}^\top=\frac{1}{x_1^2+x_2^2}\begin{bmatrix}
  x_1^2   & x_1 x_2\\
  x_1 x_2 & x_2^2
   \end{bmatrix}.
\end{eqnarray*}
where
 \begin{eqnarray*}
  \phi_j(z)&:=&\frac{i}4 H_j^{(1)}(z)\\
  k_p^2&=&\frac{\omega^2}{\lambda+2\mu},\quad k_s^2=\frac{\omega^2}{\mu}.
 \end{eqnarray*} 
 
By definition
 \[
  \phi_j=\frac{i}4\big(J_j+{i}\:Y_j\big)
 \]
 where $J_j, Y_j$ are the Bessel functions of first and second kind and order $j$. Moreover cf. 
 \cite[\S 10.6]{NIST:DLMF}
\begin{eqnarray*}
\frac{\rm d}{{\rm d}z}\left[\phi_0(k z)\right]=-k \phi_1(kz),\qquad
\frac{\rm d}{{\rm d}z}\left[\frac{\phi_1(k z)}{kz}\right]=\frac{kz\phi_0(kz)-2\phi_1(kz)}{kz^2} 
\end{eqnarray*}
%
and \cite[\S 10.8]{NIST:DLMF}
 \[
  \phi_r(k z)=-\frac{1}{2\pi} J_r(k z)\log z +z^rC_{r,k}(z)+
  \begin{cases}
        0,&r=0,\\[2ex]
        \displaystyle\frac{1}{2\pi k z},&r=1,\\[2ex]
        \displaystyle\frac{1}{\pi k^2 z^2}+\frac{1}{4\pi},&r=2,\\[2ex]
        \displaystyle\frac{4}{\pi k^3 z^3}+\frac{1}{2\pi k z}+\frac{k z}{16\pi},&r=3,
                                                            \end{cases}
 \]
with $C_{r,k}(z)$  smooth functions as $z\ge 0$. 

We are ready to state the first result for the kernel of the single layer operator: 
 


\begin{lemma}\label{lemma:Ap:01}
Let $V$ be the kernel of the (parameterized) single layer operator \eqref{eq:BLO:SL}. Then 
\begin{eqnarray*}
V(\tau,t)
&=&V_0(\tau,t) + a_{\rm log}^{(2)}(r)r^2\log r \:   I_2+a_{\rm log}^{(3)}(r)r^2\log r \: G(\rr)+a_{\rm reg}^{(4)}(r) I_2 + a_{\rm reg}^{(5)}(r) r^2 G(\rr) 
\end{eqnarray*} 
where
\begin{eqnarray*}
a_{\rm log}^{(2)}(z)&:=&\frac{1}{2\pi z^2 \omega^2}\left(-k_p^2 \frac{J_1(k_p z)}{k_p z}-k_s^2 J_0(k_s z)+k_s^2 \frac{J_1(k_s z)}{k_s z}+ \frac{k_s^2+k_p^2}{2}\right)\\
a_{\rm log}^{(3)}(z)&:=&  \frac{1}{2\pi z^2\omega^2}\left(-k_p^2 J_0(k_p z)+2k_p^2\frac{J_1(k_p z)}{k_p z} +k_s^2 J_0(k_s z)-2k_s^2\frac{J_1(k_s z)}{k_s z}
\right) \\
a_{\rm reg}^{(4)}(z)&:=&\frac{1}{\omega^2}\left( k_p C_{1,k_p}(z)+k_s^2 C_{0,k_s}(z)-k_s C_{1,k_s}(z)\right)\\
a_{\rm reg}^{(5)}(z)&:=&\frac{1}{\omega^2 z^2 }\left(k_p^2 {C_{0,k_p}(z)}-2 k_p C_{1,k_p}(z)-k_s^2 {C_{0,k_s}(z)+2 k_s C_{1,k_s}(z)}+\frac{1}{4\pi}(k_p^2-k_s^2)
\right)
\end{eqnarray*}
and
\[
 V_0(\tau,t):= V_0(\bm{r})=
- \frac{\lambda +3 \mu }{ \mu  (\lambda +2 \mu )}\frac{1}{\otherCorrections{4\pi}} \log r \: I_2
 + \frac{\lambda + \mu }{\mu  (\lambda +2 \mu )}\frac{1}{4\pi }G(\rr).
\]
Furthermore, 
\begin{eqnarray*}
a_{\rm log}^{(2)}(z)& =&\frac1{32\pi \omega^2}\left(k_p^4+3 k_s^4\right)
+{\cal O}(z^2)\\
a_{\rm log}^{(3)}(z)& =&\frac1{16\pi \omega^2}\left(k_p^4-k_s^4\right)
+{\cal O}(z^2)\\
a_{\rm reg}^{(4)}(z)& =&
\frac{(1-2 E +{i}\pi ) k_p^2-2 k_p^2 \log \left(\frac{k_p}{2}\right)+(-1-2 E +{i}\pi )
   k_s^2-2 k_s^2 \log \left(\frac{k_s}{2}\right)}{8 \pi  \omega^2} 
 +{\cal O} (z^2).
\\
a_{\rm reg}^{(5)}(z)& =& \frac{ (-3+4 E -2\pi i  )
   \left(k_p^4-k_s^4\right)+4 k_p^4 \log \left(\frac{k_p}{2}\right)-4 k_s^4 \log
   \left(\frac{k_s}{2}\right) }{64 \pi\omega^2 } 
   +{\cal O}(z^{2})
\end{eqnarray*}
where $E\approx 0.5772$ is the Euler-Mascheroni constant.
\end{lemma}

For the \otherCorrections{double layer} operator \eqref{eq:BLO:DL}, we need the following vector fields (we follow the notation from \cite{chapko2000numerical}):
\begin{eqnarray*}
 U_1(t,\tau)&=&\lambda \bm{\nu}(t)\:\rr^\top+\mu \rr\: \bm{\nu}(t)^\top+\mu (\bm{\nu}(t)\cdot  \rr)I_2\\
 U_2(t,\tau)&=&(\lambda+2\mu) \bm{\nu}(t)\:{\rr}^\top+\mu {\rr}\: \bm{\nu}(t)^\top+\mu (\bm{\nu}(t)\cdot  \rr)(I_2-4G(\rr)).
\end{eqnarray*}
(Recall  that ${\bm \nu}(t)=(x_2'(t),-x_1'(t))={\bm n\circ{\bf x}(t)}|{\bf x}'(t)|$ the non-normalized vector field given by the parameterization.) 
Using the identities (recall the definition of the normal stress tensor in \eqref{eq:def:T} and its parameterized version in \eqref{eq:def:T:2})
\begin{eqnarray}
T(fF)&=&T(f I_2)F+f\:T(F),\quad \text{for any pair $f$ and $F$ of scalar and matrix functions}\nonumber\\
{[}T_t(f(r)\, I_2){]}(\rr)&=&-\frac{f'(r)}{r} U_1(t,\tau),\quad 
{[}T_t G{]}(\rr)\ =\ -\frac{1}{r^2} U_2(t,\tau)\label{propT}
\end{eqnarray}
it is possible to prove 
\begin{lemma}\label{lemma:Ap:02}
Let $K(\tau,t)$ the kernel of the \otherCorrections{double layer} operator \eqref{eq:BLO:DL}. Then 
\begin{eqnarray*}
K(\tau,t)
&=&K_0(\tau,t)
        \\
&& +\left[b_{\log}^{(2)}(r) U_1^\top(\tau,t)+b_{\log}^{(3)}(r)  U_1^\top(\tau,t) G(\rr)+
b_{\log}^{(4)}(r) U_2^\top(\tau,t)\right] \:  \log r\nonumber\\
&&+ b_{\rm reg}^{(5)}(r) U_1^\top(\tau,t)+b_{\rm reg}^{(6)}(r) U_1^\top(\tau,t) G(\rr)
+ b_{\rm reg}^{(7)}(r) U_2^\top(\tau,t)
\end{eqnarray*} 
where the functions 
\begin{eqnarray*}
b_{\rm log}^{(2)}(z)&:=&\frac{1}{2\pi\omega^2}\left(-k_p^4\frac{J_2(k_p z)}{(k_p z)^2}-k_s^4\frac{J_1(k_s z)}{k_s z}+k_s^4\frac{J_2(k_s z)}{(k_s z)^2}\right)\\
b_{\rm log}^{(3)}(z)&:=&\frac{1}{2\pi\omega^2}\left(-k_p^4\frac{J_1(k_p z)}{k_p z}+
2k_p^4\frac{J_2(k_p z)}{(k_p z)^2}+k_s^4\frac{J_1(k_s z)}{k_s z}-
2k_s^4\frac{J_2(k_s z)}{(k_s z)^2}
\right)\\
b_{\rm log}^{(4)}(z)&:=&-a_{\rm log}^{(3)}(z)=-\frac{1}{2\pi\omega^2 z^2}\left(
-k_p^2 J_0(k_p z) 
+ 2k_p^2\frac{J_1(k_p z)}{k_p z}
+k_s^2 J_0(k_s z) 
-2k_s^2\frac{J_1(k_s z)}{k_s z}\right) \\
b_{\rm reg}^{(5)}(z)&:=&-\frac{1}{\omega^2}\left(
-k_s^3 C_{1,k_s}(z) + k_s^2 C_{2,k_s}(z) - k_p^2 C_{2,k_p}(z)\right) \\
b_{\rm reg}^{(6)}(z)&:=&-\frac{1}{\omega^2}\left(
k_s^3 C_{1,k_s}(z) - 2 k_s^2 C_{2,k_s}(z)  - 
     k_p^3 C_{1,k_p}(z)  + 2 k_p^2 C_{2,k_p}(z)   \right)\\
b_{\rm reg}^{(7)}(z)&:=&-a_{\rm reg}^{(5)}(z)=-\frac{1}{\omega^2 z^2 }\left(k_p^2 {C_{0,k_p}(z)}-2 k_p C_{1,k_p}(z)-k_s^2 {C_{0,k_s}(z)+2 k_s C_{1,k_s}(z)}+\frac{1}{4\pi}(k_p^2-k_s^2)
\right)
\end{eqnarray*}
are smooth for $r\ge 0$ and 
\[
 K_0(\tau,t)=\frac{\mu}{\lambda+2\mu} \left(\frac{\partial}{\partial t}
        \frac{1}{2\pi }\log r\right) \: \begin{bmatrix}
                                         &-1\\
                                         1&
                                        \end{bmatrix}
+\frac{1}{2\pi r^2}\left({\bm \nu}(t)\cdot\rr\right) \left(\frac{\mu}{\lambda+2\mu}I_2+
 2\frac{\lambda+\mu}{\lambda+2\mu}G(\rr)
 \right). 
\] 
Furthermore, 
\begin{eqnarray*}
b^{(2)}_{\log}(z)&=&-\frac{1}{16  \pi  \omega^2 }(k_p^4+3 k_s^4)
+{\cal O}\left(z^2\right)\\
b^{(3)}_{\log}(z)&=&\frac{1}{8 \pi  \omega^2}(k_s^4-k_p^4)
+{\cal O}\left(z^2\right)\\
b^{(4)}_{\log}(z)&=&-\frac{1}{16 \pi  \omega^2}(k_p^4-k_s^4) +{\cal O}\left(z^2\right)\\
b^{(5)}_{\rm reg}(z)&=&\frac{(3-4 E +2\pi i  ) k_p^4-4 k_p^4 \log \left(\frac{k_p}{2}\right)+(5-12 E +6\pi i  ) k_s^4-12 k_s^4 \log
   \left(\frac{k_s}{2}\right)}{64 \pi  \omega^2}+{\cal O}\left(z^2\right)\\
b^{(6)}_{\rm reg}(z)&=&   -\frac{(-1+4 E -2\pi i  ) \left(k_p^4-k_s^4\right)+4 k_p^4 \log \left(\frac{k_p}{2}\right)-4 k_s^4 \log
   \left(\frac{k_s}{2}\right)}{32 \pi \omega^2} 
  +{\cal O}\left(z^2\right)\\
b^{(7)}_{\rm reg}(z)&=&-\frac{(-3+4 E -2\pi i  ) \left(k_p^4-k_s^4\right)+4 k_p^4 \log \left(\frac{k_p}{2}\right)-4 k_s^4 \log
   \left(\frac{k_s}{2}\right)}{64 \pi \omega^2} 
 +{\cal O}\left(z^2\right)
\end{eqnarray*}

\end{lemma}
The hypersingular kernel requires to consider the action of operator $T_s$ on $U_1^\top (\tau,t)$ and $U_2^\top (\tau,t)$.  Hence, we use that with
\begin{eqnarray*}
R_1(\tau,t) &:=&2\lambda(\lambda+2\mu) {\bm\nu}(\tau)\:{\bm\nu}^\top (t)+2\mu^2\left({\bm\nu}(t)\:{\bm\nu}^\top (\tau)+({\bm\nu}(\tau)\cdot{\bm\nu}(t))I_2\right) \\
R_2(\tau,t) &:=& 2(\lambda+2\mu)(\lambda+\mu){\bm\nu}(\tau)\:{\bm\nu}^\top (t)\\
&&+2\mu\left(\mu {\bm\nu}(t) {\bm\nu}^\top(\tau)+\lambda {\bm\nu}(\tau){\bm\nu}^\top (t)+\mu({\bm\nu}(\tau)\cdot{\bm\nu}(t))I\right)(I-2 G(\rr))\\
&&+4\mu ( \rr\cdot{\bm\nu}(t))\frac{1}{r^2} U_2(t,\tau)\\
R_3(\tau,t) &:=& 2\lambda(\lambda+\mu){\bm\nu}(\tau){\bm\nu}^\top (t)\\
&&+2\mu\left(\mu {\bm\nu}(t)\:{\bm\nu}^\top (\tau)+\lambda {\bm\nu}(\tau)\:{\bm\nu}^\top (t)+\mu ({\bm\nu}(\tau)\cdot{\bm\nu}(t))I\right)G(\rr) \\
&&-2\mu  (\rr\cdot{\bm\nu}(t)) \frac{1}{r^2}U_2(t,\tau) 
\end{eqnarray*}
it holds
\[
 T_s[U_1^\top(\tau,t)]=R_1(\tau,t),\quad  T_s[U_2^\top(\tau,t)]=R_2(\tau,t),\quad T_s\big(G(\rr)U_1^\top(\tau,t)\big)=R_3(\tau,t)
\]

\begin{lemma}\label{lemma:Ap:03}
Let $W(\tau,t)$ the kernel of the hypersingular operator \eqref{eq:BLO:Hyp}. Then 
 \begin{eqnarray*}
  W(\tau,t) &=& W_0(\tau,t) \\
  &&+\Big[ c_{\rm log}^{(1)}(r) \frac{1}{r^2}U_1(t,\tau)U_1^\top(\tau,t) +
c_{\rm log}^{(2)}(r) R_1(\tau,t) + c_{\rm  log}^ {(3)}(r) \frac{1}{r^2}U_1(t,\tau)G(\rr)U_1^\top(\tau,t)  \\
&& \qquad +c_{\rm log}^{(4)}(r) R_3(\tau,t) +c_{\rm log}^{(5)}(r) \frac{1}{r^2}U_1(t,\tau)U_2^\top(\tau,t) +c_{\rm log}^{(6)}(r) R_2(\tau,t)
\Big]\log r\\
&& +\Big[ c_{\rm reg}^{(7)}(r)\frac{1}{r^2} U_1(t,\tau)U_1^\top(\tau,t) +
c_{\rm reg}^{(8)}(r) R_1(\tau,t)  + c_{\rm  reg}^ {(9)}(r)\frac{1}{r^2}U_1(t,\tau)G(\rr)U_1^\top(\tau,t)\\
&&\qquad +c_{\rm reg}^{(10)}(r) R_3(\tau,t)+c_{\rm reg}^{(11)}(r) \frac{1}{r^2}U_1(t,\tau)U_2^\top(\tau,t) +c_{\rm reg}^{(12)}(r) R_2(\tau,t)
\Big] 
 \end{eqnarray*}
where the functions 
\begin{eqnarray*}
 c_{\rm log}^{(1)}(z)&:=&-\frac{z^2}{2\pi\omega^2}\left( k_p^6\frac{J_3(k_p z)}{(k_p z)^3}+k_s^6 \frac{J_2(k_s z)}{(k_s z)^2}-k_s^6\frac{J_3(k_s z)}{(k_s z)^3}
\right)\\
 c_{\rm log}^{(2)}(z)&:=&-b_{\rm log}^{(2)}(z)= -\frac{1}{2\pi\omega^2}\left(-k_p^4\frac{J_2(k_p z)}{(k_p z)^2}-k_s^4\frac{J_1(k_s z)}{k_s z}+k_s^4\frac{J_2(k_s z)}{(k_s z)^2}\right)\\
 c_{\rm log}^{(3)}(z)&:=&-\frac{z^2}{2\pi\omega^2}\left(k_p^6
 \frac{J_2(k_p z)}{(k_p z)^2}-2k_p^6\frac{J_3(k_p z)}{(k_p z)^3}-k_s^6\frac{J_2(k_s z)}{(k_s z)^2}+2k_s^6\frac{J_3(k_s z)}{(k_p z)^3}
\right)\\          
 c_{\rm log}^{(4)}(z)&:=&-b_{\rm log}^{(3)}(z)=-\frac{1}{2\pi\omega^2}\left(-k_p^4\frac{J_1(k_p z)}{k_p z}+
2k_p^4\frac{J_2(k_p z)}{(k_p z)^2}+k_s^4\frac{J_1(k_s z)}{k_s z}-
2k_s^4\frac{J_2(k_s z)}{(k_s z)^2}
\right)\\ 
 c_{\rm log}^{(5)}(z)&:=&  -\frac{1}{2\pi\omega^2 z^2}\Big(-2 k_p^2 J_0(k_p z)-k_p^3 r J_1(k_p z)+4 k_p^2\frac{
   J_1(k_p z)}{k_p z}-2 k_p^2 J_2(k_p z)\\
   &&\qquad+2 k_s^2 J_0(k_s z)+k_s^3 r J_1(k_s z)-4 k_s^2\frac{J_1(k_s z)}{k_s z}-2 k_s^2 J_2(k_s
   r)\Big) \\
 c_{\rm log}^{(6)}(z)&:=&-b_{\rm log}^{(4)}(z)=-\frac{1}{2\pi\omega^2 z^2}\left(
k_p^2 J_0(k_p z) -
2k_p^2\frac{J_1(k_p z)}{k_p z}-
k_s^2 J_0(k_s z) +
2k_s^2\frac{J_1(k_s z)}{k_s z}\right)\\
  c_{\rm reg}^{(7)}(z)&:=&\frac{1}{16\pi\omega^2}(3k_s^4+k_p^4)+\frac{z^2}{\omega^2}(k_p^3 C_{3,k_p}(z)+k_s^4 C_{2,k_s}(z)-k_s^3C_{3,k_s}(z))\\
 c_{\rm reg}^{(8)}(z)&:=&-b_{\rm reg}^{(5)}(z)=
 \frac{1}{\omega^2}\left(
-k_s^3 C_{1,k_s}(z) + k_s^2 C_{2,k_s}(z) - k_p^2 C_{2,k_p}(z)\right)\\
 c_{\rm reg}^{(9)}(z)&:=& -\frac{z^2}{\omega^2} (2 k_p^3C_{3,k_p}(z) - k_p^4 C_{2,k_p}(z) - 
      2 k_s^3 C_{3,k_s}(z) + k_s^4 C_{2,k_s}(z)) + 
  \frac{k_p^4 - k_s^4}{8\pi \omega^2} ;
 \\
 c_{\rm reg}^{(10)}(z)&:=&- b_{\rm reg}^{(6)}(z)= \frac{1}{\omega^2}\left(
k_s^3 C_{1,k_s}(z) - 2 k_s^2 C_{2,k_s}(z)  - 
     k_p^3 C_{1,k_p}(z)  + 2 k_p^2 C_{2,k_p}(z)   \right)\\
 c_{\rm reg}^{(11)}(z)&:=&
 \frac{1}{z^2 \omega^2} \bigg(-k_p^3  z^2  C_{1,k_p}(z)+2 k_p^2 z^2 C_{2,k_p}(z)-2 k_p^2 C_{0,k_p}(z)+4
   k_p C_{1,k_p}(z)\\
   &&\qquad +k_s^3 z^2 C_{1,k_s}(z)-2 k_s^2 z^2 C_{2,k_s}(z)+2 k_s^2
   C_{0,k_s}(z)-4 k_s C_{1,k_s}(z)\\
   && \quad -\frac{z^2 \left(k_p^4-k_s^4\right)}{16 \pi
   }-\frac{k_p^2-k_s^2}{2 \pi }\bigg) + \frac{k_p^4-k_s^4}{16 \pi  \omega^2}
 \\
 c_{\rm reg}^{(12)}(z)&:=&-b_{\rm log}^{(7)}(z)=\frac{1}{z^2 \omega^2}\left(k_p^2 C_{0,k_p}(z)-2 k_p C_{1,k_p}(z)-k_s^2 C_{0,k_s}(z)+2 k_s
   C_{1,k_s}(z)+\frac{k_p^2-k_s^2}{4 \pi }\right)
\end{eqnarray*}  
are smooth for $r\ge 0$ and 
 \[
    W_0(\tau,t)={-\frac{\mu(\lambda+\mu)}{\lambda+2\mu} \frac{\partial^2}{\partial \tau \partial t}\frac{1}{{\pi}}\left(-\log r\:I_2+G(\rr)\right)}.
 \]
Furthermore, 
\begin{eqnarray*} 
  c_{\rm log}^{(1)}(z)&=&{\cal O}\left(z^2\right)\\
  c_{\rm log}^{(2)}(z)&=& \frac{ 1}{16 \pi  \omega^2}(k_p^4+3 k_s^4)+{\cal O}\left(z^2\right)\\
  c_{\rm log}^{(3)}(z)&=&{\cal O}\left(z^2\right)\\
  c_{\rm log}^{(4)}(z)&=&-\frac{1}{8 \pi  \omega^2}(k_s^4-k_p^4)+{\cal O}\left(z^2\right)\\
  c_{\rm log}^{(5)}(z)&=& {\cal O}\left(z^2\right) \\
  c_{\rm log}^{(6)}(z)&=&
    -\frac{1}{16 \pi  \omega^2}(-k_p^4+  k_s^4) +{\cal O}\left(z^2\right)\\
c_{\rm reg}^{(7)}(z)&=& \frac{k_p^4+3 k_s^2}{16 \pi  \omega^2} +{\cal O}\left(z^2\right)
\\
c_{\rm reg}^{(8)}(z)&=&-\frac{-4 k_p^4 \log k_p +k_p^4 (3-4 E +2 \pi i +\log (16))-12 k_s^4 \log k_s +k_s^4 (5-12 E +6 \pi i +\log    (4096))}{64 \pi  \omega^2} 
\\
   &&+{\cal O}\left(z^2\right)\\
c_{\rm reg}^{(9)}(z)&=&\frac{k_p^4-k_s^4}{8 \pi  \omega^2} +{\cal O}\left(z^2\right)\\
\\
c_{\rm reg}^{(10)}(z)&=&-\frac{(-1+4 E -2 \pi i -4 \log (2)) \left(k_p^4-k_s^4\right)+4 k_p^4 \log k_p -4 k_s^4 \log k_s }{32 \pi 
   \omega^2 }+{\cal O}\left(z^2\right)\\
\\
c_{\rm reg}^{(11)}(z)&=&\frac{k_p^4-k_s^4}{16 \pi  \omega^2}+{\cal O}\left(z^2\right)\\
\\
c_{\rm reg}^{(12)}(z)&=& \frac{(-3+4 E -2 \pi i -4 \log (2)) \left(k_p^4-k_s^4\right)+4 k_p^4 \log k_p -4 k_s^4 \log k_s }{64  \pi    \omega^2} +{\cal O}\left(z^2\right)\\
\end{eqnarray*}

\end{lemma}
 

\begin{remark} \label{remark:3.4}The functions $V_0(\tau,t)$, $K_0(\tau,t)$, $K_0^\top (t,\tau)$ y $W_0(\tau,t)$ are the kernel of the single layer, double layer, adjoint double layer and hypersingular operator for elasticity operator. It can be easily seen, from Lemmas \ref{lemma:Ap:01}-\ref{lemma:Ap:03}, that
\begin{eqnarray*}
 V(\tau,t)&=&  -\frac{\lambda +3 \mu }{4 \mu  (\lambda +2 \mu )}\frac{1}{2\pi} \log\left(4e^{-1}\sin^2\frac{\tau-t}2\right) I_2+ A(\tau,t) \sin^2\frac{\tau-t}{2}\log \sin^2\frac{\tau-t}2+B(\tau,t)\\
 K(\tau,t)&=&  \frac{\mu}{\lambda+2\mu} \left(\frac{\partial}{\partial t}
        \frac{1}{4\pi }\log\sin^2\frac{\tau-t}2  \right) \: \begin{bmatrix}
                                         &-1\\
                                         1&
                                        \end{bmatrix}
+ C(\tau,t)  \sin(\tau-t)\log \sin^2\frac{\tau-t}2+D(\tau,t)\\
W(\tau,t)&=&  -\frac{2\mu(\lambda+\mu)}{\lambda+2\mu} \frac{\partial^2}{\partial \tau \partial t}\left(-\frac{1}{2{\pi}}\log \sin^2\frac{\tau-t}2\right)\:I_2+
   E(\tau,t)  \log \sin^2\frac{\tau-t}2+F(\tau,t)\\
  &=& {  \frac{ \mu(\lambda+\mu)}{\lambda+2\mu}  \csc ^2\left(\frac{\tau-t}{2}\right)\:I_2+
   E(\tau,t)  \log \sin^2\frac{\tau-t}2+F(\tau,t)}
   . 
\end{eqnarray*}
with $A,B,C,D,E$ and $F$ smooth biperiodic functions. 

Indeed, for the single layer operator the result follows from the fact that
\[
 \frac{r^2}{\sin^2(\tau-t)/2}
\]
is a smooth non-vanishing function. For $K(\tau,t)$ we just have to start from Lemma \ref{lemma:Ap:02} and notice that
\[
 U_1(\tau,t), U_2(\tau,t) ={\cal O}(\tau-t),\quad \text{as }\tau-t\to 0
\]
from where one derives easily that, for $\psi$ a smooth cut-off function with support, say,  in $[-\pi/3,\pi/3]$ and $\psi(\tau)=1$ for $s\in[-\pi/4,\pi/4]$, the functions  
\[
 \frac{1}{\sin (\tau-t)}\psi(\tau-t) U_1(\tau,t),\quad 
 \frac{1}{\sin (\tau-t)}\psi(\tau-t) U_2(\tau,t)
\]
are  smooth and $2\pi-$periodic.   For $W$ the result follows from similar ideas. 

{Furthermore, it is a well established result, see for instance the excellent textbook \cite{Saranen}, that the integral operators 
\begin{eqnarray*}
 V_1\varphi &:=& \int_0^{2\pi} A(\cdot,t) \sin(\cdot-t)\log \sin^2\frac{\cdot-t}{2}\,{\rm d}t \\
 V_2\varphi &:=& \int_0^{2\pi} B(\cdot,t) \sin^2\frac{\cdot-t}2 \log \sin^2\frac{\cdot-t}{2}\,{\rm d}t, 
\end{eqnarray*}
with $A,B$ above being 
smooth $2\pi-$periodic functions in both variables, 
can be extended to define continuous pseudodifferential operators $ V_1:H^r\to H^{r+2}$ and $ V_2:H^r\to H^{r+3}$ where $H^r$ is the $2\pi-$periodic Sobolev space of order $r$. In other words, with 
\[ 
\bm{\Lambda} =\begin{bmatrix}
               \Lambda& \\
               & \Lambda  
              \end{bmatrix},\quad \bm{\Lambda}^{-1} =\begin{bmatrix}
               \Lambda^{-1}& \\
               & \Lambda^{-1}  
              \end{bmatrix},\quad 
\bm{H} = 
              \begin{bmatrix}
               & -H\\
               H &
              \end{bmatrix}=\bm{H}^\top
\]
where 
\begin{eqnarray*}
 \Lambda \varphi  &=&  -\frac{1}{2\pi}\int_{0}^{2\pi}\log\Big(4 e^{-1} \sin^2\frac{\cdot-t}2\Big)\varphi(t)\,{\rm d}t\\ 
\mathrm{H} g    &=&   - i\Lambda g' +\widehat{g}(0) = \mathrm{p.v.}\, \frac{1}{2\pi i}\int_{0}^{2\pi} \cot\frac{t\,-\,\cdot}2 g(t)\,{\rm d}t +\frac{1}{2\pi}\int_0^{2\pi}  g(t)\,{\rm d}t,  
\end{eqnarray*}
and noticing that 
\[
  \Lambda^{-1}g =  \operatorname{f.p.}  \frac{1}{4\pi}    \int_{0}^{2\pi}   \csc^2\left(\frac{\,\cdot\,-t}{2}\right)g(t)\,{\rm d}t +\int_0^{2\pi} g(t)\,{\rm d}t
\]
we can conclude that for any $r\in\mathbb{R}$ 
\[
 \begin{aligned}
 \Kii-\alpha\bm{H}&:H^r\times H^r\to H^{r+2}\times H^{r+2}&
 \Vii-\beta\bm{\Lambda}&:H^r\times H^r\to H^{r+3}\times H^{r+3}\\ 
 \Wii-\delta\bm{\Lambda}^{-1}&:H^r\times H^r\to H^{r+1}\times H^{r+1}&
 \Kii^\top-\alpha\bm{H}&:H^r\times H^r\to H^{r+2}\times H^{r+2}
 \end{aligned}
 \]
 where  
\[
  \alpha = \frac{i \mu }{2 (\lambda +2 \mu )},\quad \beta = \frac{\lambda +3 \mu }{4 \mu  (\lambda +2 \mu )}, \quad \delta =  -\frac{\mu  (\lambda +\mu )}{\lambda +2 \mu }. 
\]   }

\end{remark}

\end{document}